\documentclass[preprint,3p,times]{elsarticle}
\usepackage[toc,page,title,titletoc,header]{appendix}
\usepackage{amsmath}
\usepackage{amssymb}
\usepackage{amsthm}
\usepackage{mathrsfs}
\usepackage{graphicx}
\usepackage{epsfig}
\usepackage{tikz,pgfplots,tikz-3dplot}
\usepackage{epstopdf}
\usetikzlibrary{fit}
\biboptions{sort&compress}

\newtheorem{thm}{Theorem}[section]
\newtheorem{lemma}[thm]{Lemma}
\newtheorem{rem}[thm]{Remark}
\newtheorem{myDef}{Definition}[section]
\newtheorem{example}{Example}[section]

\newcommand{\tx}{\textbf x}
\newcommand{\mR}{\mathbb{R}}

\newcommand{\nn}{\nonumber}

\def\epsilon{\varepsilon} 

\begin{document}
\begin{frontmatter}
\title{
On the rotating nonlinear Klein--Gordon equation with multiscale effects: structure-preserving methods and applications to vortex dynamics
}

\author[a]{Meng Li}
\address[a]{School of Mathematics and Statistics, Zhengzhou University,
Zhengzhou 450001, China}
\ead{limeng@zzu.edu.cn}
\author[a]{Chunyan Niu*}
\ead{Corresponding author:chunyanniu@zzu.edu.cn}
\author[a]{Huifei Wang}
\author[b]{Junjun Wang}
\address[b]{School of Mathematics and Statistics, Pingdingshan University, Pingdingshan 467000, China}
\fntext[funding]{The research was supported by the National Natural Science Foundation of China (Nos. 11801527, U23A2065), the Natural Science Foundation (NSF) of Henan Province (Nos.252300420343, 242300421211).}

\begin{abstract}
  We study numerical methods for the rotating nonlinear Klein--Gordon (RKG) equation, a fundamental model in relativistic quantum physics, which exhibits highly oscillatory multiscale behavior due to the presence of a small parameter~$\epsilon$. 
  The RKG equation models rotating galaxies under the Minkowski metric and also provides an effective description of phenomena such as cosmic superfluids. 
  This work focuses on the development and rigorous analysis of structure-preserving Galerkin finite element methods (FEMs) for the RKG equation. 
  A central challenge is that the rotational terms prevent traditional nonconforming FEMs from simultaneously conserving energy and charge. By employing a conservation-adjusting technique, we construct a consistent structure-preserving algorithm applicable to both conforming and nonconforming FEMs.
  Moreover, we provide a comprehensive convergence analysis, establishing unconditional optimal and high-order accuracy error estimates. 
  These theoretical results are further validated through extensive numerical experiments, which demonstrate the accuracy, efficiency, and robustness of the structure-preserving schemes. 
  Finally, simulations of vortex dynamics, ranging from the relativistic to the nonrelativistic regimes, are presented to illustrate vortex creation, relativistic effects on bound states, and interactions of vortex pairs.
\end{abstract} 


\end{frontmatter}

\setcounter{equation}{0}
\section{Introduction} \label{sec:intro}

The universe can be viewed as a type of superfluid, an idea originating from the seminal works of Witten \cite{witten1985superconducting} and Zurek \cite{zurek1985cosmological}, where such systems are modeled by complex-valued scalar fields, in analogy with the Ginzburg-Landau theory, with an order parameter characterizing the underlying phase coherence \cite{ginzburg1950toward}. Within this framework, the Klein-Gordon (KG) equation \cite{bao2014uniformly} describes the dynamics of spin-0 (scalar) particles, including the Higgs boson, pions ($\pi$ mesons), and axion-like dark matter candidates. By incorporating nonlinear potential terms, one obtains the nonlinear KG (NKG) equation \cite{xiong2014relativistic}, which can be viewed as a relativistic generalization of the Gross-Pitaevskii equation, a nonlinear Schrödinger equation that serves as a fundamental model for Bose-Einstein condensates (BECs) \cite{Bao2013Kinetic}. Further including rotation effects, such as Coriolis and centrifugal forces, leads to the rotating NKG (RKG) equation \cite{huang2016a, low2016fantasia, Mauser2020on, wang2025stable}. This equation provides a unified framework linking relativistic quantum field theory, nonlinear dynamics, and BEC physics \cite{Bao2013Kinetic,zhang2007dynamics,tang2017robust}, while introducing two major computational challenges: the asymmetric spatial discretization induced by rotation, and the numerical instabilities arising from the interplay of nonlinearities and high-order derivatives.

In this work, we consider the following RKG equation
\begin{align}\label{eq:model1}
	\frac{1}{c^2}\partial_{tt}\Psi(\textbf{x}, t)-\Delta \Psi(\textbf{x}, t)
	+\left(\frac{mc}{\hbar}\right)^2\Psi(\textbf{x}, t)
	+\left(V(\textbf{x})+m\lambda\,|\Psi(\textbf{x}, t)|^2\right)\Psi(\textbf{x}, t)
	-R_{Co}-R_{ce}=0,
\end{align}
where $\Psi=\Psi(\tx, t): \mR^d\times \mR^+\rightarrow \mathbb{C}$ $(d=2, 3)$ is the scaler field, $V(\tx): \mR^d\rightarrow \mR$ is an external trapping potential, $m$ represents the mass, $c$ denotes the speed of light, $\hbar$ is the Plank constant, $\lambda$ denotes the self-interaction constant, and 
the Coriolis term $R_{Co}$ and the centrifugal term $R_{ce}$ are denoted by 
\begin{align*}
	R_{Co}:=\text{i}\frac{2\Omega}{c^2}L_z\partial_t\Psi(\textbf{x}, t),
	\qquad
	R_{ce}:=\frac{\Omega^2}{c^2}L_z^2\Psi(\tx, t),\qquad (\tx, t)\in \mR^d\times \mR^+,
\end{align*}
with $\Omega$ representing the angular velocity and the angular momentum operator 
\begin{align*}
	L_z:=-\text{i}\hbar\left(x\partial_y-y\partial_x\right).
\end{align*}

The RKG equation \eqref{eq:model1} has been used to describe the creation and dynamics of quantized vortices in galaxies, in two- or three-dimensional (2D/3D) models, depicting this phenomenon as a ``cosmic rotating superfluid.'' The formation of vortices follows the relativistic Feynman relation, with quantum turbulence manifesting as tangled vortices, consistent with mechanisms of matter creation during the Big Bang, while the generation and propagation of Kelvin waves are also in good agreement with observations~\cite{guo2018towards,huang2016a,low2016fantasia,xiong2014relativistic,xiong2016relativistic}. Furthermore, generalizations of the RKG equation to curved spacetimes, such as the BTZ and Kerr metrics describing black holes, have also been investigated~\cite{bezerra2014klein,pourhassan2016the,rowan1977the,chakrabarti2009normal,valtancoli2016scalar}.

To nondimensionalize the RKG equation, we introduce the following scaling:
\begin{align*}
	\widetilde{\tx} = \frac{\tx}{x_s}, \quad 
	\widetilde{t} = \frac{t}{t_s}, \quad 
	t_s = \frac{m x_s^2}{\hbar}, \quad 
	\widetilde{\Psi}(\widetilde{\tx}, \widetilde{t}) = \Psi(\tx, t),
\end{align*}
where $x_s$ and $t_s$ represent the characteristic length and time scales, respectively.
We further denote 
\begin{align*}
	\epsilon:=\frac{\hbar}{mcx_s},\quad 
	\widetilde{\Omega}:=\Omega mx_s^2,\quad
	\widetilde{\lambda}:=\lambda mx_s^2,\quad
	\widetilde{V}(\widetilde{\tx})=x_s^2V(\tx),\quad
	\widetilde{L}_z:=-\text{i}\left(\widetilde{x}\partial_{\widetilde y}-\widetilde{y}\partial_{\widetilde{x}}\right). 
\end{align*}
Then, we obtain the following dimensionless RKG equation \cite{Mauser2020on}
\begin{align}\label{eq:model2}
	\epsilon^2\partial_{tt}\Psi-\Delta \Psi+
	\frac{1}{\epsilon^2}\Psi
	+\left(V+\lambda\,|\Psi|^2\right)\Psi
	-2\text{i}\Omega \epsilon^2L_z\partial_t\Psi-\Omega^2\epsilon^2L_z^2\Psi=0,\qquad (\tx, t)\in \mR^d\times \mR^+,
\end{align}
where we have removed all the~~$\widetilde{}$~~for simplicity. 
In this work, we fix $\hbar$ and consider the nonrelativistic limit $\epsilon\rightarrow 0$ (or $c\rightarrow \infty$). In addition, we use the following usual initial conditions
\begin{align}\label{eq:initial}
	\Psi(\tx, 0)=\psi_0(\tx),\qquad 
	\partial_t\Psi(\tx, 0)=\frac{1}{\epsilon^2}\psi_1(\tx),\qquad \tx\in\mathbb{R}^d, 
\end{align}
where $\psi_0$ and $\psi_1$ are given bounded complex-valued functions. 
The system \eqref{eq:model2}-\eqref{eq:initial} conserves the (Hamiltonian) energy and charge, i.e.,  
\begin{align}\label{eq:conser}
	E(t)\equiv E(0),\qquad Q(t)\equiv Q(0),\qquad t\geq 0, 
\end{align}
where the total energy and charge are defined by 
\begin{align}
	&E(t):=\int_{\mathbb{R}^d}\left[\epsilon^2|\partial_t\Psi|^2+
	|\nabla \Psi|^2+\frac{1}{\epsilon^2}|\Psi|^2
	+V|\Psi|^2+\frac{\lambda}{2}|\Psi|^4
	-\Omega^2\epsilon^2|L_z\Psi|^2\right]\text{d}\tx,\label{eq:ener} \\
	&Q(t):=\frac{\text{i}\epsilon^2}{2}\int_{\mathbb{R}^d}\left[\Psi\partial_t\overline{\Psi}-\overline{\Psi}\partial_t\Psi+2\Omega\Psi\left(x\partial_y-y\partial_x\right)\overline{\Psi}\right]\text{d}\tx. \label{eq:charge}
\end{align}

From \cite{bao2014uniformly, Mauser2020on}, in the nonrelativistic regime $0<\epsilon \ll 1$, by applying the multiscale expansion
\[
\Psi(\mathbf{x}, t) = e^{\text{i}t/\epsilon^2} z_+(\mathbf{x}, t) + e^{-\text{i}t/\epsilon^2} \overline{z_-}(\mathbf{x}, t) + r(\mathbf{x}, t),\qquad \tx\in\mathbb R^d,\quad t\geq 0, 
\]
with $z_{\pm}=z_{\pm}(\tx, t)$ and $r=r(\tx, t)$ being unknown to be determined,  
and taking $\epsilon \to 0$, the relativistic RKG equation \eqref{eq:model2} formally reduces to the coupled rotating nonlinear Schrödinger (RNLS) equations:
\begin{align}\label{eq:coupled-NLS}
\begin{cases}
2\text{i}\,\partial_t z_+ - \Delta z_+ + V z_+ + \lambda \left( |z_+|^2 + 2|z_-|^2 \right) z_+ + 2 \Omega L_z z_+ = 0,\\[1ex]
2\text{i}\,\partial_t z_- - \Delta z_- + V z_- + \lambda \left( |z_-|^2 + 2|z_+|^2 \right) z_- + 2 \Omega L_z z_- = 0,
\end{cases}
\end{align}
with initial conditions
\[
z_+(\tx,0) = \frac{1}{2}\bigg(\psi_0(\tx) - \text{i} \psi_1(\tx)\bigg), \qquad z_-(\tx,0) = \frac{1}{2}\bigg(\overline{\psi_0}(\tx) - \text{i}\,\overline{\psi_1}(\tx)\bigg),
\]
ensuring consistency between the original RKG data and the limiting system. The system \eqref{eq:coupled-NLS}  conserves the mass and energy in the senses that 
\begin{align*}
    M_{\pm}\left(t\right)\equiv M_{\pm}(0),\qquad E_{[z_+, z_-]}(t)\equiv E_{[z_+, z_-]}(0),\qquad t\geq 0, 
\end{align*}
with the mass $M_{\pm}\left(t\right) := \int_{\mathbb{R}^d}\left|z_{\pm}\right|^2 \, \text{d}\mathbf{x}$ and the energy 
\begin{align*}
E_{[z_+, z_-]}(t) := \frac{1}{2} \int_{\mathbb{R}^d} \bigg[ |\nabla z_+|^2 + |\nabla z_-|^2 + V\big( |z_+|^2 + |z_-|^2 \big)  + \frac{\lambda}{2} \Big( |z_+|^4 + |z_-|^4 + 2 |z_+|^2 |z_-|^2 \Big) \bigg] \, \text{d}\mathbf{x}- \Omega \int_{\mathbb{R}^2} \mathrm{Re} \left( \overline{z_+} L_z z_+ + \overline{z_-} L_z z_- \right) \, \text{d}\mathbf{x}.
\end{align*}
In Appendix A, we detail the computation procedure of the bound state solutions of the coupled RNLS equations by using the normalized gradient flow with Lagrange multipliers (GFLM)~\cite{liu2021normalized,cai2021efficient}, which are then employed to investigate how relativistic effects modify the vortex patterns of the nonrelativistic coupled RNLS equations.

In \cite{Mauser2020on}, the authors studied the RKG equation \eqref{eq:model2} across relativistic and nonrelativistic regimes, analyzed vortex states of the corresponding coupled RNLS equations, and proposed efficient numerical methods for simulating vortex dynamics. These methods, which discretize the RKG equation on a truncated finite domain using either polar or rotating Lagrangian coordinates, may obscure the intrinsic mathematical structure of the model, potentially resulting in drawbacks such as the loss of structure-preserving properties. 
It is well recognized that structure-preserving schemes generally outperform non-structure-preserving ones. This advantage stems from their ability to maintain certain invariant properties, which allows them to accurately capture intricate details of the underlying physical processes. From this perspective, discrete schemes that preserve the invariants of the original continuous model provide a natural benchmark for assessing the reliability of numerical simulations.
Recently, Wang et al.~\cite{wang2025stable} proposed stable and conservative finite-difference time-domain methods for the RKG equation in natural Cartesian coordinates. However, their work only presented the structure-preserving scheme based on the finite difference method and did not include a detailed convergence analysis.

For the model under consideration, finite element methods (FEMs) offer at least two notable advantages over other spatial discretization techniques: they are well-suited for solving the RKG equation with rough or discontinuous potentials, and their compatibility with adaptive meshes makes them particularly effective for resolving complex vortex dynamics. Additionally, FEMs are capable of handling models defined on more complex computational domains. 
In general, both conforming and nonconforming FEMs are extensively used for solving partial differential equations, with the latter relaxing the continuity constraint at element interfaces and thus allowing inconsistencies at shared nodes between adjacent elements.
 This added flexibility provides distinct advantages, especially in managing complex geometries and capturing localized phenomena, while often imposing lower regularity requirements on the underlying continuous solution. 
In this paper, for theoretical completeness, we aim to develop a unified framework for solving the RKG equation adopting both the conforming and nonconforming FEMs. The presence of rotation terms presents a significant challenge for constructing structure-preserving schemes with nonconforming FEMs. This problem is resolved by applying a conservation-adjusting strategy that incorporates a novel stabilization term into the fully discrete scheme. 
Another challenge arises from the implicit nature of structure-preserving FEMs, which complicates the convergence analysis, especially when a stabilization term is introduced in the nonconforming setting. For the structure-preserving conforming and nonconforming FEMs constructed in this work, we employ the cut-off technique \cite{henning2017crank} combined with the time-space error splitting approach \cite{wang2014new} to provide a detailed convergence analysis, including proofs of both optimal and high-order convergence. These estimates are obtained without requiring any time-space step coupling condition and cover optimal $L^2$ and $H^1$ error bounds, as well as high-order convergence rates in the $H^1$-norm.
We emphasize that, for the RKG equation, no FEMs have yet been developed for its numerical solution. Beyond the choice of spatial discretization, there has also been no convergence analysis for structure-preserving schemes. These gaps constitute the main innovations of the present work.
In addition, extensive numerical experiments validate the theoretical results, demonstrating that the structure-preserving methods are accurate, efficient, and robust, while simulations of vortex dynamics across relativistic and nonrelativistic regimes reveal vortex formation, relativistic effects on bound states, and vortex-pair interactions.

The remainder of the paper is organized as follows. Section~\ref{sec2} presents the time-discrete method and the full-discrete method, and also establishes their energy and charge conservation. Section~\ref{sec3} states the main results on optimal and high-order error estimates, with detailed proofs provided in Section~\ref{sec4}. Section~\ref{sec5} reports numerical experiments that validate the theoretical analysis and illustrate vortex dynamics in the RKG equation. Conclusions are drawn in Section~\ref{sec6}. The Appendices A and B provide the bound state solutions of the coupled RNLS equations computed via the normalized GFLM, together with the proof of Lemma~\ref{lem:timediscrete_truncate_existence}.

\textbf{Notations.} 
In pratical computation, we truncate the whole space problem onto a bounded convex polyhedron $U\subset \mathbb R^d (d=2, 3).$ 
We use the standard notation for Sobolev spaces, which are denoted as \( W^{s,p}(U) \). Within this context, \( |\cdot|_{W^{s,p}} \) signifies the seminorm, while \( \|\cdot\|_{W^{s,p}} \) signifies the norm. Specifically, when \( p = 2 \), the space is designated as \( H^s(U) = W^{s,2}(U) \). Furthermore, the space \( H_0^1(U) \) is defined as the collection of functions \( v \) that belong to \( H^1(U) \) and are zero on the boundary \( \partial\Omega \), formally expressed as \( H_0^1(U) := \{ v \in H^1(U) : v|_{\partial\Omega} = 0 \} \). For the case when \( s = 0 \), the space reduces to \( L^p(U) = W^{0,p}(U) \).
For a finite time interval \( J = [0, T] \) with a positive constant \( T \), we define the Bochner space for a strongly measurable function \( \phi \) as:
\[ L^p(J, U) := \left\{ \phi: J \rightarrow X \mid \|\phi\|_{L^p(J, U)} < \infty \right\}, \]
where  
\begin{equation*}
\|\phi\|_{L^p(J, U)}:= 
\left\{
\begin{aligned}
	&\left[\int_0^T\|\phi(t)\|_{L^p(U)}^p\text{d}t\right]^{\frac1p},~~~~~~~~~~~~~~&\text{if}~p<+\infty,\\
	&\inf\left\{C: \|\phi\|_{L^\infty(U)}\leq C~~\text{a.e.}~\text{on}~J\right\},~~~~~~~~~~~~~~&\text{if}~p=+\infty.
\end{aligned}
\right.
\end{equation*}

\section{The discrete methods}\label{sec2}
\setcounter{equation}{0}

For the purpose of clarity and brevity in our exposition, we shall confine our discussion to 2D scenarios.
Nonetheless, the theoretical results and analytical techniques for 3D cases are, in principle, analogous to those in two dimensions.

\subsection{The time-discrete method}
Consider the uniform partitions of the time interval $J=[0, T]$ denoted by $\{I_n; n\in \mathbb N, 0\leq n\leq N-1\}$, where 
$I_n:=[t_{n}, t_{n+1}]$ with $0=t_0<t_1<\ldots<t_N=T$ and $|I_n|=\tau$. In what follows, we consider the time-discrete method for the RKG equation \eqref{eq:model2}.

\begin{myDef}
	(Time-discrete method for the RKG equation) Let 
	\begin{align}\label{eq:timediscrete_initial}
		\Psi_\tau^0:=\psi_0,\qquad \Psi_\tau^1=\psi_0+\frac{\tau}{\epsilon^2}\psi_1+\frac{\tau^2}{2\epsilon^2}\left[\Delta \psi_0-
		\frac{1}{\epsilon^2}\psi_0
		-\left(V+\lambda\,|\psi_0|^2\right)\psi_0
		+2\textup{i}\Omega L_z\psi_1+\Omega^2\epsilon^2L_z^2\psi_0\right].
	\end{align}
	Then for $n\geq 1$, we define the following time-discrete method, which is to find $\Psi_\tau^{n+1}\in H_0^1(U)$, $1\leq n\leq N-1$ such that 
	\begin{align}\label{eqn:timediscrete}
\epsilon^2\delta_{t\overline t}^2\Psi_\tau^n-\Delta \widetilde{\Psi}_\tau^n+
\frac{1}{\epsilon^2}\widetilde{\Psi}_\tau^n
+V\widetilde{\Psi}_\tau^n+\lambda\,\frac{\left|\Psi^{n+1}_\tau\right|^2+\left|\Psi^{n-1}_\tau\right|^2}{2}\widetilde{\Psi}_\tau^n
-2\textup{i}\Omega \epsilon^2L_z\delta_{\hat{t}}\Psi_\tau^n-\Omega^2\epsilon^2L_z^2\widetilde{\Psi}_\tau^n=0,
	\end{align}
where 
\begin{align*}
	\delta_t\Psi^n_\tau:=\frac{\Psi_\tau^{n+1}-\Psi_\tau^{n}}{\tau},
	\quad
	\delta_{\overline t}\Psi^n_\tau=\frac{\Psi_\tau^n-\Psi_\tau^{n-1}}{\tau},
	\quad
	\widetilde{\Psi}_\tau^n = \frac{\Psi_\tau^{n+1}+\Psi_\tau^{n-1}}{2},
	\quad 
	\delta_{\hat{t}}\Psi_\tau^n= \frac{\Psi_\tau^{n+1}-\Psi_\tau^{n-1}}{2\tau}.
\end{align*}

\end{myDef}

The well-posedness and convergence of the system \eqref{eqn:timediscrete} will be shown in Section \ref{sec4}. Furthermore, we can prove that the time-discrete method  \eqref{eqn:timediscrete} keeps the energy and charge conservation. To this end, we initially present the following lemmas, which will be instrumental for our subsequent theoretical analysis. 
Since their proofs are straightforward, we omit the detailed derivations.

\begin{lemma}\label{lem:xx}
 For any complex-valued functions $u$ and $v$, the following identity holds:
 \begin{align*}
 	\left(L_z u, v\right) = \left(u, L_z v\right),\qquad \left(L_z^2 u, v\right) = \left(L_z u, L_z v\right).
 \end{align*}
\end{lemma}

\begin{lemma}\label{lem:relations}
	For any sequence $U^n$, the following identities hold 
	\begin{align*}
		&\textup{Im}\left(\delta_{t\overline{t}}^2 U^n, \widetilde{U}^n\right)=\frac{\textup{Im} \left(\delta_tU^n, U^n\right)-\textup{Im}\left(\delta_tU^{n-1}, U^{n-1}\right)}{\tau},\\
		& \textup{Re}\left(\delta_{t\overline{t}}^2 U^n, \delta_{\hat{t}}U^n\right)=\frac{\left\|\delta_tU^{n}\right\|^2_{L^2}-\left\|\delta_tU^{n-1}\right\|^2_{L^2}}{2\tau},\\
		&\textup{Re}\left(\widetilde{U}^n, \delta_{\hat{t}}U^n\right)
		=\frac{\|U^{n+1}\|_{L^2}^2-\|U^{n-1}\|_{L^2}^2}{4\tau},\\
		&\textup{Re}\left(\nabla \widetilde{U}^n, \nabla \delta_{\hat{t}}U^n\right)
		=\frac{\left\|\nabla U^{n+1}\right\|_{L^2}^2-\left\|\nabla U^{n-1}\right\|_{L^2}^2}{4\tau},\\
		&\textup{Re}\left(L_z \widetilde{U}^n, L_z \delta_{\hat{t}}U^n\right)
		=\frac{\left\|L_z U^{n+1}\right\|_{L^2}^2-\left\|L_z U^{n-1}\right\|_{L^2}^2}{4\tau}.
	\end{align*} 
\end{lemma}
 
\begin{thm}\label{thm-timediscrete-conservation}
	The time-discrete method \eqref{eqn:timediscrete} is conservative in the senses of the energy and charge:
	\begin{equation}\label{eqn-timediscrete-conservation}
	E^{n}\equiv E^1,\qquad Q^n\equiv Q^1, 
    \qquad 1\leq n\leq N,
	\end{equation}
	where 
	\begin{align}\label{thm-timediscrete-energy}
			E^n = &\epsilon^2\left\|\delta_t\Psi_\tau^{n-1}\right\|^2_{L^2}
		+\frac{\left\|\nabla \Psi_\tau^{n}\right\|_{L^2}^2+\left\|\nabla \Psi_\tau^{n-1}\right\|_{L^2}^2}{2}
		+\frac{\left\| \Psi_\tau^{n}\right\|_{L^2}^2+\left\| \Psi_\tau^{n-1}\right\|_{L^2}^2}{2 \epsilon^2}
		+\frac{\left(V\Psi_\tau^{n}, \Psi_\tau^{n}\right)+\left(V\Psi_\tau^{n-1}, \Psi_\tau^{n-1}\right)}{2}\nn\\
		&+\frac{\lambda}{4}\left(\|\Psi_\tau^{n}\|_{L^4}^4+\|\Psi_\tau^{n-1}\|_{L^4}^4\right)
		-\frac{\Omega^2\epsilon^2}{2}\left(\left\|L_z \Psi_\tau^{n}\right\|_{L^2}^2+\left\|L_z \Psi_\tau^{n-1}\right\|_{L^2}^2\right),
	\end{align}
	and 
		\begin{align}\label{thm-timediscrete-charge}
        Q^n = &\epsilon^2 \textup{Im}
		\left(\delta_t\Psi_\tau^{n-1}, \Psi_\tau^{n-1}\right)
		-\frac{\Omega\epsilon^2}{2}
		\left[\textup{Im}\left(\textup{i}L_z\Psi_\tau^{n}, \Psi_\tau^{n}\right)
		+\textup{Im}\left(\textup{i}L_z\Psi_\tau^{n-1}, \Psi_\tau^{n-1}\right)\right].
	\end{align}
\end{thm}
\begin{proof}
Taking the inner product of \eqref{eqn:timediscrete} with $\delta_{\hat{t}}\Psi_\tau^n$ and extracting the real part, we obtain
\begin{align}
	\label{thm-timediscrete-conservation-pf1}
	&\epsilon^2\text{Re}\left(\delta_{t\overline t}^2\Psi_\tau^n, \delta_{\hat{t}}\Psi_\tau^n\right)
	+\text{Re}\left(\nabla \widetilde{\Psi}_\tau^n, \nabla\delta_{\hat{t}}\Psi_\tau^n\right)
	+
	\frac{1}{\epsilon^2}\text{Re}\left(\widetilde{\Psi}_\tau^n, 
	\delta_{\hat{t}}\Psi_\tau^n\right)
	+\text{Re}\left(V\widetilde{\Psi}_\tau^n, \delta_{\hat{t}}\Psi_\tau^n\right)+\lambda\,\text{Re}\left(\frac{|\Psi^{n+1}_\tau|^2+|\Psi^{n-1}_\tau|^2}{2}\widetilde{\Psi}_\tau^n, \delta_{\hat{t}}\Psi_\tau^n\right)\nn\\
	&
	-2\Omega\epsilon^2\text{Re}\left(iL_z\delta_{\hat t}\Psi_\tau^n, \delta_{\hat t}\Psi_\tau^n\right) -\Omega^2\epsilon^2\text{Re}\left(L_z^2\widetilde{\Psi}_\tau^n, \delta_{\hat t}\Psi_\tau^n\right)=0.
\end{align}
Using Lemmas \ref{lem:xx}-\ref{lem:relations}, and thanks to 
\begin{align}\label{thm-timediscrete-conservation-pf2}
	\text{Re}\left(iL_z\delta_{\hat t}\Psi_\tau^n, \delta_{\hat t}\Psi_\tau^n\right)=\text{Re}\left(\left[x\partial_y-y\partial_x\right]\delta_{\hat t}\Psi_\tau^n, \delta_{\hat t}\Psi_\tau^n\right) = 0,
\end{align}
and 
\begin{align}\label{thm-timediscrete-conservation-pf3}
	&\text{Re}\left(L_z^2\widetilde{\Psi}_\tau^n, \delta_{\hat t}\Psi_\tau^n\right)
	=\text{Re}\left(L_z\widetilde{\Psi}_\tau^n, L_z\delta_{\hat t}\Psi_\tau^n\right)= \frac{\left\|L_z \Psi_\tau^{n+1}\right\|_{L^2}^2-\left\|L_z \Psi_\tau^{n-1}\right\|_{L^2}^2}{4\tau},
\end{align}
 we further obtain 
\begin{align}\label{thm-timediscrete-conservation-pf4}
&	\epsilon^2\frac{\left\|\delta_t\Psi_\tau^{n}\right\|^2_{L^2}-\left\|\delta_t\Psi_\tau^{n-1}\right\|^2_{L^2}}{2\tau}
	+\frac{\left\|\nabla \Psi_\tau^{n+1}\right\|_{L^2}^2-\left\|\nabla \Psi_\tau^{n-1}\right\|_{L^2}^2}{4\tau}
+\frac{\left\| \Psi_\tau^{n+1}\right\|_{L^2}^2-\left\| \Psi_\tau^{n-1}\right\|_{L^2}^2}{4\tau \epsilon^2}
+\frac{\left(V\Psi_\tau^{n+1}, \Psi_\tau^{n+1}\right)-\left(V\Psi_\tau^{n-1}, \Psi_\tau^{n-1}\right)}{4\tau}\nn\\
&+\lambda\frac{\|\Psi_\tau^{n+1}\|_{L^4}^4-\|\Psi_\tau^{n-1}\|_{L^4}^4}{8\tau}
-\Omega^2\epsilon^2\frac{\left\|L_z \Psi_\tau^{n+1}\right\|_{L^2}^2-\left\|L_z \Psi_\tau^{n-1}\right\|_{L^2}^2}{4\tau}=0,
\end{align}
which yields the energy conservation $E^n\equiv E^1$ for $1\leq n\leq N$.

Additionally, taking the inner product of \eqref{eqn:timediscrete} with $\widetilde{\Psi}_\tau^n$, and selecting the imaginary parts of both sides, we obtain 
\begin{align}
	\label{thm-timediscrete-conservation-pf5}
	&\epsilon^2\text{Im}\left(\delta_{t\overline t}^2\Psi_\tau^n, \widetilde{\Psi}_\tau^n\right)
	-2\Omega \epsilon^2\text{Im}\left(\text{i}L_z\delta_{\hat{t}}\Psi_\tau^n, \widetilde{\Psi}_\tau^n\right)
	-\Omega^2\epsilon^2\text{Im}\left(L_z^2\widetilde{\Psi}_\tau^n, \widetilde{\Psi}_\tau^n\right)=0.
\end{align}
By using integration by parts, we have 
\begin{align}\label{thm-timediscrete-conservation-pf6}
 \text{Im}\left(\text{i}L_z\delta_{\hat{t}}\Psi_\tau^n, \widetilde{\Psi}_\tau^n\right) 
 &=\frac{1}{4\tau}
 \text{Im}\left(\text{i}L_z\left(\Psi_\tau^{n+1}-\Psi_\tau^{n-1}\right), \Psi_\tau^{n+1}+\Psi_\tau^{n-1}\right)\nn\\
 &=\frac{1}{4\tau} \left[ 
 \text{Im}\left(\text{i}L_z\Psi_\tau^{n+1}, \Psi_\tau^{n+1}\right)-
 \text{Im}\left(\text{i}L_z\Psi_\tau^{n-1}, \Psi_\tau^{n-1}\right)+
 \text{Im}\left(\text{i}L_z\Psi_\tau^{n+1}, \Psi_\tau^{n-1}\right)-
 \text{Im}\left(\text{i}L_z\Psi_\tau^{n-1}, \Psi_\tau^{n+1}\right)  \right] \nn\\
 &=\frac{1}{4\tau}\left[ 
 \text{Im}\left(\text{i}L_z\Psi_\tau^{n+1}, \Psi_\tau^{n+1}\right)-
 \text{Im}\left(\text{i}L_z\Psi_\tau^{n-1}, \Psi_\tau^{n-1}\right)\right]+
 \frac{1}{4\tau}\text{Im}\left[\left(\text{i}L_z\Psi_\tau^{n+1}, \Psi_\tau^{n-1}\right)+\left(\Psi_\tau^{n-1}, \text{i}L_z\Psi_\tau^{n+1}\right)\right]\nn\\
 &=\frac{1}{4\tau} \left[ 
 \text{Im}\left(\text{i}L_z\Psi_\tau^{n+1}, \Psi_\tau^{n+1}\right)-
 \text{Im}\left(\text{i}L_z\Psi_\tau^{n-1}, \Psi_\tau^{n-1}\right) \right]. 
\end{align}
Furthermore, using Lemma \ref{lem:xx}, we derive 
\begin{align}\label{thm-timediscrete-conservation-pf7}
  \text{Im}\left(L_z^2\widetilde{\Psi}_\tau^n, \widetilde{\Psi}_\tau^n\right)
  =\text{Im}\left(L_z\widetilde{\Psi}_\tau^n, L_z\widetilde{\Psi}_\tau^n\right) =0. 
\end{align}
Taking \eqref{thm-timediscrete-conservation-pf6} and \eqref{thm-timediscrete-conservation-pf7} into \eqref{thm-timediscrete-conservation-pf5} and from Lemma \ref{lem:relations}, we obtain 
\begin{align}
\epsilon^2	\frac{\text{Im} \left(\delta_t\Psi_\tau^n, \Psi_\tau^n\right)-\text{Im}\left(\delta_t\Psi_\tau^{n-1}, \Psi_\tau^{n-1}\right)}{\tau}
-
\frac{\Omega \epsilon^2}{2\tau}\left[
\text{Im}\left(\text{i}L_z\Psi_\tau^{n+1}, \Psi_\tau^{n+1}\right)-
\text{Im}\left(\text{i}L_z\Psi_\tau^{n-1}, \Psi_\tau^{n-1}\right)
\right]=0, 
\end{align}
which yields the charge conservation $Q^n\equiv Q^1$ for $1\leq n\leq N$. 
\end{proof}

\subsection{The full-discrete method}
Suppose that $U$ is a rectangular region in the $(x,y)$ plane with edges aligned parallel to the coordinate axes. Let $\mathscr{T}_{h}$  represent a regular rectangular subdivision of $U$, and set  $K\in\mathscr{T}_{h}$ and $h=\max\limits_{K}diam(K)$. We consider the following two types of elements: 
\begin{itemize}
	\item Conforming FE space:
	\begin{align}
		V_{h}^C:=\left\{v_h\in H_0^1(U); v_h|_{K}\in Q_{11}(K), \forall~ K\in \mathscr{T}_{h}\right\},
	\end{align}
where $Q_{11}(K)=\operatorname{span}\{1, x, y, xy\}$. 
	\item Nonconforming FE space ($EQ_1^{rot}$):
 \begin{align}
 	V_{h}^{NC}:=\left\{v_h\in L^2(U); v_h|_{K}\in \operatorname{span}\{1, x, y, x^2, y^2\}, \int_F[v_h]\text{d}s=0, F\subset\partial K, \forall~ K\in \mathscr{T}_{h}\right\},
 \end{align}
where 
$[v_h]$ means the jump of $v_h$ across the edge $F$ if $F$ is an internal edge, and $v_h$ itself if $F$ is a boundary edge.  
\end{itemize}

\textbf{Weak formulation}: Given initial data $\psi_0,~\psi_1\in H_0^1(U):=H_0^1(U; \mathbb C)$, find a wave function $\Psi\in L^\infty(J, H_0^1(U))$, such that 
\begin{align}\label{eqn:model-vf}
\epsilon^2\left(\partial_{tt}\Psi, \omega\right)
+
\left(\nabla\Psi, \nabla\omega\right)
+\frac{1}{\epsilon^2}
\left(\Psi, \omega\right)
+\left(\left[V+\lambda|\Psi|^2\right]\Psi, \omega\right)
-2\textup{i}\Omega \epsilon^2\left(L_z\partial_t\Psi, \omega\right)
-\Omega^2\epsilon^2\left(L_z\Psi, L_z\omega\right)=0
,
\quad\forall~ \omega\in H_0^1(U).
\end{align}


Next,
we uniformly define \( V_h \) as the finite element spaces utilized within the numerical schemes, where \( V_h = V_h^C \) denotes the conforming case and \( V_h = V_h^{NC} \) denotes the nonconforming case, respectively. We define the inner product and corresponding norms piecewisely, as specified by
\begin{equation*}
(\phi, \psi)_h:=\sum_K\int_K	\phi\cdot\overline{\psi} \text{d}\mathbf x,\quad \|\phi\|_{h}=\sqrt{(\phi, \phi)_h},
\quad |\phi|_{1,h}:=\left(\sum_K|\phi|_{1, K}^2\right)^{1/2},
\quad \|\phi\|_{1,h}:=\|\phi\|_{h}+|\phi|_{1,h},
\end{equation*}
and 
\begin{align*}
\left\langle\phi, \psi\right\rangle:=\sum_K\int_{\partial K}(\phi\cdot\bar\psi)(\mathbf x\cdot\mathbf n^\perp)\text{d}s \quad \text{with}
\quad \mathbf n^\perp=(n_y, -n_x). 
\end{align*}
For the following theoretical analysis, we define the Ritz projection $R_h: H^1_0(U)\rightarrow V_h$, such that 
\begin{align}\label{eqn:ritz}
   \left(\nabla v-\nabla R_hv, \nabla \omega_h\right)_h=0,\qquad \forall~ \omega_h\in V_h.
\end{align}
There exists a generic $h$-independent constant $C_{R_h}$, such that 
\begin{align}
\label{eqn:ritz_error}
 	\|v-R_hv\|_{L^2}+h\|v-R_hv\|_{1,h}\leq C_{R_h}h^s\|v\|_{H^s},\qquad s=1, 2, \quad\text{for~all~}v\in H_0^1(U)\cap H^2(U).
\end{align}
Let the four vertices of an element $K$ be denoted by $a_i$, and its edges by $l_i = a_i a_{i+1}$ for $i = 1,\dots,4 \pmod{4}$. The interpolation operator $I_h$ is then defined by
\begin{itemize}
  \item For the case of $V_h = V_h^C$, define the associated interpolation operator $I_h$ as:
  \[
  I_h : u \in H^2(\Omega) \rightarrow I_h u \in V_h, \quad I_h u|_K = I_K, \quad I_K u(a_i) = u(a_i), \quad i = 1 \sim 4 \pmod{4};
  \]

  \item For the case of $V_h = V_h^{NC}$, define the associated interpolation operator $I_h$ as:
  \[
  I_h : u \in H^1(\Omega) \rightarrow I_h u \in V_h, \quad I_h u|_K = I_K,
  \]
  \[
  \int_{l_i} (I_K u - u) \, \text{d}s = 0, \quad i = 1 \sim 4 \pmod{4} \quad \text{and} \quad \int_K (I_K u - u) \, \text{d}\mathbf{x} = 0.
  \]
\end{itemize}
Then, the interpolation operator $I_h$ satisfies \cite{brenner2008mathematical, shi2017unconditional}:
\begin{equation}\label{eqn:Ih_error}
\left\|v-I_h v\right\|_{L^2}\leq C_{I_h}h^2\|v\|_{H^2}\qquad \text{and}\qquad 
\left(\nabla(v-I_h v), \nabla  \omega_h\right)_h= 
\left\{
\begin{aligned}
O(h^2)\|v\|_{H^3}\|\omega_h\|_{1,h},\qquad&\text{for}~\omega_h\in V_h^C,\\
	0,~~~~~~~\qquad&\text{for}~\omega_h\in V_h^{NC}.
\end{aligned}
\right.
\end{equation}

With above preparation, based on the weak formulation $\eqref{eqn:model-vf}$, we introduce the following full-discrete method for the RKG equation \eqref{eq:model2}, which is established by using the second-order difference method in time and the FEMs in space. 

\begin{myDef}\label{Def:fully}
	(Full-discrete method for the RKG equation) Let 
	\begin{align}\label{eq:fulldiscrete_initial}
		\Psi_{h, \tau}^0:=I_h\psi_0,\qquad \Psi_{h, \tau}^1=I_h\left(\psi_0+\frac{\tau}{\epsilon^2}\psi_1+\frac{\tau^2}{2\epsilon^2}\left[\Delta \psi_0-
		\frac{1}{\epsilon^2}\psi_0
		-\left(V+\lambda\,|\psi_0|^2\right)\psi_0
		+2\textup{i}\Omega L_z\psi_1+\Omega^2\epsilon^2L_z^2\psi_0\right]\right).
	\end{align}
	Then for $n\geq 1$, we define the following full-discrete method, which is to find $\Psi_{h,\tau}^{n+1}\in V_h$, $1\leq n\leq N-1$ such that 
	\begin{align}\label{eqn:fullydiscrete}
			&\epsilon^2\left(\delta_{t\overline t}^2\Psi_{h,\tau}^n
			, \omega_h\right)
			+\left(\nabla \widetilde{\Psi}_{h,\tau}^n, \nabla \omega_h\right)_h
			+
			\frac{1}{\epsilon^2}\left(\widetilde{\Psi}_{h,\tau}^n, \omega_h\right)
			+
			\left(V\widetilde{\Psi}_{h,\tau}^n, \omega_h\right)
			+\lambda\,\left(\frac{\left|\Psi^{n+1}_{h,\tau}\right|^2				
				+\left|\Psi^{n-1}_{h,\tau}\right|^2}{2}\widetilde{\Psi}_{h,\tau}^n, \omega_h\right)\nn\\
			&-2\textup{i}\Omega \epsilon^2\left(L_z\delta_{\hat{t}}\Psi_{h,\tau}^n, \omega_h\right)_h-
			\Omega^2\epsilon^2\left(L_z\widetilde{\Psi}_{h,\tau}^n, L_z\omega_h\right)_h
	+\left\langle S^{n+1}, \omega_h\right\rangle=0,
	\qquad \forall~ \omega_h\in V_h,
	\end{align}
 where $\left\langle S^{n+1}, \omega_h\right\rangle$ is defined by 
\begin{equation}\label{eqn:fullydiscrete-stability}
\left\langle S^{n+1}, \omega_h\right\rangle:=
\begin{cases}
0,\qquad &\text{for conforming case,}\\
\Omega\epsilon^2\textup{Re}\left\langle\delta_{\hat{t}}\Psi_{h,\tau}^n, \omega_h\right\rangle
+\textup{i}\frac{\Omega\epsilon^2}{\tau}\textup{Im}\left\langle\Psi_{h,\tau}^{n+1}, \omega_h\right\rangle
,&\text{for nonconforming case.}\\
\end{cases}
\end{equation}
\end{myDef}

We will establish the well-posedness and convergence of the system \eqref{eqn:fullydiscrete} in Section \ref{sec4}. The following theorem demonstrates that the full-discrete method \eqref{eqn:fullydiscrete} preserves both total energy and charge conservation.

\begin{thm}\label{thm-fullydiscrete-conservation} (Energy and charge conservation)
	The full-discrete method \eqref{eqn:fullydiscrete} is conservative in the senses of total discrete energy and charge:
	\begin{equation}\label{eqn-fullydiscrete-conservation}
	E_h^{n}\equiv E_h^1,\qquad Q_h^n\equiv Q_h^1\qquad 1\leq n\leq N,
	\end{equation}
	where 
	\begin{align}\label{eqn-discrete-energy}
	E_h^n := &\epsilon^2\left\|\delta_t\Psi_{h,\tau}^{n-1}\right\|^2_{L^2}
	+\frac{\left|\Psi_{h,\tau}^{n}\right|_{1,h}^2+\left| \Psi_{h,\tau}^{n-1}\right|_{1,h}^2}{2}
	+\frac{\left\| \Psi_{h,\tau}^{n}\right\|_{L^2}^2+\left\| \Psi_{h,\tau}^{n-1}\right\|_{L^2}^2}{2 \epsilon^2}
	+\frac{\left(V\Psi_{h,\tau}^{n}, \Psi_{h,\tau}^{n}\right)+\left(V\Psi_{h,\tau}^{n-1}, \Psi_{h,\tau}^{n-1}\right)}{2}\nn\\
	&+\frac{\lambda}{4}\left(\left\|\Psi_{h,\tau}^{n}\right\|_{L^4}^4+\left\|\Psi_{h,\tau}^{n-1}\right\|_{L^4}^4\right)
	-\frac{\Omega^2\epsilon^2}{2}\left(\left\|L_z \Psi_{h,\tau}^{n}\right\|_{h}^2+\left\|L_z \Psi_{h,\tau}^{n-1}\right\|_{h}^2\right),
\end{align}
	and 
	\begin{align}
    Q_h^n := &\epsilon^2 \textup{Im}
		\left(\delta_t\Psi_{h,\tau}^{n-1}, \Psi_{h,\tau}^{n-1}\right)_h
		-\frac{\Omega\epsilon^2}{2}
		\left[\textup{Im}\left(\textup{i}L_z\Psi_{h,\tau}^{n}, \Psi_{h,\tau}^{n}\right)_h
		+\textup{Im}\left(\textup{i}L_z\Psi_{h,\tau}^{n-1}, \Psi_{h,\tau}^{n-1}\right)_h\right].
	\end{align}

\end{thm}

\begin{proof}
The results for the conforming scenario are straightforward, so we focus on providing proof details specifically for the nonconforming case.
By taking $\omega_h=\delta_{\hat{t}}\Psi_{h, \tau}^n$ in \eqref{eqn:fullydiscrete}, and selecting the real parts of both sides, we obtain 
	\begin{align}\label{eqn:fullydiscrete_ener_pf1}
	&\epsilon^2\text{Re}\left(\delta_{t\overline t}^2\Psi_{h,\tau}^n
	, \delta_{\hat{t}}\Psi_{h, \tau}^n\right)
	+\text{Re}\left(\nabla \widetilde{\Psi}_{h,\tau}^n, \nabla \delta_{\hat{t}}\Psi_{h, \tau}^n\right)_h
	+
	\frac{1}{\epsilon^2}\text{Re}\left(\widetilde{\Psi}_{h,\tau}^n, \delta_{\hat{t}}\Psi_{h, \tau}^n\right)
	+
	\text{Re}\left(V\widetilde{\Psi}_{h,\tau}^n, \delta_{\hat{t}}\Psi_{h, \tau}^n\right)
	\nn\\
	&+\lambda\,\text{Re}\left(\frac{|\Psi^{n+1}_{h,\tau}|^2				
		+|\Psi^{n-1}_{h,\tau}|^2}{2}\widetilde{\Psi}_{h,\tau}^n, \delta_{\hat{t}}\Psi_{h, \tau}^n\right)
		-2\Omega \epsilon^2 \text{Re}\left(\text{i}L_z\delta_{\hat{t}}\Psi_{h,\tau}^n, \delta_{\hat{t}}\Psi_{h, \tau}^n\right)_h\nn\\
	&-
	\Omega^2\epsilon^2\text{Re}\left(L_z\widetilde{\Psi}_{h,\tau}^n, L_z\delta_{\hat{t}}\Psi_{h, \tau}^n\right)_h+\text{Re}\left\langle S^{n+1}, \delta_{\hat{t}}\Psi_{h, \tau}^n\right\rangle=0.
\end{align}
Using the integration by parts, we have 
\begin{align}\label{eqn:fullydiscrete_ener_pf2}
	&-2\Omega \epsilon^2 \text{Re}\left(\text{i}L_z\delta_{\hat{t}}\Psi_{h,\tau}^n, \delta_{\hat{t}}\Psi_{h, \tau}^n\right)_h=-2\Omega \epsilon^2 \text{Re}\left(\left[x\partial_y-y\partial_x\right]\delta_{\hat{t}}\Psi_{h,\tau}^n, \delta_{\hat{t}}\Psi_{h, \tau}^n\right)_h\nn\\
	&=-\Omega \epsilon^2
	\bigg[\left(\left[x\partial_y-y\partial_x\right]\delta_{\hat{t}}\Psi_{h,\tau}^n, \delta_{\hat{t}}\Psi_{h, \tau}^n\right)_h+
	\left(\delta_{\hat{t}}\Psi_{h,\tau}^n, \left[x\partial_y-y\partial_x\right]\delta_{\hat{t}}\Psi_{h, \tau}^n\right)_h\bigg]\nn\\
	&=-\Omega\epsilon^2\left\langle\delta_{\hat{t}}\Psi_{h,\tau}^n, \delta_{\hat{t}}\Psi_{h,\tau}^n\right\rangle
	=-\text{Re}\left\langle S^{n+1}, \delta_{\hat{t}}\Psi_{h, \tau}^n\right\rangle. 
\end{align}
From Lemma \ref{lem:relations} and thanks to \eqref{eqn:fullydiscrete_ener_pf2}, we can obtain 
\begin{align}\label{eqn:fullydiscrete_ener_pf3}
	&	\epsilon^2\frac{\left\|\delta_t\Psi_{h,\tau}^{n}\right\|^2_{L^2}-\left\|\delta_t\Psi_{h,\tau}^{n-1}\right\|^2_{L^2}}{2\tau}
	+\frac{\left\| \Psi_{h,\tau}^{n+1}\right\|_{1,h}^2-\left\|\Psi_\tau^{n-1}\right\|_{1,h}^2}{4\tau}
	+\frac{\left\| \Psi_{h,\tau}^{n+1}\right\|_{L^2}^2-\left\| \Psi_{h,\tau}^{n-1}\right\|_{L^2}^2}{4\tau \epsilon^2}
	+\frac{\left(V\Psi_\tau^{n+1}, \Psi_\tau^{n+1}\right)-\left(V\Psi_\tau^{n-1}, \Psi_\tau^{n-1}\right)}{4\tau}\nn\\
	&+\lambda\frac{\|\Psi_{h,\tau}^{n+1}\|_{L^4}^4-\|\Psi_{h,\tau}^{n-1}\|_{L^4}^4}{8\tau}
	-\Omega^2\epsilon^2\frac{\left\|L_z \Psi_{h,\tau}^{n+1}\right\|_{h}^2-\left\|L_z \Psi_{h,\tau}^{n-1}\right\|_{h}^2}{4\tau}=0,
\end{align}
which yields the energy conservation $E_h^n\equiv E_h^1$ for $1\leq n\leq N$.

Taking $\omega_h = \widetilde{\Psi}_{h,\tau}^n$ in \eqref{eqn:fullydiscrete} and extracting the imaginary part, we obtain
\begin{align}\label{eqn-fullydiscrete-conservation-proof1}
 \epsilon^2\text{Im}\left(\delta_{t\overline t}^2\Psi_{h,\tau}^n
 , \widetilde{\Psi}_{h,\tau}^n\right)
 -2\Omega\epsilon^2 \text{Im}\left\{\left(\text{i}L_z\delta_{\hat{t}}\Psi_{h,\tau}^n, \widetilde{\Psi}_{h,\tau}^n\right)_h\right\}
 	+\text{Im}\left\langle S^{n+1}, \widetilde{\Psi}_{h,\tau}^n\right\rangle=0.
\end{align}
From Lemma \ref{lem:relations}, we have 
\begin{align}\label{eqn-fullydiscrete-conservation-proof2}
	 \epsilon^2\text{Im}\left(\delta_{t\overline t}^2\Psi_{h,\tau}^n
	, \widetilde{\Psi}_{h,\tau}^n\right)= \epsilon^2\frac{\text{Im} \left(\delta_t\Psi_{h,\tau}^n, \Psi_{h,\tau}^n\right)-\text{Im}\left(\delta_t\Psi_{h,\tau}^{n-1}, \Psi_{h,\tau}^{n-1}\right)}{\tau}.
\end{align}
Obviously, there holds 
\begin{align}\label{eqn-fullydiscrete-conservation-proof3}
	& -2\Omega\epsilon^2 \text{Im}\left\{\left(\text{i}L_z\delta_{\hat{t}}\Psi_{h,\tau}^n, \widetilde{\Psi}_{h,\tau}^n\right)_h\right\}\nn\\
	 &=-\frac{\Omega\epsilon^2}{2\tau}
	 \bigg[\text{Im}\left(\left[x\partial_y-y\partial_x\right]\Psi_{h,\tau}^{n+1}, \Psi_{h,\tau}^{n+1}\right)_h-\text{Im}\left(\left[x\partial_y-y\partial_x\right]\Psi_{h,\tau}^{n-1}, \Psi_{h,\tau}^{n-1}\right)_h\bigg]\nn\\
	 &~~~~-\frac{\Omega\epsilon^2}{2\tau}
	 \bigg[\text{Im}\left(\left[x\partial_y-y\partial_x\right]\Psi_{h,\tau}^{n+1}, \Psi_{h,\tau}^{n-1}\right)_h-\text{Im}\left(\left[x\partial_y-y\partial_x\right]\Psi_{h,\tau}^{n-1}, \Psi_{h,\tau}^{n+1}\right)_h\bigg] \nn \\
     &=-\frac{\Omega\epsilon^2}{2\tau}
	 \bigg[\text{Im}\left(\text{i}L_z\Psi_{h,\tau}^{n+1}, \Psi_{h,\tau}^{n+1}\right)_h-\text{Im}\left(\text{i}L_z\Psi_{h,\tau}^{n-1}, \Psi_{h,\tau}^{n-1}\right)_h\bigg]\nn\\
	 &~~~~-\frac{\Omega\epsilon^2}{2\tau}
	 \bigg[\text{Im}\left(\left[x\partial_y-y\partial_x\right]\Psi_{h,\tau}^{n+1}, \Psi_{h,\tau}^{n-1}\right)_h-\text{Im}\left(\left[x\partial_y-y\partial_x\right]\Psi_{h,\tau}^{n-1}, \Psi_{h,\tau}^{n+1}\right)_h\bigg].
\end{align}
By the definition of $S^{n+1}$ given in \eqref{eqn:fullydiscrete-stability} and using the integration by parts, we have 
\begin{align}\label{eqn-fullydiscrete-conservation-proof4}
&	\text{Im}\left\langle S^{n+1}, \widetilde{\Psi}_{h,\tau}^n\right\rangle
	=\frac{\Omega\epsilon^2}{\tau}\text{Im}\left\langle\Psi_{h, \tau}^{n+1}, \widetilde{\Psi}_{h,\tau}^n\right\rangle\nn\\
&~~~	=\frac{\Omega\epsilon^2}{2\tau}\text{Im}\left\langle\Psi_{h,\tau}^{n+1}, {\Psi}_{h,\tau}^{n+1}\right\rangle
+
\frac{\Omega\epsilon^2}{2\tau}\text{Im}\left\langle\Psi_{h,\tau}^{n+1}, {\Psi}_{h,\tau}^{n-1}\right\rangle
=\frac{\Omega\epsilon^2}{2\tau}\text{Im}\left\langle\Psi_{h,\tau}^{n+1}, {\Psi}_{h,\tau}^{n-1}\right\rangle\nn\\
&~~~=\frac{\Omega\epsilon^2}{2\tau}
\bigg[\text{Im}\left(\left[x\partial_y-y\partial_x\right]\Psi_{h,\tau}^{n+1}, \Psi_{h,\tau}^{n-1}\right)_h+\text{Im}\left(\left[x\partial_y-y\partial_x\right]\Psi_{h,\tau}^{n-1}, \Psi_{h,\tau}^{n+1}\right)_h\bigg].
\end{align}
Taking \eqref{eqn-fullydiscrete-conservation-proof2}, \eqref{eqn-fullydiscrete-conservation-proof3} and \eqref{eqn-fullydiscrete-conservation-proof4} into 
\eqref{eqn-fullydiscrete-conservation-proof1} gives that 
\begin{align*}
	\epsilon^2\frac{\text{Im} \left(\delta_t\Psi_{h,\tau}^n, \Psi_{h,\tau}^n\right)-\text{Im}\left(\delta_t\Psi_{h,\tau}^{n-1}, \Psi_{h,\tau}^{n-1}\right)}{\tau}-\frac{\Omega\epsilon^2}{2\tau}
	\bigg[\text{Im}\left(\text{i}L_z\Psi_{h,\tau}^{n+1}, \Psi_{h,\tau}^{n+1}\right)_h-\text{Im}\left(\text{i}L_z\Psi_{h,\tau}^{n-1}, \Psi_{h,\tau}^{n-1}\right)_h\bigg]=0,
\end{align*}
which implies that 
\begin{align*}
&	\epsilon^2\text{Im} \left(\delta_t\Psi_{h,\tau}^n, \Psi_{h,\tau}^n\right)
	-\frac{\Omega\epsilon^2}{2}
	\bigg[\text{Im}\left(\text{i}L_z\Psi_{h,\tau}^{n+1}, \Psi_{h,\tau}^{n+1}\right)_h+\text{Im}\left(\text{i}L_z\Psi_{h,\tau}^{n}, \Psi_{h,\tau}^{n}\right)_h\bigg]\nn\\
	&=\epsilon^2\text{Im} \left(\delta_t\Psi_{h,\tau}^{n-1}, \Psi_{h,\tau}^{n-1}\right)
	-\frac{\Omega\epsilon^2}{2}
	\bigg[\text{Im}\left(\text{i}L_z\Psi_{h,\tau}^{n}, \Psi_{h,\tau}^{n}\right)_h+\text{Im}\left(\text{i}L_z\Psi_{h,\tau}^{n-1}, \Psi_{h,\tau}^{n-1}\right)_h\bigg].
\end{align*}
Finally, we derive the charge conservation $Q_h^n=Q_h^1$ for $1\leq n\leq N$. 
\end{proof}

Due to the equivalence of norms, we can assume  
\begin{align}\label{eqn:equi1}  
	C_{U_1}\|L_z\omega_h\|_{L^2} \leq |\omega_h|_{1,h} \leq C_{U_2}\|L_z\omega_h\|_{L^2},\qquad \forall~ \omega_h\in V_h.  
\end{align}  
Additionally, by applying the Poincaré inequality, we have   
\begin{align}\label{eqn:equi2}  
	\|\omega_h\|_{L^2} \leq C_U |\omega_h|_{1,h},\qquad \forall~ \omega_h\in V_h.  
\end{align}  
For convenience, we also use $C_U$, $C_{U_1}$, and $C_{U_2}$ to denote the norm equivalence constants for functions in $H_0^1(U)$.
Then, using the energy conservation in Theorem \ref{thm-fullydiscrete-conservation}, we can get the boundedness of $\Psi_{h,\tau}^n$. 

\begin{thm}\label{thm:boundedness}
Assume that the parameter $\epsilon$ satisfies $\epsilon < C_{U_1}/\Omega$ and $\lambda>0$. 
Then, it holds that
\begin{align}\label{eqn:boundedness}
	\left\|\Psi_{h,\tau}^n\right\|_{1,h} \leq M,
\end{align}
where $M$ is a positive bounded constant (which may vary in different contexts).
\end{thm}
\begin{proof}
	Using the energy conservation in \eqref{eqn-fullydiscrete-conservation}, and by virtue of \eqref{eqn:equi1}, we can obtain 
	\begin{align}\label{eqn:boundedness_pf1}
		&\epsilon^2\left\|\delta_t\Psi_{h,\tau}^{n-1}\right\|^2_{L^2}
		+\frac{\left|\Psi_{h,\tau}^{n}\right|_{1,h}^2+\left| \Psi_{h,\tau}^{n-1}\right|_{1,h}^2}{2}
		+\frac{\left\| \Psi_{h,\tau}^{n}\right\|_{L^2}^2+\left\| \Psi_{h,\tau}^{n-1}\right\|_{L^2}^2}{2 \epsilon^2}
		+\frac{\lambda}{4}\left(\|\Psi_{h,\tau}^{n}\|_{L^4}^4+\|\Psi_{h,\tau}^{n-1}\|_{L^4}^4\right)\nn\\
		&=E_h^1
		+\frac{\Omega^2\epsilon^2}{2}\left(\left\|L_z \Psi_{h,\tau}^{n}\right\|_{h}^2+\left\|L_z \Psi_{h,\tau}^{n-1}\right\|_{h}^2\right)
		-\frac{\left(V\Psi_{h,\tau}^{n}, \Psi_{h,\tau}^{n}\right)+\left(V\Psi_{h,\tau}^{n-1}, \Psi_{h,\tau}^{n-1}\right)}{2}\nn\\
&		\leq |E_h^1|
		+\frac{\Omega^2\epsilon^2}{2C_{U_1}^2}\left(\left|\Psi_{h,\tau}^{n}\right|_{1,h}^2+\left| \Psi_{h,\tau}^{n-1}\right|_{1,h}^2\right)
		+\frac{C_V}{2}\left(\left\|\Psi_{h,\tau}^{n}\right\|_{L^2}^2+\left\|\Psi_{h,\tau}^{n-1}\right\|_{L^2}^2\right),
	\end{align}
	where we have assumed $\|V\|_{L^\infty}\leq C_V$. 
Since \(\epsilon\) is a relatively small number, we assume that 
\begin{align}\label{eqn:ep}
\frac{\Omega^2\epsilon^2}{2C_{U_1}^2}+\epsilon_0=\frac{1}{2},
\qquad \frac{C_V}{2}+\epsilon_1=\frac{1}{2\epsilon^2},
\end{align}
where $\epsilon_0$ and $\epsilon_1$ are two positive constants, 
which allows us to further obtain
	\begin{align*}
	\epsilon^2\left\|\delta_t\Psi_{h,\tau}^{n-1}\right\|^2_{L^2}
	+\epsilon_0\left(\left|\Psi_{h,\tau}^{n}\right|_{1,h}^2+\left| \Psi_{h,\tau}^{n-1}\right|_{1,h}^2\right)
	+\epsilon_1\left(\left\| \Psi_{h,\tau}^{n}\right\|_{L^2}^2+\left\| \Psi_{h,\tau}^{n-1}\right\|_{L^2}^2\right)
	+\frac{\lambda}{4}\left(\|\Psi_{h,\tau}^{n}\|_{L^4}^4+\|\Psi_{h,\tau}^{n-1}\|_{L^4}^4\right)\leq |E_h^1|.
\end{align*}
Therefore, we complete the proof. 
\end{proof}

\begin{rem}
From \cite{bainov1995nonexistence}, it is known that the NKG equation with focusing self-interaction (\(\lambda < 0\)) can lead to possible finite-time blow-up, whereas in the defocusing case (\(\lambda > 0\)), the existence of a global solution is assured. In this work, we focus on scenarios that are far from the critical blow-up time. Additionally, the boundedness of the numerical solution, as stated in Theorem \ref{thm:boundedness}, requires \(\lambda\) to be positive. 
\end{rem}

\section{Main results}\label{sec3}
\setcounter{equation}{0}

In this section, we present the primary convergence results, including optimal and high-order error estimates. The corresponding proof will be provided in the following section. We assume, in this work, that the solution to \eqref{eq:model2} exists and satisfies the following regularity condition:
\begin{align}\label{eqn:reg}
	\big\|\Psi_0\big\|_{H^2}+\big\|\Psi\big\|_{L^\infty((0, T); H^3)}
	+ \big\|\Psi_t\big\|_{L^2((0, T); H^3)}
	+ \big\|\Psi_{tt}\big\|_{L^2((0, T); H^4)}+ \big\|\Psi_{ttt}\big\|_{L^2((0, T); H^2)}\leq C_u. 
\end{align}

\begin{thm}\label{thm:main}
	Let $\Psi_{h,\tau}^{n+1}\in V_h$, $1\leq n\leq N-1$ be the solution of the full-discrete method \eqref{eqn:fullydiscrete}. Then, under the regularity assumption \eqref{eqn:reg}, there hold 
	\begin{itemize}
		\item [(a)] ($L^\infty$-norm boundedness)
		\begin{align}\label{eqn:L_infty}
			\sup_{0\leq m\leq N}\left\|\Psi_{h,\tau}^m\right\|_{L^\infty}\leq M, 
		\end{align}
		\item [(b)] (optimal $L^2$-norm error estimate)
		\begin{align}\label{eqn:convergenceL2}
			\sup_{0\leq m\leq N}\left\|\Psi^n-\Psi_{h,\tau}^m\right\|_{L^2}
			\leq C(h^2+\tau^2),
		\end{align}
		\item [(c)] (optimal $H^1$-norm error estimate)
		\begin{align}\label{eqn:convergenceH1}
			\sup_{0\leq m\leq N}\left\|\Psi^n-\Psi_{h,\tau}^m\right\|_{1,h}\leq C(h+\tau^2),
		\end{align}
		\item [(d)] (high-order $H^1$-norm error estimates)
		\begin{align}
			&\sup_{0\leq m\leq N}\|I_h\Psi^n-\Psi_{h,\tau}^m\|_{1,h}\leq C(h^2+\tau^2),\label{eqn:supercloseH1}\\
			&\sup_{0\leq m\leq N}\left\|\Psi^n-I_{2h}\Psi_{h,\tau}^m\right\|_{1,h}\leq C(h^2+\tau^2)\label{eqn:superconvergenceH1},
		\end{align}
	\end{itemize}
	where $C>0$ is a constant independent of $h$ and $\tau$. 
\end{thm}

The proof of Theorem \ref{thm:main} will be given in the following section. 

\section{Error analysis for the FEMs}\label{sec4}
\setcounter{equation}{0}
In this section, we provide an elaborate proof of Theorem \ref{thm:main}.

\subsection{Error analysis for the time-discrete method}
In this subsection, we analyze the convergence and boundedness of the solution to the time-discrete method \eqref{eqn:timediscrete}. 
However, since the method is implicit, we first investigate its truncated auxiliary system.
The following truncated function plays a pivotal role in convergence analysis. Under the regularity assumption \eqref{eqn:reg}, we define 
\begin{equation}
	\label{eqn:timebound}
	K_0:=\sup_{M\in \mathbb N}\left\{\max_{0\leq m\leq M}\|\Psi^m\|_{L^\infty}\right\}+1.
\end{equation}
\begin{myDef} \label{def1}
(Truncated function)
	Define the truncated function as 
	\begin{align}
		\label{eqn:cutoff}
		\mu_A(s)=s\cdot \chi(s\cdot K_0^{-2}),
	\end{align}
	where 
	\begin{equation}\label{eqn:cutoffx}
		\chi(x)=\left\{
		\begin{aligned}
			&\,0,  &|x|\in [2, +\infty), \\
			&\exp\left(1+\frac{1}{(|x|-1)^2-1}\right),  &|x|\in [1, 2), \\
			&1,  &|x|\in [0, 1). 
		\end{aligned}
		\right.
	\end{equation}
\end{myDef}
The truncated function $\mu_A(\cdot)$ exhibits the following characteristics.
\begin{lemma}\label{lem:cutoff_pro}
	The truncated function $\mu_A(s)$ belongs to $C^\infty(\mathbb{R})$, and there hold 
	\begin{subequations}
		\label{eqn:cutoff_pro}
		\begin{align}
			&\|\mu_A(w)\|_{L^\infty}\leq C_M, \quad  \left|\mu_A(w_1)-\mu_A(w_2)\right|\leq C_\mu|w_1-w_2|, \label{eqn:cutoff_pro_a}\\
			&\left|\mu_A(|w_1|^2)-\mu_A(|w_2|^2)\right|\leq C_\mu|w_1-w_2|,\label{eqn:cutoff_pro_b}
		\end{align}
	\end{subequations}
	where $C_M$ and $C_\mu$ are positive and bounded constants.   
\end{lemma}
\begin{proof}
	Based on the simple elementary properties of functions, we can directly obtain the conclusion of this lemma.  
\end{proof}

\begin{rem}
	Let $r=s\cdot K_0^{-2}$ in \eqref{eqn:cutoff}, then 
	\begin{align*}
		\mu_A=K_0^2r\cdot \chi(r)=:K_0^2\,\hat{\mu}_A(r).	
	\end{align*}
	To study the properties \eqref{eqn:cutoff_pro} of the 
	function $\mu_A(s)$, we only need to consider the function $\hat{\mu}_A(r)$.  For giving a vivid description, we plot the figures of $\hat{\mu}_A(r)$, $\textup{d}\hat{\mu}_A(r)/\textup{d}r$ and $\textup{d}^2\hat{\mu}_A(r)/\textup{d}r^2$ in Figure \ref{figure:1}. As depicted in Figure \ref{figure:1}, we can clearly observe the boundedness of the function $\hat{\mu}_A(r)$ and its first-order and second-order derivatives. 
	
	\begin{figure}[!htp]
		\centering \includegraphics[width=0.32\linewidth]{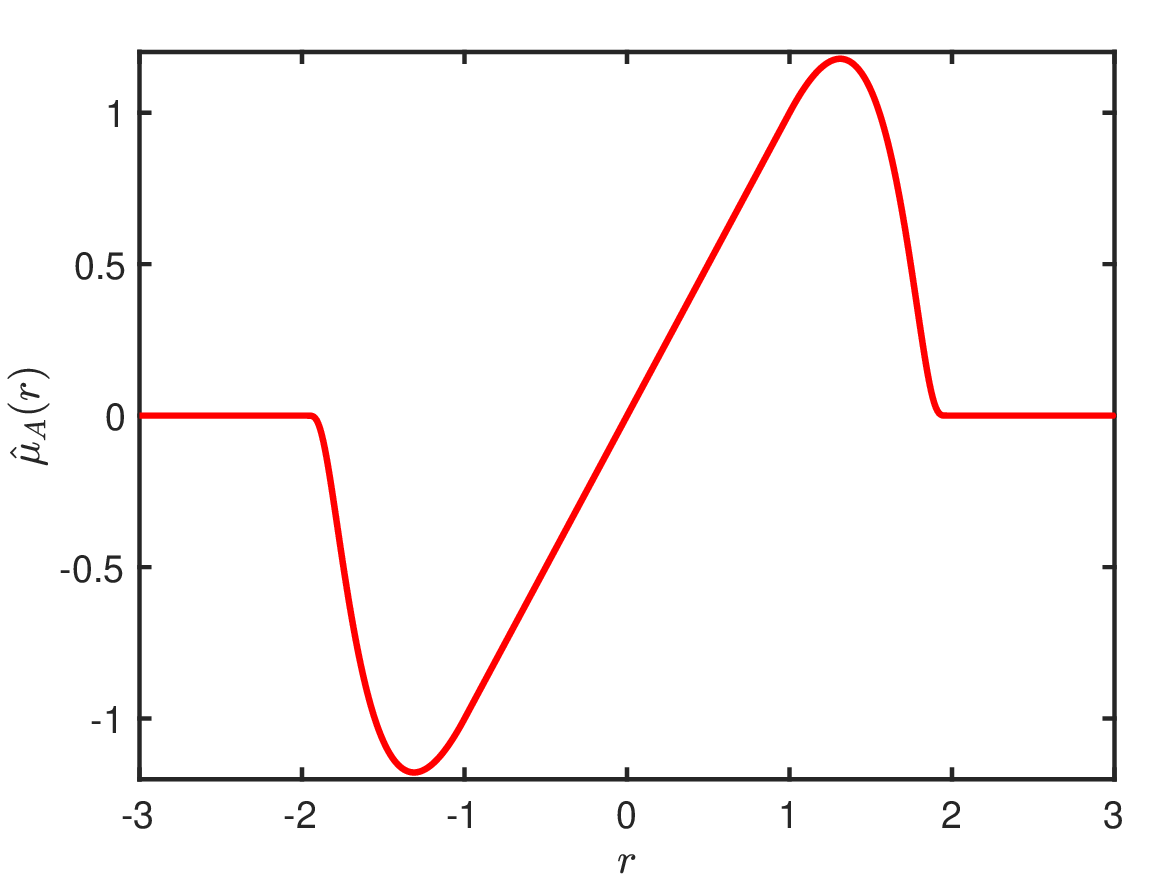}
        \includegraphics[width=0.32\linewidth]{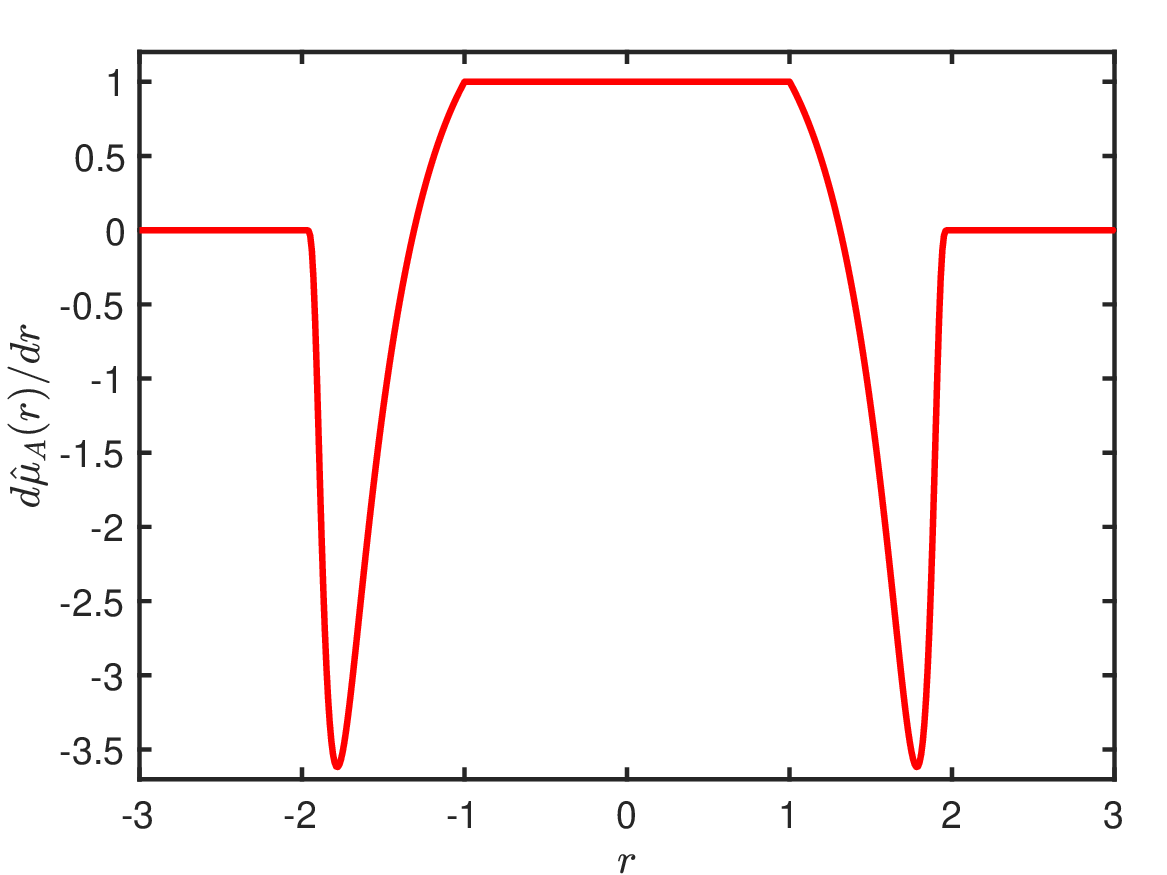}
	    \includegraphics[width=0.32\linewidth]{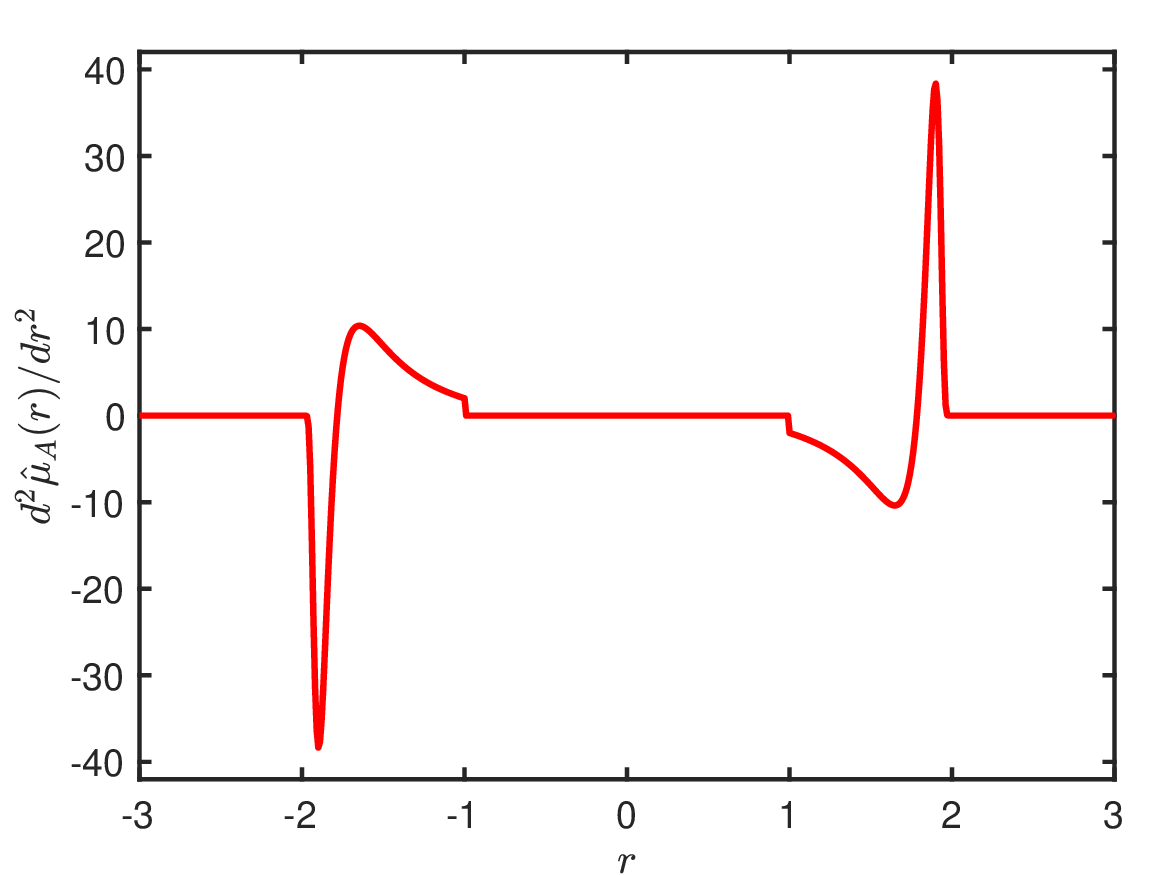}	
		\caption{Figures of the functions $\hat{\mu}_A(r)$, $\textup{d}\hat{\mu}_A(r)/\textup{d}r$ and $\textup{d}^2\hat{\mu}_A(r)/\textup{d}r^2$}
		\label{figure:1}
	\end{figure}

\end{rem}

By using the truncated function $\mu_A(\cdot)$, we introduce the following truncated time-discrete method.
\begin{myDef} (Truncated time-discrete method) Let $\Psi_\tau^{T, 0}:=\Psi^0$. Then for 
	$n\geq 1$, we define the truncated time-discrete method $\Psi^{T, n+1}_\tau\in H_0^1(U)$, $1\leq n\leq N-1$ as the solution to 
		\begin{align}\label{eqn:timediscrete_truncate}
		\epsilon^2\delta_{t\overline t}^2\Psi_\tau^{T, n}-\Delta \widetilde{\Psi}_\tau^{T, n}+
		\frac{1}{\epsilon^2}\widetilde{\Psi}_\tau^{T, n}
		+V\widetilde{\Psi}_\tau^{T, n}+\lambda\,\frac{\mu_A\left(\left|\Psi^{T, n+1}_\tau\right|^2\right)+\mu_A\left(\left|\Psi^{T, n-1}_\tau\right|^2\right)}{2}\widetilde{\Psi}_\tau^{T, n}
		-2\textup{i}\Omega \epsilon^2L_z\delta_{\hat{t}}\Psi_\tau^{T, n}-\Omega^2\epsilon^2L_z^2\widetilde{\Psi}_\tau^{T, n}=0.
	\end{align}
\end{myDef}

To ensure the convergence of the time-discrete method, we begin by deriving the error estimate for its truncated system.  
Furthermore, it is crucial to first establish the existence of the truncated time-discrete method \eqref{eqn:timediscrete_truncate} as a foundation for further analysis.
\begin{lemma}\label{lem:timediscrete_truncate_existence}
For the truncated time-discrete method \eqref{eqn:timediscrete_truncate}, 
there exists at least one solution $\Psi_\tau^{T, n+1}\in H_0^1(U)$ for $1\leq n\leq N-1$.
\end{lemma}
\begin{proof}
See the proof in Appendix B. 
\end{proof}

In the following, we establish the error estimates for the solution of the truncated time-discrete method.

\begin{lemma}\label{lem:timediscrete_truncate_err}
Suppose that \(\Psi\) is a solution of \eqref{eq:model2} satisfying \eqref{eqn:reg}, and denote by \(\Psi_{\tau}^{T,n+1}\), $1\leq n\leq N-1$ the solution of \eqref{eqn:timediscrete_truncate}. Then, the following estimates hold
 \begin{align}\label{eqn:timediscrete_truncate_err}
  \left\|\Psi^m-\Psi_{\tau}^{T,m}\right\|_{H^2} \leq C\tau^2, \qquad 0\leq m\leq N, 
 \end{align}
 where $C>0$ is a constant independent of $\tau$. 
\end{lemma}

\begin{proof}
 Denote $e_{\tau}^{T,n}=\Psi^n-\Psi_{\tau}^{T,n}$, then from~\eqref{eq:model2} and~\eqref{eqn:timediscrete_truncate}, we have
\begin{align}\label{eqn:timediscrete_truncate_err_pf1}
\epsilon^2\delta_{t\bar{t}}^2e_{\tau}^{T,n}-\Delta \widetilde{e}_{\tau}^{T,n}+
	\frac{1}{\epsilon^2}\widetilde{e}_{\tau}^{T,n}
	+V \widetilde{e}_{\tau}^{T,n}-2\text{i}\Omega \epsilon^2L_z\delta_{\hat{t}}e_{\tau}^{T,n}-\Omega^2\epsilon^2L_z^2\widetilde{e}_{\tau}^{T,n} +\mathcal{L}_\tau^{T,n}+\mathcal{N}_{\tau}^{T,n}=0,	
\end{align}
where
\begin{align}
\mathcal{L}_\tau^{T,n}:=\,&
 \epsilon^2\left(\partial_{tt}\Psi^n-\delta_{t\bar{t}}^2\Psi^{n}\right)-\left(\Delta\Psi^{n}-\Delta\widetilde{\Psi}^{n}\right)+\left(\frac{1}{\epsilon^2}+V\right)\left(\Psi^{n}-\widetilde{\Psi}^{n}\right)-2\text{i}\Omega \epsilon^2\left(L_z\partial_{t}\Psi^n-L_z\delta_{\hat{t}}\Psi^n\right) \nn\\
 &-\Omega^2\epsilon^2\left(L_z^2\Psi^{n}-L_z^2\widetilde{\Psi}^{n}\right)+\lambda\,\left[\left|\Psi^{n}\right|^2\Psi^{n}-\frac{\left|\Psi^{ n+1}\right|^2+\left|\Psi^{n-1}\right|^2}{2}\widetilde{\Psi}^{n}\right],\label{eqn:timediscrete_truncate_err_pf2}\\
\mathcal{N}_{\tau}^{T,n}:=\,&\lambda\,\left[\frac{\left|\Psi^{ n+1}\right|^2+\left|\Psi^{n-1}\right|^2}{2}\widetilde{\Psi}^{n}-\frac{\mu_A\left(\left|\Psi^{T, n+1}_\tau\right|^2\right)+\mu_A\left(\left|\Psi^{T, n-1}_\tau\right|^2\right)}{2}\widetilde{\Psi}_\tau^{T, n}\right]. \label{eqn:timediscrete_truncate_err_pf3}
\end{align}
Using the definition of the truncation function in \eqref{eqn:cutoff}, we have
\begin{align}\label{eqn:timediscrete_truncate_err_pf4}
\mathcal{N}_{\tau}^{T,n}=\,&\lambda\,\left[\frac{\left|\Psi^{ n+1}\right|^2+\left|\Psi^{n-1}\right|^2}{2}\widetilde{\Psi}^{n}-\frac{\mu_A\left(\left|\Psi^{T, n+1}_\tau\right|^2\right)+\mu_A\left(\left|\Psi^{T, n-1}_\tau\right|^2\right)}{2}\widetilde{\Psi}_\tau^{T, n}\right]\nn\\
=\,&\lambda\,\left[\frac{\mu_A\left(\left|\Psi^{ n+1}\right|^2\right)+\mu_A\left(\left|\Psi^{n-1}\right|^2\right)}{2}-\frac{\mu_A\left(\left|\Psi^{T, n+1}_\tau\right|^2\right)+\mu_A\left(\left|\Psi^{T, n-1}_\tau\right|^2\right)}{2}\right]\widetilde{\Psi}^{n} \nn\\
&+\,\lambda\,\left[\frac{\mu_A\left(\left|\Psi^{ T,n+1}\right|^2\right)+\mu_A\left(\left|\Psi^{T,n-1}\right|^2\right)}{2}\right]\frac{e_\tau^{T, n+}+e_\tau^{T, n-1}}{2}.
\end{align}
By using the properties of $\mu_A(\cdot)$ shown in Lemma \ref{lem:cutoff_pro}, we can derive
\begin{align}\label{eqn:timediscrete_truncate_err_pf5}
\mathcal{N}_{\tau}^{T,n}=\,&\lambda\,\left[\frac{\mu_A\left(\left|\Psi^{ n+1}\right|^2\right)-\mu_A\left(\left|\Psi^{T, n+1}_\tau\right|^2\right)}{2}+\frac{\mu_A\left(\left|\Psi^{n-1}\right|^2\right)-\mu_A\left(\left|\Psi^{T, n-1}_\tau\right|^2\right)}{2}\right]\widetilde{\Psi}^{n} \nn\\
&+\,\lambda\,\left[\frac{\mu_A\left(\left|\Psi^{ T,n+1}\right|^2\right)+\mu_A\left(\left|\Psi^{T,n-1}\right|^2\right)}{2}\right]\frac{e_\tau^{T, n+1}+e_\tau^{T, n-1}}{2} \nn\\
\leq \, & \lambda K_0\left[\frac{C_{\mu}\left|\Psi^{ n+1}-\Psi_\tau^{T, n+1}\right|}{2}+\frac{C_{\mu}\left|\Psi^{ n-1}-\Psi_\tau^{T,n-1}\right|}{2}\right]+\frac{\lambda C_{M}}{2}\left(\left|e_{\tau}^{T, n+1}\right|+\left|e_{\tau}^{T, n-1}\right|\right) \nn\\
\leq \, & \frac{\lambda C_{\mu}K_0}{2}\left(\left|e_{\tau}^{T, n+1}\right|+\left|e_{\tau}^{T, n-1}\right|\right)+\frac{\lambda C_{M}}{2}\left(\left|e_{\tau}^{T, n+1}\right|+\left|e_{\tau}^{T, n-1}\right|\right),
\end{align}
which further implies that 
\begin{align}\label{eqn:timediscrete_truncate_err_pf6}
\left|\mathcal{N}_{\tau}^{T,n}\right|_{H^s} \leq \, & C_{\mathcal{N},\tau}\left(\left|e_{\tau}^{T, n+1}\right|_{H^s}+\left|e_{\tau}^{T, n-1}\right|_{H^s}\right), \quad s=0, 1,
\end{align}
where $C_{\mathcal{N},\tau}$ is a positive constant independent of $\tau$.

Taking the inner product of \eqref{eqn:timediscrete_truncate_err_pf1} with $\delta_{\hat{t}}e_{\tau}^{T,n}$, and selecting the real part, we have
\begin{align}\label{eqn:timediscrete_truncate_err_pf7}
&\epsilon^2 \text{Re}\left(\delta_{t\bar{t}}^2 e_{\tau}^{T,n},\delta_{\hat{t}}e_{\tau}^{T,n}\right)+\text{Re}\left(\nabla\widetilde{e}_{\tau}^{T,n},\nabla \delta_{\hat{t}}e_{\tau}^{T,n}\right)+
	\frac{1}{\epsilon^2}\text{Re}\left(\widetilde{e}_{\tau}^{T,n},\delta_{\hat{t}}e_{\tau}^{T,n}\right)+V \text{Re}\left(\widetilde{e}_{\tau}^{T,n},\delta_{\hat{t}}e_{\tau}^{T,n}\right) \nn\\
    &-2\Omega \epsilon^2\text{Re}\left(\text{i}L_z\delta_{\hat{t}}e_{\tau}^{T,n},\delta_{\hat{t}}e_{\tau}^{T,n}\right)
    -\Omega^2\epsilon^2\text{Re}\left(L_z^2\widetilde{e}_{\tau}^{T,n} ,\delta_{\hat{t}}e_{\tau}^{T,n}\right)+\text{Re}\left(\mathcal{L}_\tau^{T,n},\delta_{\hat{t}}e_{\tau}^{T,n}\right)+\text{Re}\left(\mathcal{N}_{\tau}^{T,n},\delta_{\hat{t}}e_{\tau}^{T,n}\right)=0.
\end{align}
Then, using Lemma \ref{lem:relations}, and thanks to 
\begin{align}\label{eqn:timediscrete_truncate_err_pf8}
	-\text{Re}\left(\text{i}L_z\delta_{\hat t}e_\tau^{T,n}, \delta_{\hat t}e_\tau^{T,n}\right) = 0, \qquad 
    -\text{Re}\left(L_z^2\widetilde{e}_\tau^{T,n}, \delta_{\hat t}e_\tau^{T,n}\right)
	= -\frac{\left\|L_z e_\tau^{T,n+1}\right\|_{L^2}^2-\left\|L_z e_\tau^{T,n-1}\right\|_{L^2}^2}{4\tau},
\end{align}
there holds 
\begin{align}\label{eqn:timediscrete_truncate_err_pf9}
	&	\epsilon^2\frac{\left\|\delta_t e_{\tau}^{T,n}\right\|^2_{L^2}-\left\|\delta_t e_{\tau}^{T,n-1}\right\|^2_{L^2}}{2\tau}
	+\frac{\left\| \nabla e_{\tau}^{T,n+1}\right\|_{L^2}^2-\left\|\nabla e_\tau^{T,n-1}\right\|_{L^2}^2}{4\tau}
	+\left(\frac{1}{\epsilon^2}+|V|\right)\frac{\left\| e_{\tau}^{T,n+1}\right\|_{L^2}^2-\left\| e_{\tau}^{T,n-1}\right\|_{L^2}^2}{4\tau } \nn\\
	&
	-\Omega^2\epsilon^2\frac{\left\|L_z e_{\tau}^{T,n+1}\right\|_{L^2}^2-\left\|L_z e_{\tau}^{T,n-1}\right\|_{L^2}^2}{4\tau}+\text{Re}\left(\mathcal{L}_\tau^{T,n},\delta_{\hat{t}}e_{\tau}^{T,n}\right)+\text{Re}\left(\mathcal{N}_{\tau}^{T,n},\delta_{\hat{t}}e_{\tau}^{T,n}\right)=0. 
\end{align}
From~\eqref{eqn:timediscrete_truncate_err_pf6}, we have
\begin{align}\label{eqn:timediscrete_truncate_err_pf10}
    -\text{Re}\left(\mathcal{N}_{\tau}^{T,n},\delta_{\hat{t}}e_{\tau}^{T,n}\right) \leq 
    C\left(\left\|e_{\tau}^{T, n+1}\right\|_{0}^2+\left\|e_{\tau}^{T, n-1}\right\|_{0}^2\right)+C\left\|\delta_{\hat{t}}e_{\tau}^{T,n}\right\|_{0}^2.
\end{align}
Using the Taylor expansion, we have
\begin{align}\label{eqn:timediscrete_truncate_err_pf11}
-\text{Re}\left(\mathcal{L}_\tau^{T,n},\delta_{\hat{t}}e_{\tau}^{T,n}\right)\leq C\tau^4+C\left\|\delta_{\hat{t}}e_{\tau}^{T,n}\right\|_{0}^2.
\end{align}
For convenience, we denote 
\begin{align}\label{eqn:timediscrete_truncate_err_pf11}
    \mathscr E_\tau^{T, n}:=&2\epsilon^2\left\|\delta_t e_{\tau}^{T,n}\right\|^2_{L^2}
    +\left\{ \left\| \nabla e_{\tau}^{T,n+1}\right\|_{L^2}^2+\left\| \nabla e_{\tau}^{T,n}\right\|_{L^2}^2-\Omega^2\epsilon^2\left[\left\|L_z e_{\tau}^{T,n+1} \right\|_{L^2}^2-\left\|L_z e_{\tau}^{T,n} \right\|_{L^2}^2\right] \right\}\nn \\
    &+\left(\frac{1}{\epsilon^2}+V\right)\left[\left\| e_{\tau}^{T,n+1}\right\|_{L^2}^2+\left\| e_{\tau}^{T,n}\right\|_{L^2}^2\right] . 
\end{align}
The condition \eqref{eqn:ep} ensures that the second term in \eqref{eqn:timediscrete_truncate_err_pf11} remains positive.
Then, from \eqref{eqn:timediscrete_truncate_err_pf9}-\eqref{eqn:timediscrete_truncate_err_pf11}, we have 
\begin{align}\label{eqn:timediscrete_truncate_err_pf12}
\mathscr E_\tau^{T, n}\leq \mathscr E_\tau^{T, n-1}
+C\tau\left(\mathscr E_\tau^{T, n}+\mathscr E_\tau^{T, n-1}\right)+C\tau^5. 
\end{align}
By applying the discrete Gronwall inequality to \eqref{eqn:timediscrete_truncate_err_pf12}, we obtain  
\begin{align}\label{eqn:timediscrete_truncate_err_pf13}  
    \mathscr{E}_\tau^{T, n} \leq C \tau^4,  
\end{align}  
provided that \(\tau\) is sufficiently small, which further implies that 
\begin{align}\label{eqn:timediscrete_truncate_err_pf14}  
    \left\|\delta_t e_{\tau}^{T,n}\right\|_{L^2}
    +\left\| e_{\tau}^{T,n+1}\right\|_{H^1}\leq C\tau^2. 
\end{align}


We then establish the $H^2$-norm error estimates. 
Taking the inner product of \eqref{eqn:timediscrete_truncate_err_pf1} with $\delta_{\hat{t}}\Delta e_{\tau}^{T,n}$, and selecting the real part, we have
\begin{align}\label{eqn:timediscrete_truncate_err_pf15} 
	&	\epsilon^2\frac{\left\|\delta_t\nabla e_{\tau}^{T,n}\right\|^2_{L^2}-\left\|\delta_t \nabla e_{\tau}^{T,n-1}\right\|^2_{L^2}}{2\tau}
	+\frac{\left\| \Delta e_{\tau}^{T,n+1}\right\|_{L^2}^2-\left\|\Delta e_\tau^{T,n-1}\right\|_{L^2}^2}{4\tau}
	+\left(\frac{1}{\epsilon^2}+V\right)\frac{\left\| \nabla e_{\tau}^{T,n+1}\right\|_{L^2}^2-\left\|\nabla e_{\tau}^{T,n-1}\right\|_{L^2}^2}{4\tau } \nn\\
	&
	-\Omega^2\epsilon^2\frac{\left\| L_z \nabla e_{\tau}^{T,n+1}\right\|_{L^2}^2-\left\| L_z \nabla e_{\tau}^{T,n-1}\right\|_{L^2}^2}{4\tau}+\text{Re}\left(  \mathcal{L}_\tau^{T,n}, \delta_{\hat{t}} \Delta e_{\tau}^{T,n}\right)+\text{Re}\left( \mathcal{N}_{\tau}^{T,n}, \delta_{\hat{t}}\Delta e_{\tau}^{T,n}\right)=0. 
\end{align}
Similar to ~\eqref{eqn:timediscrete_truncate_err_pf10}, we have
\begin{align}\label{eqn:timediscrete_truncate_err_pf16}
    -\text{Re}\left(\mathcal{N}_{\tau}^{T,n},\delta_{\hat{t}}\Delta e_{\tau}^{T,n}\right) \leq 
    C\left(\left\|e_{\tau}^{T, n+1}\right\|_{0}^2+\left\|e_{\tau}^{T, n-1}\right\|_{0}^2+\left\|\nabla e_{\tau}^{T, n+1}\right\|_{0}^2+\left\|\nabla e_{\tau}^{T, n-1}\right\|_{0}^2\right)+C\left\|\nabla \delta_{\hat{t}}e_{\tau}^{T,n}\right\|_{0}^2.
\end{align}
Using the Taylor expansion, there holds 
\begin{align}\label{eqn:timediscrete_truncate_err_pf17}
-\text{Re}\left(\mathcal{L}_\tau^{T,n},\delta_{\hat{t}} \Delta e_{\tau}^{T,n}\right)\leq C\tau^4+C\left\|\delta_{\hat{t}}\nabla e_{\tau}^{T,n}\right\|_{0}^2.
\end{align}
For convenience, we denote 
\begin{align}\label{eqn:timediscrete_truncate_err_pf18}
    \mathscr H_\tau^{T, n}:= & 2\epsilon^2\left\|\delta_t \nabla e_{\tau}^{T,n}\right\|^2_{L^2}
    +\left[\left\| \Delta e_{\tau}^{T,n+1}\right\|_{L^2}^2+\left\| \Delta e_{\tau}^{T,n}\right\|_{L^2}^2-\Omega^2\epsilon^2\left(\left\|L_z \nabla e_{\tau}^{T,n+1}\right\|_{L^2}^2+\left\|L_z \nabla e_{\tau}^{T,n}\right\|_{L^2}^2\right)\right]\nn\\
    &+\left(\frac{1}{\epsilon^2}+V\right)\left(\left\| \nabla e_{\tau}^{T,n+1}\right\|_{L^2}^2+\left\| \nabla e_{\tau}^{T,n}\right\|_{L^2}^2\right). 
\end{align}
The condition \eqref{eqn:ep} can guarantee that the second term in \eqref{eqn:timediscrete_truncate_err_pf18} is positive.
From \eqref{eqn:timediscrete_truncate_err_pf15}-\eqref{eqn:timediscrete_truncate_err_pf18}, we have 
\begin{align}\label{eqn:timediscrete_truncate_err_pf19}
\mathscr H_\tau^{T, n}\leq \mathscr H_\tau^{T, n-1}
+C\tau\left(\mathscr H_\tau^{T, n}+\mathscr H_\tau^{T, n-1}\right)+C\tau^5. 
\end{align}
By applying the discrete Gronwall inequality to \eqref{eqn:timediscrete_truncate_err_pf19}, we obtain  
\begin{align}\label{eqn:timediscrete_truncate_err_pf20}  
    \mathscr{H}_\tau^{T, n} \leq C \tau^4,  
\end{align}  
provided that \(\tau\) is sufficiently small, which further implies that 
\begin{align}\label{eqn:timediscrete_truncate_err_pf21}  
    \left\|\delta_t \nabla e_{\tau}^{T,n}\right\|_{L^2}
    +\left\| \Delta e_{\tau}^{T,n+1}\right\|_{L^2}+\left\| \nabla e_{\tau}^{T,n+1}\right\|_{L^2}\leq C\tau^2. 
\end{align}
Finally, by combining \eqref{eqn:timediscrete_truncate_err_pf14} and \eqref{eqn:timediscrete_truncate_err_pf21}, we obtain  
\begin{align}\label{eqn:timediscrete_truncate_err_pf22}  
    \left\|\delta_t  e_{\tau}^{T,n}\right\|_{H^1}  
    +\left\|e_{\tau}^{T,n+1}\right\|_{H^2}\leq C\tau^2.  
\end{align}  
This concludes the proof.
\end{proof}

Next, we establish the $H^2$-norm error estimate for the solution of the time-discrete method \eqref{eqn:timediscrete}. 
\begin{lemma}\label{lem:timediscrete_err}
Suppose that \(\Psi\) is a solution of \eqref{eq:model2} satisfying \eqref{eqn:reg}. Then the time-discrete method \eqref{eqn:timediscrete} has a unique solution 
\(\Psi_{\tau}^{n+1}\) satisfying 
 \begin{align}\label{eqn:timediscrete_err}
   \left\|\Psi^m-\Psi_{\tau}^{m}\right\|_{H^2} \leq C\tau^2, \qquad 0\leq m\leq N, 
 \end{align}
 where $C>0$ is a constant independent of $\tau$. 
\end{lemma}
\begin{proof}
Firstly, we establish the uniqueness of the solution \(\Psi_{\tau}^{T, n+1}\) to the truncated time-discrete method \eqref{eqn:timediscrete_truncate}.  
Let \(\Psi_{\tau}^{I, T, n+1}\) and \(\Psi_{\tau}^{II, T, n+1}\) be two solutions to the truncated time-discrete method \eqref{eqn:timediscrete_truncate}. Define their difference as \(E_{\tau}^{T, n+1}=\Psi_{\tau}^{I, T, n+1}-\Psi_{\tau}^{II, T, n+1}\) and assume that \(E_{\tau}^{T, m}=0\) for \(0\leq m\leq n\).  
Following a similar approach as in the previous convergence analysis, we obtain  
\begin{align}\label{eqn:timediscrete_err_pf1}
    &2\epsilon^2\left\|\delta_t E_{\tau}^{T,n}\right\|^2_{L^2}
    +\left[\left\| \nabla E_{\tau}^{T,n+1}\right\|_{L^2}^2+\left\| \nabla E_{\tau}^{T,n}\right\|_{L^2}^2-\Omega^2\epsilon^2\left\|L_z E_{\tau}^{T,n+1}\right\|_{L^2}^2-\Omega^2\epsilon^2\left\|L_z E_{\tau}^{T,n}\right\|_{L^2}^2\right]+\left(\frac{1}{\epsilon^2}+V\right)\left(\left\| E_{\tau}^{T,n+1}\right\|_{L^2}^2+\left\| E_{\tau}^{T,n}\right\|_{L^2}^2\right) \notag\\
    &\quad
    \leq C\tau\Bigg\{2\epsilon^2\left\|\delta_t E_{\tau}^{T,n-1}\right\|^2_{L^2}
    +\left[\left\| \nabla E_{\tau}^{T,n}\right\|_{L^2}^2+\left\| \nabla E_{\tau}^{T,n-1}\right\|_{L^2}^2-\Omega^2\epsilon^2\left\|L_z E_{\tau}^{T,n}\right\|_{L^2}^2-\Omega^2\epsilon^2\left\|L_z E_{\tau}^{T,n-1}\right\|_{L^2}^2\right]\nn\\
    &~~~~~~~~+\left(\frac{1}{\epsilon^2}+V\right)\left(\left\| E_{\tau}^{T,n}\right\|_{L^2}^2+\left\| E_{\tau}^{T,n-1}\right\|_{L^2}^2\right)
    \Bigg\}.
\end{align}  
This result further implies that \(E_{\tau}^{T,n+1}=\Psi_{\tau}^{I, T, n+1}-\Psi_{\tau}^{II, T, n+1}=0\) for sufficiently small \(\tau\) and $\epsilon$, thereby proving the uniqueness of the solution.  

Furthermore, from the error estimates \eqref{eqn:timediscrete_truncate_err} and using \eqref{eqn:timebound}, for sufficiently small \(\tau\), we obtain  
\begin{align}\label{eqn:timediscrete_err_pf2}
    \left\|\Psi_\tau^{T, n}\right\|_{L^\infty} \leq \left\|\Psi^{n}\right\|_{L^\infty}+
    \left\|\Psi^{n}-\Psi_\tau^{T, n}\right\|_{L^\infty} \leq \left\|\Psi^{n}\right\|_{L^\infty}+
    C\left\|\Psi^{n}-\Psi_\tau^{T, n}\right\|_{H^2} \leq K_0. 
\end{align}  
By the definition of the truncation function \(\mu_A(\cdot)\), it follows that  
\begin{align}\label{eqn:timediscrete_err_pf3}
\mu_A\left(\left|\Psi^{T, n+1}_\tau\right|^2\right)=\left|\Psi^{T, n+1}_\tau\right|^2,\qquad 
\mu_A\left(\left|\Psi^{T, n-1}_\tau\right|^2\right)=\left|\Psi^{T, n-1}_\tau\right|^2.
\end{align}  
Thus, we can rewrite \eqref{eqn:timediscrete_truncate} as  
\begin{align}\label{eqn:timediscrete_err_pf4}
		\epsilon^2\delta_{t\overline t}^2\Psi_\tau^{T, n}-\Delta \widetilde{\Psi}_\tau^{T, n}+
		\frac{1}{\epsilon^2}\widetilde{\Psi}_\tau^{T, n}
		+V\widetilde{\Psi}_\tau^{T, n}+\lambda\,\frac{\left|\Psi^{T, n+1}_\tau\right|^2+\left|\Psi^{T, n-1}_\tau\right|^2}{2}\widetilde{\Psi}_\tau^{T, n}
		-2\text{i}\Omega \epsilon^2L_z\delta_{\hat{t}}\Psi_\tau^{T, n}-\Omega^2\epsilon^2L_z^2\widetilde{\Psi}_\tau^{T, n}=0,
	\end{align}  
which is identical to the time-discrete method \eqref{eqn:timediscrete}. Consequently, the unique existence of the truncated solution \(\Psi_\tau^{T, n+1}\) to the truncated time-discrete method \eqref{eqn:timediscrete_truncate} directly implies the unique existence of \(\Psi_\tau^{n+1}\) for the time-discrete method \eqref{eqn:timediscrete}. Overall, we can conclude that 
\begin{align}\label{eqn:timediscrete_err_pf5}
    \Psi_\tau^{T, m} =\Psi_\tau^{m},\qquad \text{for all}~~m\geq 0. 
\end{align}
Therefore, \eqref{eqn:timediscrete_err} can be directly obtained from \eqref{eqn:timediscrete_truncate_err}. We have completed the proof.  
\end{proof}

\subsection{Error analysis for the full-discrete method}

Analogous to the time-discrete method, we first consider the truncated full-discrete method.
\begin{myDef}\label{Def:fullydisrete_truncated}
	(Truncated full-discrete method) Let $\Psi_{h,\tau}^{T,j}$ be a suitable interpolation of $\Psi_\tau^j$ for $j=0, 1$. Then for $n\geq 1$, we define the following truncated full-discrete equation, which is to find $\Psi_{h,\tau}^{T,n+1}\in V_h$, $1\leq n\leq N-1$ such that 
	\begin{align}\label{eqn:fullydiscrete_truncated}
			&\epsilon^2\left(\delta_{t\overline t}^2\Psi_{h,\tau}^{T,n}
			, \omega_h\right)
			+\left(\nabla \widetilde{\Psi}_{h,\tau}^{T,n}, \nabla \omega_h\right)_h
			+
			\frac{1}{\epsilon^2}\left(\widetilde{\Psi}_{h,\tau}^{T,n}, \omega_h\right)
			+
			\left(V\widetilde{\Psi}_{h,\tau}^n, \omega_h\right)
			+\lambda\,\left(\frac{\mu_A\left(\left|\Psi^{T,n+1}_{h,\tau}\right|^2\right)				
				+\mu_A\left(\left|\Psi^{T,n-1}_{h,\tau}\right|^2\right)}{2}\widetilde{\Psi}_{h,\tau}^{T,n}, \omega_h\right)\nn\\
			&-2\textup{i}\Omega \epsilon^2\left(L_z\delta_{\hat{t}}\Psi_{h,\tau}^{T,n}, \omega_h\right)_h-
			\Omega^2\epsilon^2\left(L_z\widetilde{\Psi}_{h,\tau}^{T,n}, L_z\omega_h\right)_h
	+\left\langle S^{T,n+1}, \omega_h\right\rangle=0,
	\qquad \forall\, \omega_h\in V_h,
	\end{align}
 where $\left\langle S^{T,n+1}, \omega_h\right\rangle$ is defined by 
\begin{equation}\label{eqn:fullydiscrete-stability_truncated}
\left\langle S^{T,n+1}, \omega_h\right\rangle:=
\begin{cases}
0,\qquad &\text{for conforming case,}\\
\Omega\epsilon^2\textup{Re}\left\langle\delta_{\hat{t}}\Psi_{h,\tau}^{T, n}, \omega_h\right\rangle
+\textup{i}\frac{\Omega\epsilon^2}{\tau}\textup{Im}\left\langle\Psi_{h,\tau}^{T, n+1}, \omega_h\right\rangle
,& \text{for nonconforming case}.\\
\end{cases}
\end{equation}
\end{myDef}

\begin{lemma}\label{lem:fullydiscrete_err_truncated}
 Suppose that $\Psi_{\tau}^{n+1}$, $1\leq n\leq N-1$ is a solution of \eqref{eqn:timediscrete}, and denote by $\Psi_{h,\tau}^{T,n+1}$, $1\leq n\leq N-1$ the solution of \eqref{eqn:fullydiscrete_truncated}. Then, the following estimates hold
\begin{align}\label{eqn:fullydiscrete_err_truncated}
   \left\|\Psi_{\tau}^m-\Psi_{h,\tau}^{T,m}\right\|_{L^2}\leq C\left(h\tau^2+h^2\right),~~
   \left\|I_h \Psi_{\tau}^m- \Psi_{h,\tau}^{T,m}\right\|_{1, h}\leq C\left(h\tau^2+h^2\right),
   ~~
\left\|\Psi_{\tau}^m-\Psi_{h,\tau}^{T,m}\right\|_{1, h} \leq Ch,\quad 0\leq m\leq N, 
 \end{align}
 where $C>0$ is a constant independent of $\tau$ and $h$.
\end{lemma}

\begin{proof}
For simplicity, we restrict our attention to the more intricate nonconforming case. Using \eqref{eqn:timediscrete_err}, we easily obtain 
\begin{align}\label{eqn:timediscrete_bound}
    \left\|\Psi_\tau^n\right\|_{H^2}+\left\|\delta_{t\overline t}^2 \Psi_\tau^n\right\|_{H^2}\leq C,\quad \left\|\Psi_\tau^n\right\|_{H^3}\leq C,\quad \left\|\delta_{\hat{t}} \Psi^n\right\|_{H^3}\leq C.
\end{align}
 For convenience, we split the error
 \begin{equation}\label{eqn:error estimate for fullydiscrete_truncated_pf1}
     e_{h,\tau}^{T,n} = \Psi_{\tau}^n-\Psi_{h,\tau}^{T,n} = \left(\Psi_{\tau}^n-I_h \Psi_{\tau}^n\right)+\left(I_h \Psi_{\tau}^n- \Psi_{h,\tau}^{T,n}\right) = : \rho_{h,\tau}^{T,n}+\theta_{h,\tau}^{T,n}.
 \end{equation}
Substituting \eqref{eqn:fullydiscrete_truncated} from the variational formulation of the time-discrete method \eqref{eqn:timediscrete} and using integration by parts, we obtain 
\begin{align}\label{eqn:error estimate for fullydiscrete_truncated_pf2}
			&\epsilon^2\left(\delta_{t\overline t}^2 \theta_{h,\tau}^{T,n}
			, \omega_h\right)
			+\left(\nabla \widetilde{\theta}_{h,\tau}^{T,n}, \nabla \omega_h\right)_h
			+
			\frac{1}{\epsilon^2}\left(\widetilde{\theta}_{h,\tau}^{T,n}, \omega_h\right)+
			\left(V\widetilde{\theta}_{h,\tau}^{T,n}, \omega_h\right)-2\text{i}\Omega \epsilon^2\left(L_z\delta_{\hat{t}}\theta_{h,\tau}^{T,n}, \omega_h\right)_h -\Omega^2\epsilon^2\left(L_z\widetilde{\theta}_{h,\tau}^{T,n}, L_z\omega_h\right)_h
			 \nn \\
	&=-\epsilon^2\left(\delta_{t\overline t}^2 \rho_{h,\tau}^{T,n}
			, \omega_h\right)
			+\sum_{K}\int_{\partial K}\left(\nabla \widetilde{\Psi}_{\tau}^{n} \cdot \mathbf{n}\right) \omega_h \text{d}s
			-
			\frac{1}{\epsilon^2}\left(\widetilde{\rho}_{h,\tau}^{T,n}, \omega_h\right)-
			\left(\widetilde{\rho}_{h,\tau}^{T,n}, \omega_h\right)-\left(\mathcal{N}_{h,\tau}^{T,n},\omega _h \right)\nn \\
            &~~~~+2\text{i}\Omega \epsilon^2\left(L_z\delta_{\hat{t}}\rho_{h,\tau}^{T,n}, \omega_h\right)_h +\Omega^2\epsilon^2\left(L_z\widetilde{\rho}_{h,\tau}^{T,n}, L_z\omega_h\right)_h-\left\langle S^{T,n+1}, \omega_h\right\rangle,
	\qquad \forall\, \omega_h\in V_h^{NC},
	\end{align}
 where  $$\mathcal{N}_{h,\tau}^{T,n}:=\,\lambda\,\left[\frac{\left|\Psi_{\tau}^{ n+1}\right|^2+\left|\Psi_{\tau}^{n-1}\right|^2}{2}\widetilde{\Psi}_{\tau }^{n}-\frac{\mu_A\left(\left|\Psi^{T, n+1}_{h,\tau}\right|^2\right)+\mu_A\left(\left|\Psi^{T, n-1}_{h,\tau}\right|^2\right)}{2}\widetilde{\Psi}_{h,\tau}^{T, n}\right].$$
Taking $\omega_h=\delta_{\hat{t}}\theta_{h,\tau}^{T,n}$ in \eqref{eqn:error estimate for fullydiscrete_truncated_pf2}, and selecting its real part, we have
\begin{align}\label{eqn:error estimate for fullydiscrete_truncated_pf3}
			&\epsilon^2\text{Re}\left(\delta_{t\overline t}^2 \theta_{h,\tau}^{T,n}
			, \delta_{\hat{t}}\theta_{h,\tau}^{T,n}\right)
			+\text{Re}\left(\nabla \widetilde{\theta}_{h,\tau}^{T,n}, \nabla \delta_{\hat{t}}\theta_{h,\tau}^{T,n}\right)_h
			+
			\frac{1}{\epsilon^2}\text{Re}\left(\widetilde{\theta}_{h,\tau}^{T,n}, \delta_{\hat{t}}\theta_{h,\tau}^{T,n}\right) +
			\text{Re}\left(V\widetilde{\theta}_{h,\tau}^{T,n}, \delta_{\hat{t}}\theta_{h,\tau}^{T,n}\right)  \nn \\
            &-\text{Re}\left(2\text{i}\Omega \epsilon^2L_z\delta_{\hat{t}}\theta_{h,\tau}^{T,n}, \delta_{\hat{t}}\theta_{h,\tau}^{T,n}\right)_h-\Omega^2\epsilon^2\text{Re}\left(L_z\widetilde{\theta}_{h,\tau}^{T,n}, L_z\delta_{\hat{t}}\theta_{h,\tau}^{T,n}\right)_h \nn \\
			&
	=-\epsilon^2\text{Re}\left(\delta_{t\overline t}^2 \rho_{h,\tau}^{T,n}
			, \delta_{\hat{t}}\theta_{h,\tau}^{T,n}\right)
			+\text{Re}\left\{\sum_{K}\int_{\partial K}\left(\nabla \widetilde{\Psi}_{\tau}^{n} \cdot \mathbf{n}\right) \delta_{\hat{t}}\theta_{h,\tau}^{T,n} \text{d}s\right\} -
			\frac{1}{\epsilon^2}\text{Re}\left(\widetilde{\rho}_{h,\tau}^{T,n}, \delta_{\hat{t}}\theta_{h,\tau}^{T,n}\right)-
			\text{Re}\left(V\widetilde{\rho}_{h,\tau}^{T,n}, \delta_{\hat{t}}\theta_{h,\tau}^{T,n}\right)\nn \\
			&~~~~-\text{Re}\left(\mathcal{N}_{h,\tau}^{T,n},\delta_{\hat{t}}\theta_{h,\tau}^{T,n} \right)+\text{Re}\left(2\text{i}\Omega \epsilon^2L_z\delta_{\hat{t}}\rho_{h,\tau}^{T,n}, \delta_{\hat{t}}\theta_{h,\tau}^{T,n}\right)_h 
            +\Omega^2\epsilon^2\text{Re}\left(L_z\widetilde{\rho}_{h,\tau}^{T,n}, L_z\delta_{\hat{t}}\theta_{h,\tau}^{T,n}\right)_h 
            -\text{Re}\left\langle S^{T,n+1}, \delta_{\hat{t}}\theta_{h,\tau}^{T,n}\right\rangle. 
	\end{align}
By using Lemma~\ref{lem:relations}, we derive  
\begin{align}\label{eqn:error estimate for fullydiscrete_truncated_pf4}
	&	\epsilon^2\frac{\left\|\delta_t \theta_{h,\tau}^{T,n}\right\|^2_{L^2}-\left\|\delta_t \theta_{h,\tau}^{T,n-1}\right\|^2_{L^2}}{2\tau}
	+\frac{\left\| \nabla \theta_{h,\tau}^{T,n+1}\right\|_{L^2}^2-\left\|\nabla \theta_{h,\tau}^{T,n-1}\right\|_{L^2}^2}{4\tau}
	+\frac{\left\| \theta_{h,\tau}^{T,n+1}\right\|_{L^2}^2-\left\| \theta_{h,\tau}^{T,n-1}\right\|_{L^2}^2}{4\tau \epsilon^2 } \nn \\
	&+\frac{\left(V \theta_{h,\tau}^{T,n+1},\theta_{h,\tau}^{T,n+1}\right)-\left(V \theta_{h,\tau}^{T,n-1},\theta_{h,\tau}^{T,n-1}\right)}{4\tau } -\Omega^2\epsilon^2\frac{\left\|L_z \theta_{h,\tau}^{T,n+1}\right\|_{L^2}^2-\left\|L_z \theta_{h,\tau}^{T,n-1}\right\|_{L^2}^2}{4\tau}\nn\\
    &
	=-\epsilon^2\text{Re}\left(\delta_{t\overline t}^2 \rho_{h,\tau}^{T,n}
			, \delta_{\hat{t}}\theta_{h,\tau}^{T,n}\right)
			+\text{Re}\left\{\sum_{K}\int_{\partial K}\left(\nabla \widetilde{\Psi}_{\tau}^{n} \cdot \mathbf{n}\right) \delta_{\hat{t}}\theta_{h,\tau}^{T,n} \text{d}s\right\} -
			\frac{1}{\epsilon^2}\text{Re}\left(\widetilde{\rho}_{h,\tau}^{T,n}, \delta_{\hat{t}}\theta_{h,\tau}^{T,n}\right)-
			\text{Re}\left(V\widetilde{\rho}_{h,\tau}^{T,n}, \delta_{\hat{t}}\theta_{h,\tau}^{T,n}\right)\nn \\
			&~~~~-\text{Re}\left(\mathcal{N}_{h,\tau}^{T,n},\delta_{\hat{t}}\theta_{h,\tau}^{T,n} \right)+\text{Re}\left(2\text{i}\Omega \epsilon^2L_z\delta_{\hat{t}}\rho_{h,\tau}^{T,n}, \delta_{\hat{t}}\theta_{h,\tau}^{T,n}\right)_h 
            +\Omega^2\epsilon^2\text{Re}\left(L_z\widetilde{\rho}_{h,\tau}^{T,n}, L_z\delta_{\hat{t}}\theta_{h,\tau}^{T,n}\right)_h
            -\text{Re}\left\langle S^{T,n+1}, \delta_{\hat{t}}\theta_{h,\tau}^{T,n}\right\rangle\nn \\
            &~~~~+\text{Re}\left(2\text{i}\Omega \epsilon^2L_z\delta_{\hat{t}}\theta_{h,\tau}^{T,n}, \delta_{\hat{t}}\theta_{h,\tau}^{T,n}\right)_h. 
\end{align}
We first estimate the first term of the right-hand side of \eqref{eqn:error estimate for fullydiscrete_truncated_pf4}. 
By using \eqref{eqn:Ih_error}, \eqref{eqn:timediscrete_err}, \eqref{eqn:reg} and \eqref{eqn:timediscrete_bound}, we have 
\begin{align}\label{eqn:error estimate for fullydiscrete_truncated_pf5}
    -\epsilon^2\left(\delta_{t\overline t}^2 \rho_{h,\tau}^{T,n}
			, \delta_{\hat{t}}\theta_{h,\tau}^{T,n}\right) &\leq \epsilon^2\left\|\delta_{t\overline t}^2 \rho_{h,\tau}^{T,n}\right\|_{L^2}\left\|\delta_{\hat{t}}\theta_{h,\tau}^{T,n}\right\|_{L^2}
            \leq C_{I_h}\epsilon^2h^2\left\|\delta_{t\overline t}^2 \Psi_\tau^n\right\|_{H^2}\left\|\delta_{\hat{t}}\theta_{h,\tau}^{T,n}\right\|_{L^2}\nn\\
        &\leq Ch^2\left\|\delta_{\hat{t}}\theta_{h,\tau}^{T,n}\right\|_{L^2}
        \leq Ch^4+C\left\|\delta_{\hat{t}}\theta_{h,\tau}^{T,n}\right\|_{L^2}^2. 
\end{align}
 For the second term of the right hand side of \eqref{eqn:error estimate for fullydiscrete_truncated_pf4}, using \eqref{eqn:timediscrete_err}, we have
\begin{align}\label{eqn:error estimate for fullydiscrete_truncated_pf6}
    \text{Re}\left(\sum_{K}\int_{\partial K}\left(\nabla \widetilde{\Psi}_{\tau}^{n} \cdot \mathbf{n}\right) \delta_{\hat{t}}\theta_{h,\tau}^{T,n} \text{d}s\right) &=-\text{Re}\left(\sum_{K}\int_{\partial K}\left(\nabla \widetilde{e}_{\tau}^{n} \cdot \mathbf{n}\right) \delta_{\hat{t}}\theta_{h,\tau}^{T,n} \text{d}s\right) +\text{Re}\left(\sum_{K}\int_{\partial K}\left(\nabla \widetilde{\Psi}^{n} \cdot \mathbf{n}\right) \delta_{\hat{t}}\theta_{h,\tau}^{T,n} \text{d}s\right) \nn \\
    &\leq Ch\left\|\widetilde{e}_{\tau}^{n}\right\|_{H^2}\left\|\delta_{\hat{t}}\theta_{h,\tau}^{T,n}\right\|_{1,h}+Ch^2 \left\|\widetilde{\Psi}^{n}\right\|_{H^3}\left\|\delta_{\hat{t}}\theta_{h,\tau}^{T,n}\right\|_{1,h} \nn \\
    &\leq C\left(h\tau^2+h^2\right) \left\|\delta_{\hat{t}}\theta_{h,\tau}^{T,n}\right\|_{1,h}\leq C\left(h^2\tau^4+h^4\right)+C\left\|\delta_{\hat{t}}\theta_{h,\tau}^{T,n}\right\|_{1,h}^2,
\end{align}
where $e_\tau^n:=\Psi^n-\Psi_\tau^n$ for $n\geq 0$. 
In addition, using \eqref{eqn:Ih_error} and \eqref{eqn:timediscrete_bound}, we obtain 
\begin{align}\label{eqn:error estimate for fullydiscrete_truncated_pf7}
    &-\frac{1}{\epsilon^2}\text{Re}\left(\widetilde{\rho}_{h,\tau}^{T,n}, \delta_{\hat{t}}\theta_{h,\tau}^{T,n}\right)-\text{Re}\left(V\widetilde{\rho}_{h,\tau}^{T,n}, \delta_{\hat{t}}\theta_{h,\tau}^{T,n}\right)+\Omega^2\epsilon^2\text{Re}\left(L_z\widetilde{\rho}_{h,\tau}^{T,n}, L_z\delta_{\hat{t}}\theta_{h,\tau}^{T,n}\right)_h \nn \\ 
    &\leq Ch^2\left\|\widetilde{\Psi}_\tau^n\right\|_{H^2}\left\|\delta_{\hat{t}}\theta_{h,\tau}^{T,n}\right\|_{L^2}  +Ch^2\left\|\widetilde{\Psi}_\tau^n\right\|_{H^3}\left\|\delta_{\hat{t}}\theta_{h,\tau}^{T,n}\right\|_{1,h}  \nn\\
    & \leq Ch^2\left(\left\|\Psi_\tau^{n+1}\right\|_{H^2}+\left\|\Psi_\tau^{n-1}\right\|_{H^2}\right)\left\|\delta_{\hat{t}}\theta_{h,\tau}^{T,n}\right\|_{L^2} + Ch^2\left(\left\|\Psi_\tau^{n+1}\right\|_{H^3}+\left\|\Psi_\tau^{n-1}\right\|_{H^3}\right)\left\|\delta_{\hat{t}}\theta_{h,\tau}^{T,n}\right\|_{1,h} \nn \\
    &\leq Ch^4+\left\|\delta_{\hat{t}}\theta_{h,\tau}^{T,n}\right\|_{L^2}^2+\left\|\delta_{\hat{t}}\theta_{h,\tau}^{T,n}\right\|_{1,h}^2.
\end{align}
To estimate the term \(\text{Re}\left(\mathcal{N}_{h,\tau}^{T,n},\delta_{\hat{t}}\theta_{h,\tau}^{T,n} \right)\), similarly to \eqref{eqn:timediscrete_truncate_err_pf6}, we first use the bound
\begin{align}\label{eqn:error estimate for fullydiscrete_truncated_pf8}
   \left\|\mathcal{N}_{h,\tau}^{T,n}\right\| \leq C_\mathcal{N} \left( \left\|e_{h,\tau}^{T,n+1}\right\|_{L^2} + \left\|e_{h,\tau}^{T,n-1}\right\|_{L^2} \right).
\end{align}
Consequently, using \eqref{eqn:Ih_error} and \eqref{eqn:timediscrete_bound}, we obtain
\begin{align}\label{eqn:error estimate for fullydiscrete_truncated_pf9}
\text{Re}\left(\mathcal{N}_{h,\tau}^{T,n},\delta_{\hat{t}}\theta_{h,\tau}^{T,n} \right) &\leq  
C \left( \left\|\rho_{h,\tau}^{T,n+1}\right\|_{L^2} + \left\|\rho_{h,\tau}^{T,n-1}\right\|_{L^2} + \left\|\theta_{h,\tau}^{T,n+1}\right\|_{L^2} + \left\|\theta_{h,\tau}^{T,n-1}\right\|_{L^2} \right) \left\|\delta_{\hat{t}}\theta_{h,\tau}^{T,n}\right\|_{L^2}\leq Ch^4 + \left\|\delta_{\hat{t}}\theta_{h,\tau}^{T,n}\right\|_{L^2}^2.
\end{align}
In view of \eqref{eqn:ritz_error}, \eqref{eqn:Ih_error}, and using \eqref{lem:timediscrete_err} and \eqref{eqn:reg}, we obtain 
\begin{align}\label{eqn:error estimate for fullydiscrete_truncated_pf10}
    &\text{Re}\left(2\text{i}\Omega \epsilon^2\left(L_z\delta_{\hat{t}}\rho_{h,\tau}^{T,n}, \delta_{\hat{t}}\theta_{h,\tau}^{T,n}\right)_h \right)= -2\Omega \epsilon^2\text{Im}\left(L_z\delta_{\hat{t}}\rho_{h,\tau}^{T,n}, \delta_{\hat{t}}\theta_{h,\tau}^{T,n}\right)_h\nn \\
    & = 2\Omega \epsilon^2\text{Im}\left(L_z\delta_{\hat{t}}\left(e_{\tau}^{n}-I_h e_{\tau}^{n}\right), \delta_{\hat{t}}\theta_{h,\tau}^{T,n}\right)_h-2\Omega \epsilon^2\text{Im}\left(L_z\delta_{\hat{t}}\left(\Psi^n-I_h \Psi^n\right), \delta_{\hat{t}}\theta_{h,\tau}^{T,n}\right)_h \nn \\
    & = 2\Omega \epsilon^2\text{Im}\left(L_z\delta_{\hat{t}}\left(e_{\tau}^{n}-I_h e_{\tau}^{n}\right), \delta_{\hat{t}}\theta_{h,\tau}^{T,n}\right)_h-2\Omega \epsilon^2\text{Im}\left(L_z\delta_{\hat{t}}\left(\Psi^n-R_h \Psi^n\right), \delta_{\hat{t}}\theta_{h,\tau}^{T,n}\right)_h 
    -2\Omega \epsilon^2\text{Im}\left(L_z\delta_{\hat{t}}\left(R_h\Psi^n-I_h \Psi^n\right), \delta_{\hat{t}}\theta_{h,\tau}^{T,n}\right)_h
    \nn \\
     &\leq   2\Omega\epsilon^2
     \bigg\|L_z\delta_{\hat{t}}\left(e_{\tau}^{n}-I_h e_{\tau}^{n}\right)\bigg\|_{L^2}\bigg\|\delta_{\hat{t}}\theta_{h,\tau}^{T,n}\bigg\|_{L^2}
     +2\Omega\epsilon^2\bigg\|\delta_{\hat{t}}\left(\Psi^n-R_h \Psi^n\right) \bigg\|_{1,h}\bigg\|\delta_{\hat{t}}\theta_{h,\tau}^{T,n}\bigg\|_{L^2}+2\Omega\epsilon^2\bigg\|\delta_{\hat{t}}\left(R_h\Psi^n-I_h \Psi^n\right)\bigg\|_{1,h}\bigg\|\delta_{\hat{t}}\theta_{h,\tau}^{T,n}\bigg\|_{L^2}\nn \\
     &\leq Ch\bigg\|L_z\delta_{\hat{t}}e_{\tau}^{n}\bigg\|_{H^1}\bigg\|\delta_{\hat{t}}\theta_{h,\tau}^{T,n}\bigg\|_{L^2}+Ch^2 \bigg\|\delta_{\hat{t}}\Psi^{n}\bigg\|_{H^3}\bigg\|\delta_{\hat{t}}\theta_{h,\tau}^{T,n}\bigg\|_{L^2} 
    \nn \\
&\leq C(h\tau^2+h^2)\left\|\delta_{\hat{t}}\theta_{h,\tau}^{T,n}\right\|_{L^2}
     \leq C(h^2\tau^4+h^4)+\left\|\delta_{\hat{t}}\theta_{h,\tau}^{T,n}\right\|_{L^2} ^2,
\end{align}
where we use the estimate: $\left\|R_h\Psi-I_h\Psi\right\|_{H^1} \leq Ch^2\left\|\Psi\right\|_{H^3}$ \cite{lin2006finite}. 
Moreover, using integration by parts, we have  
\begin{align}\label{eqn:error estimate for fullydiscrete_truncated_pf12}
& \text{Re}\left(2\text{i}\Omega \epsilon^2L_z\delta_{\hat{t}}\theta_{h,\tau}^{T,n}, \delta_{\hat{t}}\theta_{h,\tau}^{T,n}\right)_h= 2\Omega \epsilon^2
 \text{Re}\left(\left[x\partial_y-y\partial_x\right]\delta_{\hat{t}}\theta_{h,\tau}^{T,n}, \delta_{\hat{t}}\theta_{h,\tau}^{T,n}\right)_h \nn \\
&~~~=2\Omega \epsilon^2\left[\left(\left[x\partial_y-y\partial_x\right]\delta_{\hat{t}}\theta_{h,\tau}^{T,n}, \delta_{\hat{t}}\theta_{h,\tau}^{T,n}\right)_h+\left(\delta_{\hat{t}}\theta_{h,\tau}^{T,n}, \left[x\partial_y-y\partial_x\right]\delta_{\hat{t}}\theta_{h,\tau}^{T,n}\right)_h\right] =  2\Omega \epsilon^2
 \text{Re}\left\langle\delta_{\hat{t}}\theta_{h,\tau}^{T,n}, \delta_{\hat{t}}\theta_{h,\tau}^{T,n} \right\rangle.
\end{align}
Then by virtue of $\left\langle \cdot, \omega_h \right\rangle = O(h^2)\left\| \omega_h\right\|_{1,h}$, and similar as \eqref{eqn:error estimate for fullydiscrete_truncated_pf6}, we can get
 \begin{align}\label{eqn:error estimate for fullydiscrete_truncated_pf13}
    & \text{Re}\left(2\text{i}\Omega \epsilon^2L_z\delta_{\hat{t}}\theta_{h,\tau}^{T,n}, \delta_{\hat{t}}\theta_{h,\tau}^{T,n}\right)_h-\text{Re}\left\langle S^{T,n+1}, \delta_{\hat{t}}\theta_{h,\tau}^{T,n}\right\rangle 
      = 2\Omega \epsilon^2
 \text{Re}\left\langle\delta_{\hat{t}}\theta_{h,\tau}^{T,n}, \delta_{\hat{t}}\theta_{h,\tau}^{T,n} \right\rangle-\text{Re}\left\langle S^{T,n+1}, \delta_{\hat{t}}\theta_{h,\tau}^{T,n}\right\rangle \nn \\
     & ~~~~\leq C(h\tau^2+h^2) \left\|\delta_{\hat{t}}\theta_{h,\tau}^{T,n}\right\|_{1,h}\leq C(h^2\tau^4+h^4)+\left\|\delta_{\hat{t}}\theta_{h,\tau}^{T,n}\right\|_{1,h} ^2.
 \end{align}
For convenience, we denote 
\begin{align*}
    \mathscr E_{h,\tau}^{T, n}
    := 2\epsilon^2\left\|\delta_t \theta_{h,\tau}^{T,n}\right\|^2_{L^2}
    +\left[\left\| \nabla \theta_{h,\tau}^{T,n+1}\right\|_{L^2}^2+\left\| \nabla \theta_{h,\tau}^{T,n}\right\|_{L^2}^2-\Omega^2\epsilon^2\left(\left\|L_z \theta_{h,\tau}^{T,n+1}\right\|_{L^2}^2 +\left\|L_z \theta_{h,\tau}^{T,n} \right\|_{L^2}^2 \right)\right]
   +\left(\frac{1}{\epsilon^2}+V\right)\left(\left\| \theta_{h,\tau}^{T,n+1}\right\|_{L^2}^2+\left\| \theta_{h,\tau}^{T,n}\right\|_{L^2}^2 \right). 
\end{align*}
Then from
\eqref{eqn:error estimate for fullydiscrete_truncated_pf5}, \eqref{eqn:error estimate for fullydiscrete_truncated_pf6}, \eqref{eqn:error estimate for fullydiscrete_truncated_pf7}, \eqref{eqn:error estimate for fullydiscrete_truncated_pf9}, \eqref{eqn:error estimate for fullydiscrete_truncated_pf10}, \eqref{eqn:error estimate for fullydiscrete_truncated_pf13} and \eqref{eqn:error estimate for fullydiscrete_truncated_pf4}, we have
\begin{align}\label{eqn:error estimate for fullydiscrete_truncated_pf15}
\mathscr E_{h,\tau}^{T, n}\leq \mathscr E_{h,\tau}^{T, n-1}
+C\tau\left(\mathscr E_{h,\tau}^{T, n}+\mathscr E_{h,\tau}^{T, n-1}\right)+C\tau\left(h^2\tau^4+h^4\right). 
\end{align}
By applying the discrete Gronwall inequality to \eqref{eqn:error estimate for fullydiscrete_truncated_pf15}, we obtain  
\begin{align}\label{eqn:error estimate for fullydiscrete_truncated_pf16}  
    \mathscr{E}_{h,\tau}^{T, n} \leq C \left(h^2\tau^4+h^4\right),  
\end{align}  
provided that \(\tau\) is sufficiently small, which further implies that 
\begin{align}\label{eqn:error estimate for fullydiscrete_truncated_pf17}  
    \left\|\delta_t \theta_{h,\tau}^{T,n}\right\|_{L^2}
    +\left\| \theta_{h,\tau}^{T,n+1}\right\|_{1,h}\leq C\left(h\tau^2+h^2\right). 
\end{align}
Finally, combining \eqref{eqn:Ih_error}, \eqref{eqn:timediscrete_bound} and \eqref{eqn:error estimate for fullydiscrete_truncated_pf17}, we obtain 
\begin{align}\label{eqn:error estimate for fullydiscrete_truncated_pf18} 
   \left\|\Psi_{\tau}^{n+1}-\Psi_{h,\tau}^{T,n+1}\right\|_{L^2}\leq C\left(h\tau^2+h^2\right),\quad 
\left\|\Psi_{\tau}^{n+1}-\Psi_{h,\tau}^{T,n+1}\right\|_{1, h} \leq Ch. 
 \end{align}
 Therefore, from \eqref{eqn:error estimate for fullydiscrete_truncated_pf17} and \eqref{eqn:error estimate for fullydiscrete_truncated_pf18}, we complete the proof. 
\end{proof}    

\begin{lemma}\label{lem:fullydiscrete_err}
 Suppose that $\Psi_{\tau}^{n+1}$, $1\leq n\leq N-1$ is the solution of \eqref{eqn:timediscrete}, and denote by $\Psi_{h,\tau}^{n+1}$, $1\leq n\leq N-1$ the solution of \eqref{eqn:fullydiscrete}. Then, the following estimates hold
\begin{align}\label{eqn:fullydiscrete_err}
   \left\|\Psi_{\tau}^m-\Psi_{h,\tau}^{m}\right\|_{L^2}\leq C\left(h\tau^2+h^2\right),\quad 
   \left\|I_h \Psi_{\tau}^m- \Psi_{h,\tau}^{m}\right\|_{1, h}\leq C\left(h\tau^2+h^2\right),
   \quad 
\left\|\Psi_{\tau}^m-\Psi_{h,\tau}^{m}\right\|_{1, h} \leq Ch,  \qquad 0\leq m\leq N, 
 \end{align}
 where $C>0$ is a constant independent of $\tau$ and $h$.
\end{lemma}
\begin{proof}
The unique existence of \eqref{eqn:fullydiscrete_truncated} can be demonstrated similarly as Lemma \ref{lem:timediscrete_truncate_existence} and Lemma \ref{lem:timediscrete_err}. Furthermore, from the error estimates \eqref{eqn:fullydiscrete_err_truncated}, \eqref{eqn:timediscrete_err} and using \eqref{eqn:timebound}, for sufficiently small \(\tau\), we obtain  
\begin{align}\label{eqn:fullydiscrete_err_pf2}
    \left\|\Psi_{h,\tau}^{T, m}\right\|_{L^\infty} \leq \left\|\Psi^{m}\right\|_{L^\infty}+\left\|\Psi^{m}-\Psi_{\tau}^{m}\right\|_{L^\infty}+
    \left\|\Psi_{\tau}^{m}-\Psi_{h,\tau}^{T, m}\right\|_{L^\infty} \leq \left\|\Psi^{m}\right\|_{L^\infty}+C\left\|\Psi^{m}-\Psi_{\tau}^{m}\right\|_{H^2}+
    C\left\|\Psi_{\tau}^{m}-\Psi_{h,\tau}^{T, m}\right\|_{H^2} \leq K_0,
\end{align}  
for $0\leq m\leq N$. 
By the definition of the truncation function \(\mu_A(\cdot)\), it follows that  
\begin{align}\label{eqn:fullydiscrete_err_pf3}
\mu_A\left(\left|\Psi^{T, n+1}_{h,\tau}\right|^2\right)=\left|\Psi^{T, n+1}_{h,\tau}\right|^2,\qquad 
\mu_A\left(\left|\Psi^{T, n-1}_{h,\tau}\right|^2\right)=\left|\Psi^{T, n-1}_{h,\tau}\right|^2.
\end{align}  
Thus, we can rewrite \eqref{eqn:fullydiscrete_truncated} as  
\begin{align}\label{eqn:fullydiscrete_err_pf4}			&\epsilon^2\left(\delta_{t\overline t}^2\Psi_{h,\tau}^{T,n}
			, \omega_h\right)
			+\left(\nabla \widetilde{\Psi}_{h,\tau}^{T,n}, \nabla \omega_h\right)_h
			+
			\frac{1}{\epsilon^2}\left(\widetilde{\Psi}_{h,\tau}^{T,n}, \omega_h\right)
			+
			\left(V\widetilde{\Psi}_{h,\tau}^n, \omega_h\right)
			+\lambda\,\left(\frac{\left|\Psi^{T,n+1}_{h,\tau}\right|^2				
				+\left|\Psi^{T,n-1}_{h,\tau}\right|^2}{2}\widetilde{\Psi}_{h,\tau}^{T,n}, \omega_h\right)\nn\\
			&-2\text{i}\Omega \epsilon^2\left(L_z\delta_{\hat{t}}\Psi_{h,\tau}^{T,n}, \omega_h\right)_h-
			\Omega^2\epsilon^2\left(L_z\widetilde{\Psi}_{h,\tau}^{T,n}, L_z\omega_h\right)_h
	+\left\langle S^{T,n+1}, \omega_h\right\rangle=0,
	\qquad \forall\, \omega_h\in V_h,
	\end{align}  
which is identical to the full-discrete method \eqref{eqn:fullydiscrete}. Consequently, the unique existence of the truncated solution \(\Psi_{h,\tau}^{T, n+1}\) to the truncated full-discrete method \eqref{eqn:fullydiscrete_truncated} directly implies the unique existence of \(\Psi_{h,\tau}^{n+1}\) for the full-discrete method \eqref{eqn:fullydiscrete}. Overall, we can conclude that 
\begin{align}\label{eqn:fullydiscrete_err_pf5}
    \Psi_{h,\tau}^{T, m} =\Psi_{h,\tau}^{m}\qquad \text{for all}~~m\geq 0. 
\end{align}
Therefore, \eqref{eqn:fullydiscrete_err} can be directly obtained from \eqref{eqn:fullydiscrete_err_truncated}. We have completed the proof.  
\end{proof}

Next, we demonstrate the results shown in Theorem \ref{thm:main}. The $L^\infty$-boundedness of $\Psi_{h, \tau}^{T, m}$ in \eqref{eqn:fullydiscrete_err_pf2} and $\Psi_{h,\tau}^{T, m} =\Psi_{h,\tau}^{m}$ in \eqref{eqn:fullydiscrete_err_pf5} immediately yield \eqref{eqn:L_infty}. In addition,
using \eqref{eqn:timediscrete_err} and \eqref{eqn:fullydiscrete_err}, we obtain the optimal $L^2$- norm error estimate
\begin{align}\label{eqn:errr1}
   \left\|\Psi^m-\Psi_{h,\tau}^{m} \right\|_{L^2} 
   \leq  \left\|\Psi^m-\Psi_{\tau}^{m} \right\|_{L^2}+\left\|\Psi_{\tau}^{m} -\Psi_{h,\tau}^{m} \right\|_{L^2} \leq C\tau^2 + C(h\tau^2+h^2) \leq C (h^2+\tau^2),
\end{align}
and the optimal $H^1$-norm error estimate
\begin{align}\label{eqn:errr2}
   \left\|\Psi^m-\Psi_{h,\tau}^{m} \right\|_{1,h} 
   &\leq  \left\|\Psi^m-\Psi_{\tau}^{m} \right\|_{H^1}+\left\|\Psi_{\tau}^{m} -\Psi_{h,\tau}^{m} \right\|_{1,h} \leq  C (h+\tau^2).
\end{align}
The estimates \eqref{eqn:errr1} and \eqref{eqn:errr2} implies \eqref{eqn:convergenceL2} and \eqref{eqn:convergenceH1}, respectively. 
In addition, from \eqref{eqn:timediscrete_truncate_err} and \eqref{eqn:fullydiscrete_err}, 
we have 
\begin{align}\label{eqn:errr3}
    \left\|I_h\Psi^m-\Psi_{h,\tau}^m\right\|_{1,h}&\leq 
    \left\|I_h\Psi^m-I_h\Psi_\tau^m\right\|_{1,h}
    + \left\|I_h\Psi_\tau^m-\Psi_{h,\tau}^m\right\|_{1,h}\nn\\
    &\leq C\left\|\Psi^m-\Psi_\tau^m\right\|_{H^1}
    + \left\|I_h\Psi_\tau^m-\Psi_{h,\tau}^m\right\|_{1,h}\nn\\
    &\leq C\tau^2+C(h\tau^2+h^2)\leq C(h^2+\tau^2),
\end{align}
which demonstrates \eqref{eqn:supercloseH1}. Finally, based on interpolated postprocessing technique, we obtain 
\begin{align}\label{eqn:errr4}
\left\|\Psi^m-I_{2h}\Psi_h^{m}\right\|_{1,h}\leq C_E(h^2+\tau^2),
\end{align}
where $I_{2h}$ is the interpolated postprocessing operator \cite{lin2006finite}. Therefore, we have completed the proof of Theorem \ref{thm:main}.  
	
\begin{rem}
In our convergence analysis, we do not provide a detailed discussion of the multiscale parameter $\epsilon$ in order to simplify the theoretical derivation. 
However, $\epsilon$ plays a crucial role in the convergence analysis. 
Due to the presence of the term $-\Omega^2 \epsilon^2 L_z^2 \Psi$ in the RKG equation, this term must be controlled by $-\Delta \Psi$, 
which requires assuming that $\epsilon$ is bounded above by a certain constant (see \eqref{eqn:ep}). 
On the other hand, when estimating the term $\mathcal L_\tau^{T, n}$ in \eqref{eqn:timediscrete_truncate_err_pf2}, 
the resulting truncation error is actually of order~$\mathcal O(\tau^2/\epsilon^2)$. 
Therefore, if the small parameter~$\epsilon$ is incorporated into the error analysis, 
the convergence rates are expected to satisfy
\begin{align*}
	\sup_{0\leq n\leq N}\left\|\Psi^n-\Psi_{h,\tau}^n\right\|_{L^2}
	\leq C\left(h^2+\frac{\tau^2}{\epsilon^2}\right), \quad 
	\sup_{0\leq n\leq N}\left\|\Psi^n-\Psi_{h,\tau}^n\right\|_{1,h}
	\leq C\left(h+\frac{\tau^2}{\epsilon^2}\right), \quad
	\sup_{0\leq n\leq N}\left\|\Psi^n-I_{2h}\Psi_{h,\tau}^n\right\|_{1,h}
	\leq C\left(h^2+\frac{\tau^2}{\epsilon^2}\right),
\end{align*}
where the constant~$C$ is independent of~$h$, $\tau$, and $\epsilon$.
Since this derivation is rather involved, we will address it in detail in our future work, 
where the effect of~$\epsilon$ on the numerical error and convergence will be systematically investigated.
\end{rem}

\section{Numerical results}\label{sec5}
In this section, we first verify the accuracy of the conforming and nonconforming FEMs as stated in Theorem~\ref{thm:main}. Next, we examine the discrete energy and charge conservation properties of both schemes. Finally, we employ the proposed numerical methods to simulate vortex dynamics governed by the RKG equation~\eqref{eq:model2}, covering both relativistic and nonrelativistic regimes.

\subsection{Accuracy and structure-preservation tests}
\begin{example} [\textbf{Accuracy test: smooth solution}]
    In this example, we test the accuracy of the conforming and nonconforming FEMs defined in Definition~\ref{Def:fully}. To this end, we add a source term at the right-hand side of the
model \eqref{eq:model2} with the exact solution $$\Psi(x, y, t)=(t+1)^3\sin\pi x\sin\pi y.$$ 
In the tests of this example, we set
\begin{align*}
    U=[-1, 1]\times [-1, 1],~~T = 1, ~~\lambda=1,~~\Omega=0.8,~~\epsilon = 0.01,~~V=\frac{x^2+y^2}{2}e^{-(x^2+y^2)}.
\end{align*}
Based on the above settings, we perform numerical experiments to verify the error estimates presented in Theorem~\ref{thm:main}. 
The results in Figure~\ref{figure:2} clearly confirm the theoretical predictions: the $H^1$-norm converges at first order, while both the $L^2$-norm and the higher-order $H^1$-norm achieve second-order accuracy for the conforming and nonconforming FEMs.
Since a source term is introduced in this example, it does not represent the model considered in this paper. The aim here is solely to examine the convergence of the numerical method, and accordingly the solution does not preserve energy and charge conservation.
\begin{figure}
    \centering \includegraphics[width=0.47\linewidth]{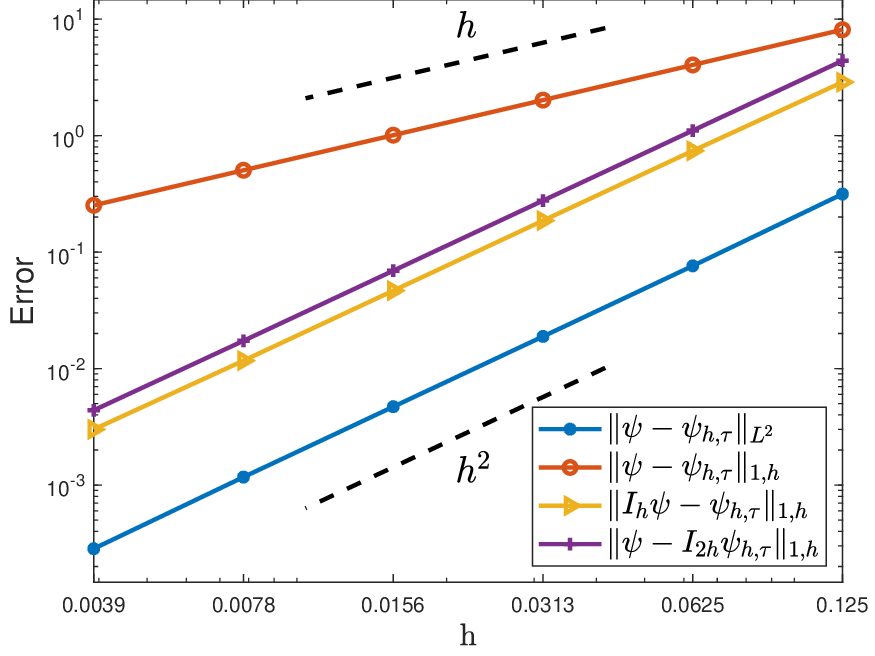} \includegraphics[width=0.47\linewidth]{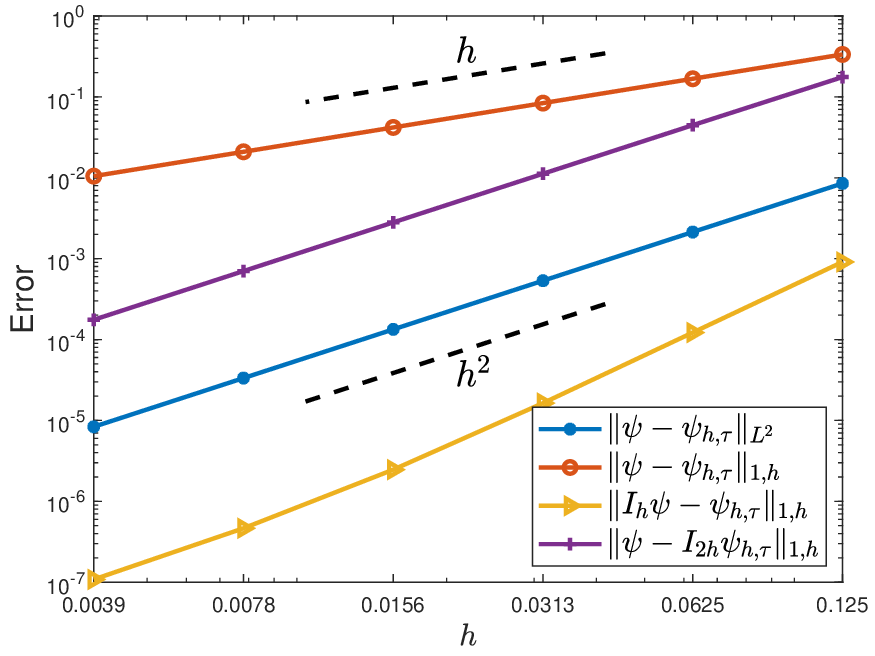}
    \caption{Error estimates and convergence rates obtained with the conforming (left) and nonconforming (right) FEMs (Example 5.1).}
    \label{figure:2}
\end{figure}
\end{example}

\begin{example} [\textbf{Accuracy test: non-smooth solution}]
In this example, we compare the convergence of the conforming and nonconforming FEMs for the RKG equation with a source term, under a non-smooth solution. The exact solution in this example is chosen as 
\begin{align*}
    u(x,y,t) = (t+1)^3 \bigg(1-|2x-1|^2\bigg)^s \bigg(1-|2y-1|^2\bigg)^s
,\qquad s\in\left[1/2,1\right].
\end{align*}
In the tests of this example, we set
\begin{align*}
    U=[0, 1]\times [0, 1],~~T = 0.1, ~~\lambda=1,~~\Omega=0.8,~~\epsilon = 0.001,~~V=\frac{x^2+y^2}{2}e^{-(x^2+y^2)}.
\end{align*}
Figure~\ref{figure:3} presents the error plots for different values of $s$. When the solution is non-smooth, the nonconforming FEM shows superior convergence. For sufficiently smooth solutions, the nonconforming FEM still attains slightly higher accuracy, but its performance is nearly indistinguishable from that of the conforming FEM.
\begin{figure}
    \centering \includegraphics[width=0.33\linewidth]{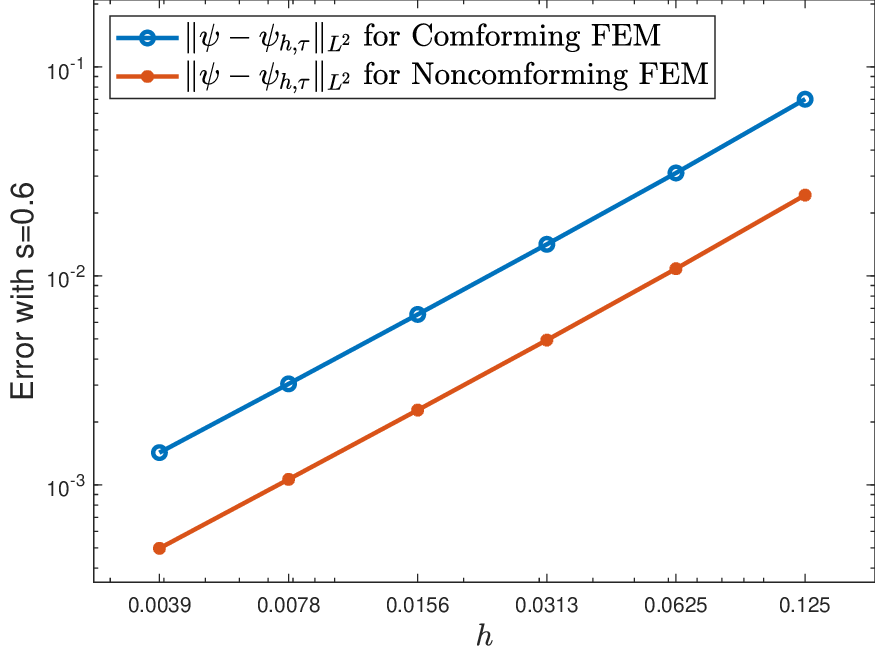} \includegraphics[width=0.33\linewidth]{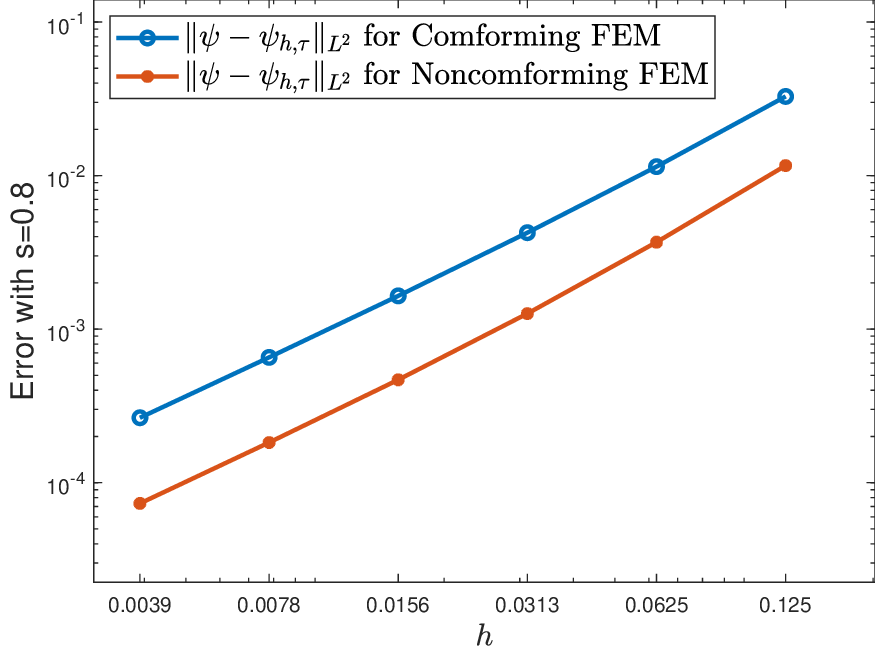} \includegraphics[width=0.33\linewidth]{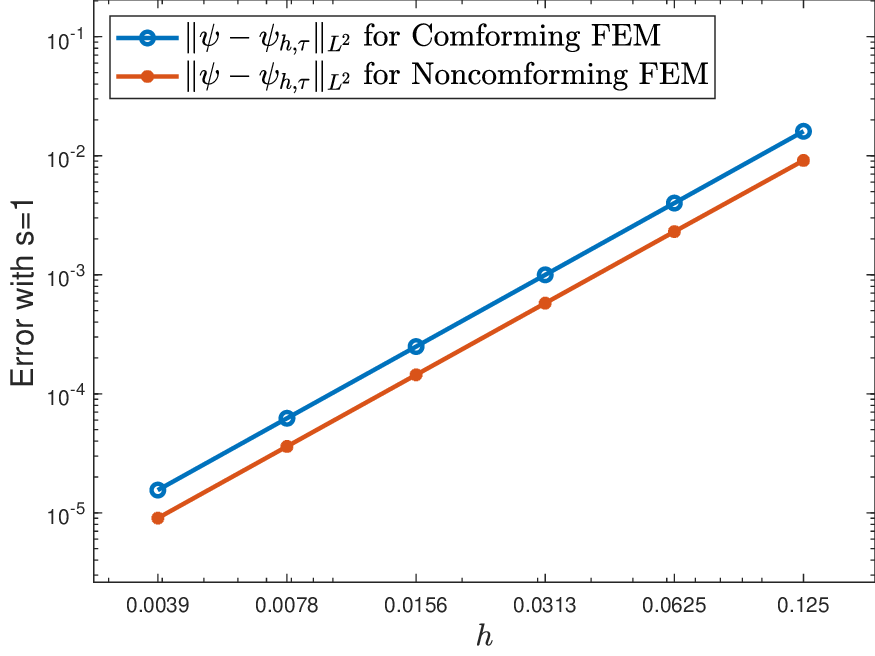}
    \caption{Error estimates and convergence rates obtained with the conforming (left) and nonconforming (right) FEMs (Example 5.2).}
    \label{figure:3}
\end{figure}
\end{example}

\begin{example} [\textbf{Structure-preservation tests}]
In this example, we investigate the structure-preserving properties of the conforming and nonconforming FEMs, as stated in Theorem~\ref{thm-fullydiscrete-conservation}, focusing on discrete energy and charge conservation. The initial data are chosen as
\begin{align*}
    \psi_0(x,y) &= \frac{x+\textup{i}y}{x^2+y^2+1}e^{-\frac{2x^2+1.5y^2}{2}}\,, \qquad 
    \psi_1(x,y) =\textup{i}\omega\psi_0 \,.
\end{align*}
The computational domain and parameters are specified as
\[
    U = [-8, 8]\times[-8, 8], \quad 
    T = 10, \quad 
    \lambda = 0.5, \quad 
    \Omega = 0.8,\quad \omega=1e-3, \quad 
    V = \tfrac{x^2+y^2}{2} e^{-(x^2+y^2)},
\]
and different values of $\epsilon = 10^{-i}$, with $i = 3,4,5$, are tested.  
The relative errors of the discrete energy and charge are defined as
\begin{align*}
  \text{Energy error}  = \frac{\big| E_h(t) - E_h(0) \big|}{E_h(0)},  \qquad \text{Charge error} = \frac{\big| Q_h(t) - Q_h(0) \big|}{Q_h(0)},
\end{align*}
where $E_h(t)$ and $Q_h(t)$ denote the discrete energy and charge at time $t$, such that 
\begin{align*}
    E_h(t_n):=E_h^n,\qquad Q_h(t_n):=Q_h^n,\qquad 1\leq n\leq N. 
\end{align*}
 The numerical results, presented in Figures~\ref{figure:4}--\ref{figure:7}, are in full agreement with the theoretical predictions of Theorem~\ref{thm-fullydiscrete-conservation}. In particular, for the nonconforming FEM, the numerical experiments indicate that including the conservation-adjusting term $\langle S^{n+1}, \omega_h \rangle$ is crucial for maintaining energy and charge conservation. 
\end{example}

\begin{figure}
    \centering \includegraphics[width=0.47\linewidth]{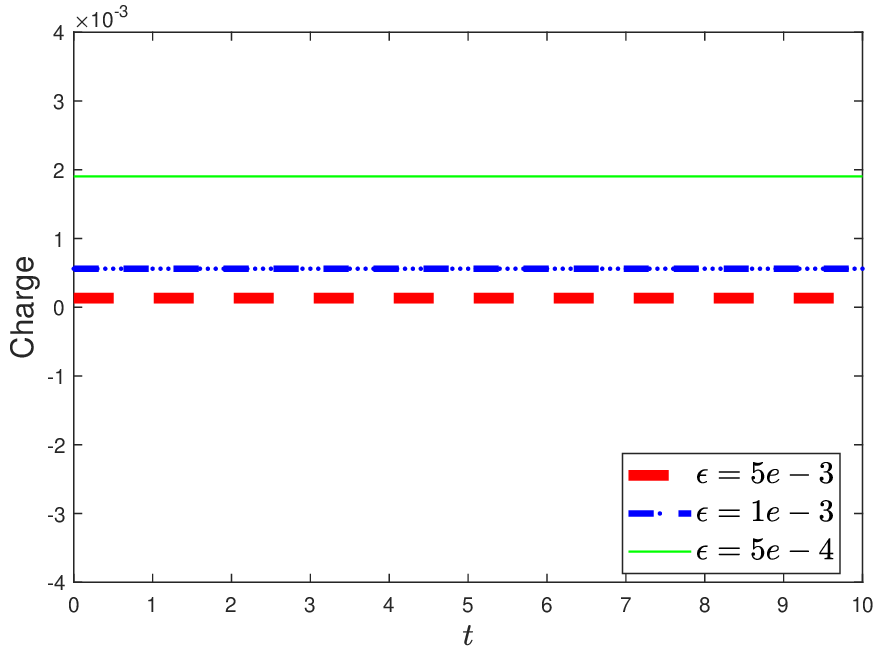} \includegraphics[width=0.47\linewidth]{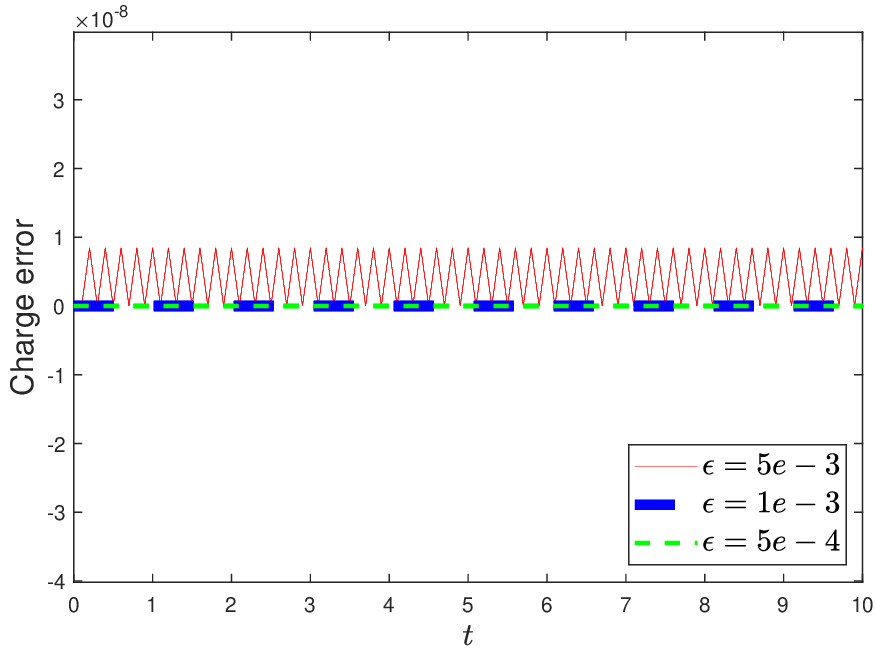}
    \caption{The discrete charge and its relative error for the conforming FEM (Example 5.3).}
    \label{figure:4}
\end{figure}
\begin{figure}
    \centering \includegraphics[width=0.47\linewidth]{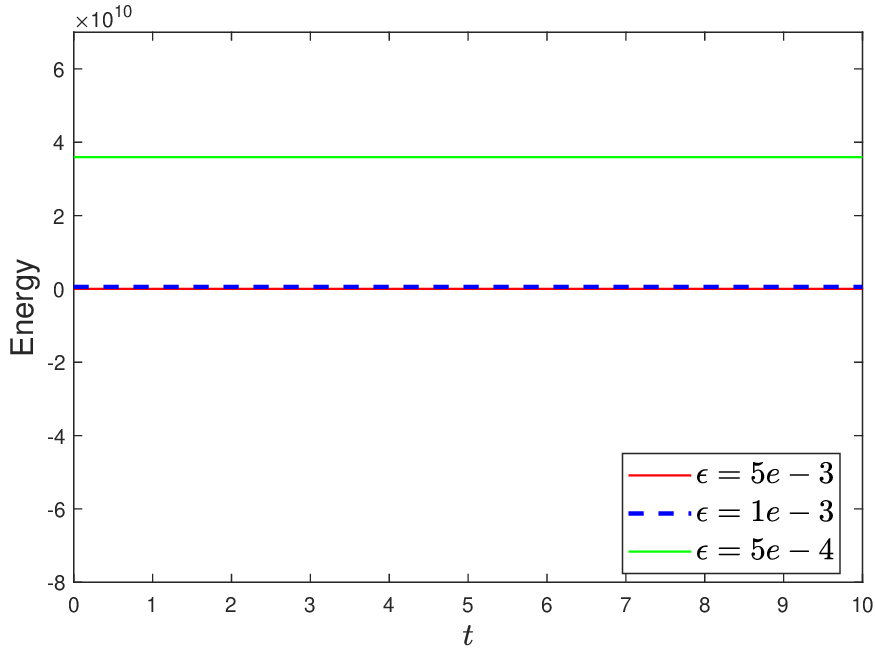} \includegraphics[width=0.47\linewidth]{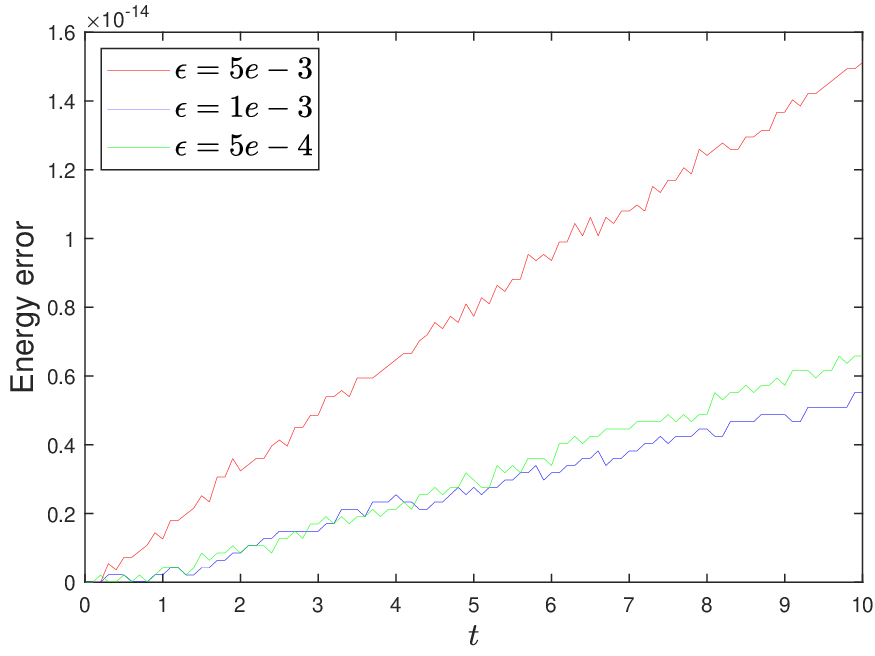}
    \caption{The discrete energy and its relative error for the conforming FEM (Example 5.3).}
    \label{figure:5}
\end{figure}
\begin{figure}
    \centering \includegraphics[width=0.47\linewidth]{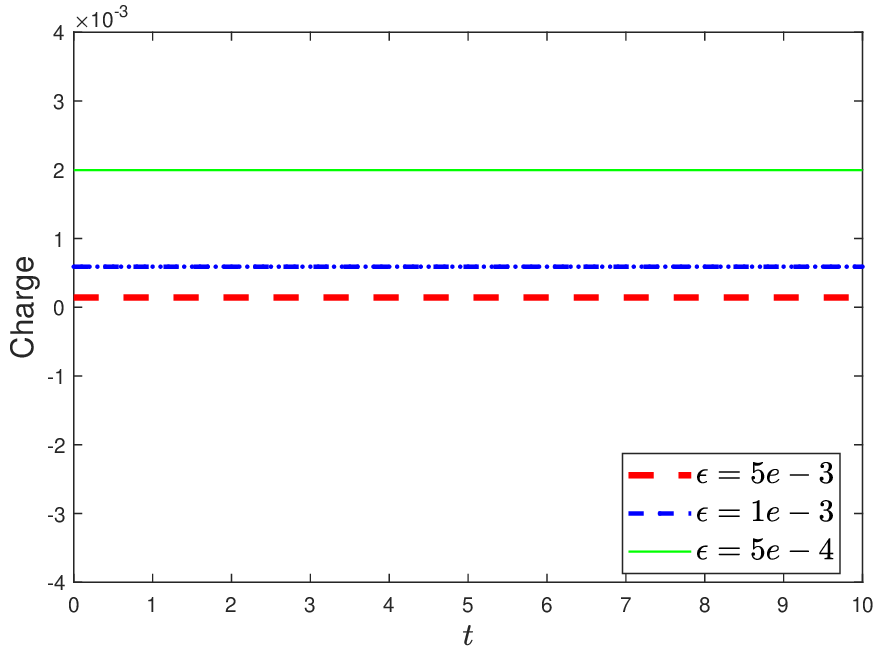} \includegraphics[width=0.47\linewidth]{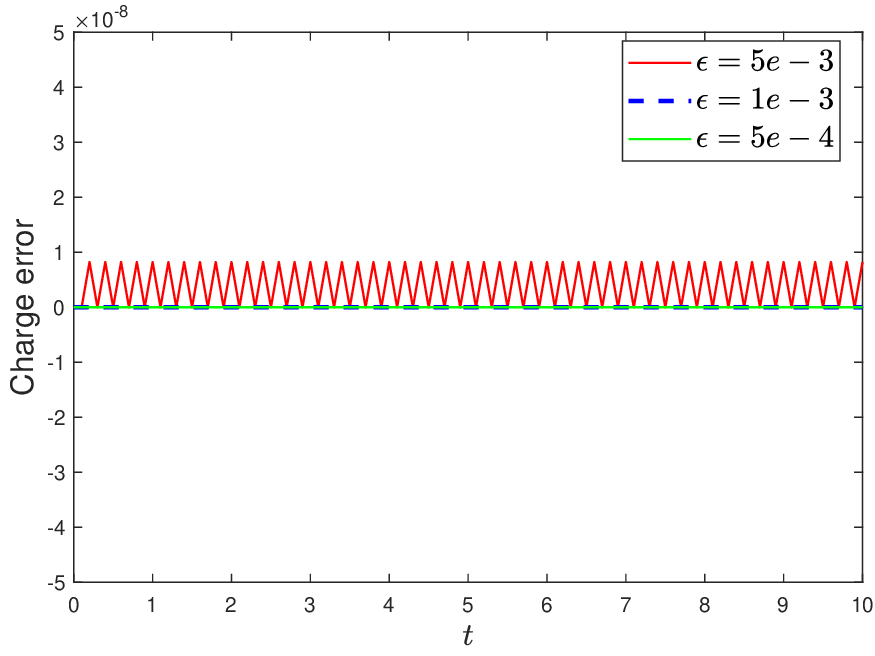}
    \caption{The discrete charge and its relative error for the nonconforming FEM (Example 5.3).}
    \label{figure:6}
\end{figure}
\begin{figure}
    \centering \includegraphics[width=0.47\linewidth]{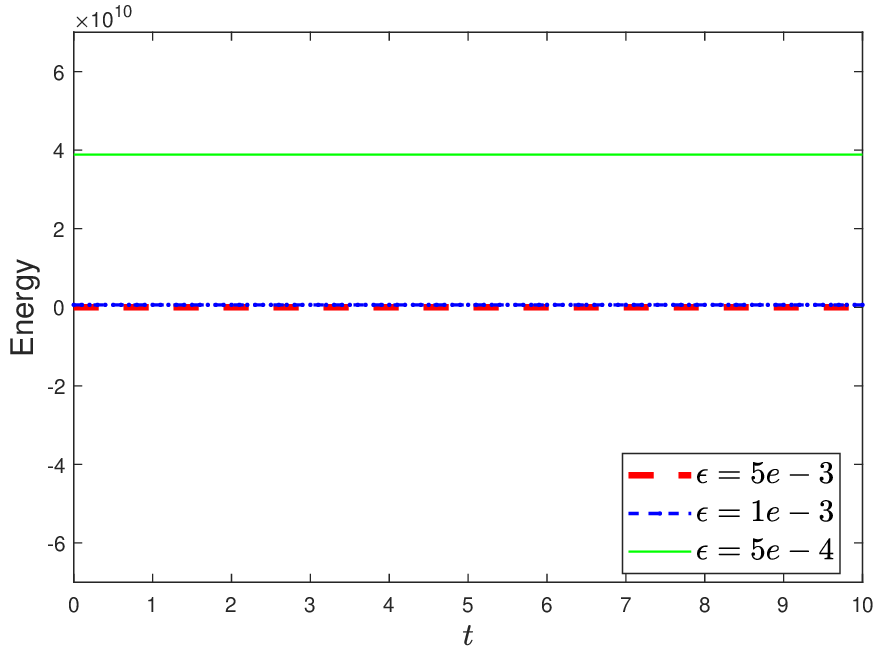} \includegraphics[width=0.47\linewidth]{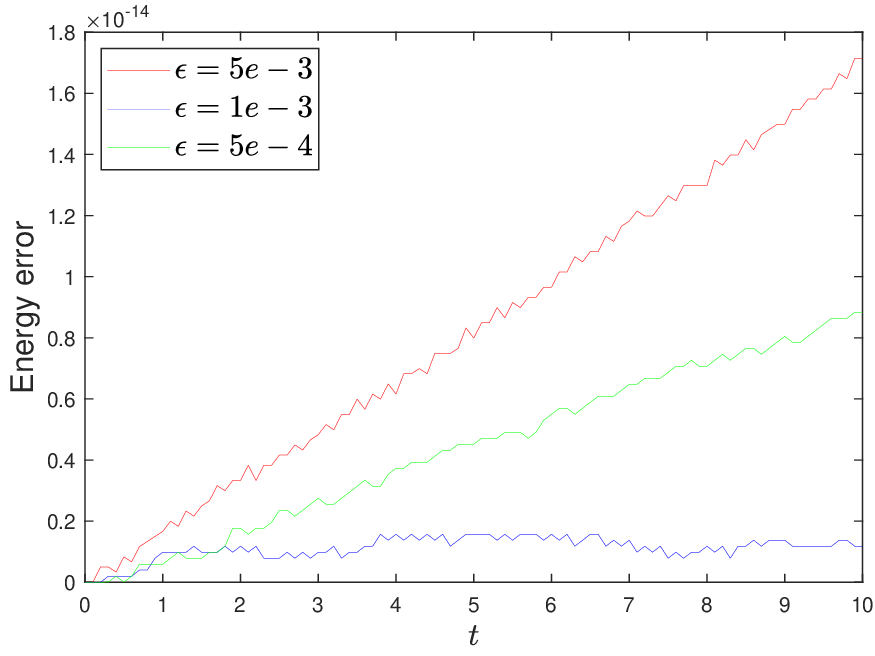}
    \caption{The discrete charge and its relative error for the nonconforming FEM (Example 5.3).}
    \label{figure:7}
\end{figure}
\subsection{Simulations of dynamics}
We apply the proposed numerical method to simulate the vortex dynamics in the RKG equation, covering the regime from the relativistic to the nonrelativistic case. For convenience, in the following examples, we only adopt the conforming FEM. 

\begin{example}[\textbf{Generation of vortices}]
In this example, we investigate vortex generation in the relativistic regime, 
i.e., $\epsilon = 1$. The initial data are chosen as
\begin{equation*}
\psi_0(x,y) = e^{-2x^2 - 1.5y^2 + \textup{i}N_0 \arctan\,\left(\tfrac{y}{x}\right)}, 
\quad 
\psi_1(x,y) = \textup{i} \psi_0(x,y),
\end{equation*}
with $N_0 = 2, 3$. In this test, we set 
\begin{equation*}
U = [-5,5]\times[-5,5], 
\quad T = 7, 
\quad \lambda = 2, 
\quad \Omega = 1, 
\quad V = x^2 + y^2.
\end{equation*}
The dynamical evolution is illustrated in Figures \ref{figure:8}-\ref{figure:9}. We observe that no vortices are present initially; however, as time progresses, vortices emerge, and their number coincides with $N_0$.
\end{example}
\begin{figure}
    \centering
    \includegraphics[width=1\linewidth]{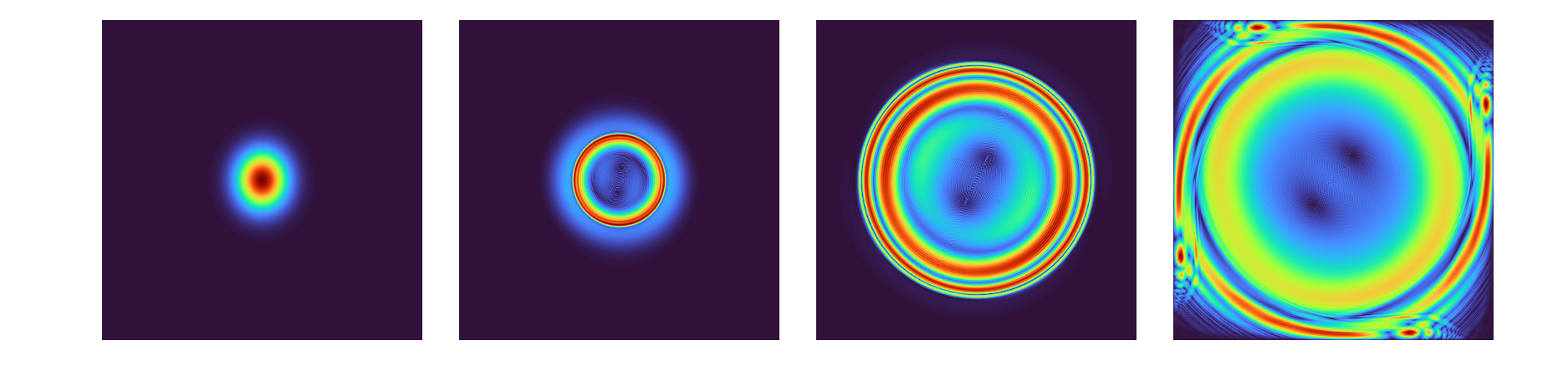}
    \includegraphics[width=1\linewidth]{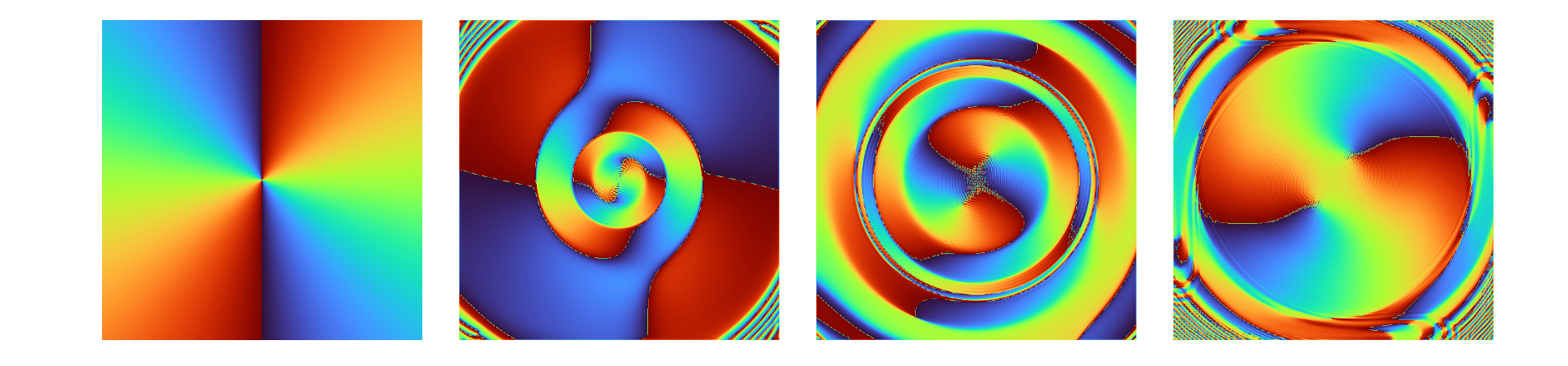}
    \caption{Contour plot of $|\Psi_{h, \tau}^n|$ and $\text{Arg}(\Psi_{h, \tau}^n)$ at $t = 0, 1.5, 2.9,  5.9$ with $N_0 = 2$ (Example 5.4).}
    \label{figure:8}
\end{figure}
\begin{figure}
    \centering
    \includegraphics[width=1\linewidth]{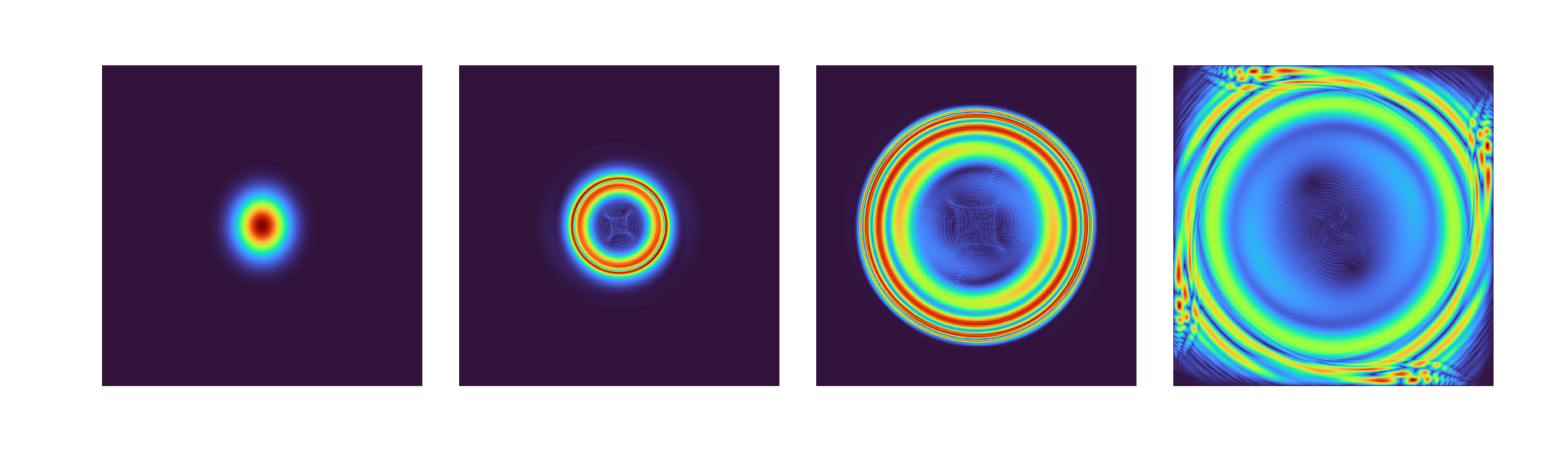}
    \includegraphics[width=1\linewidth]{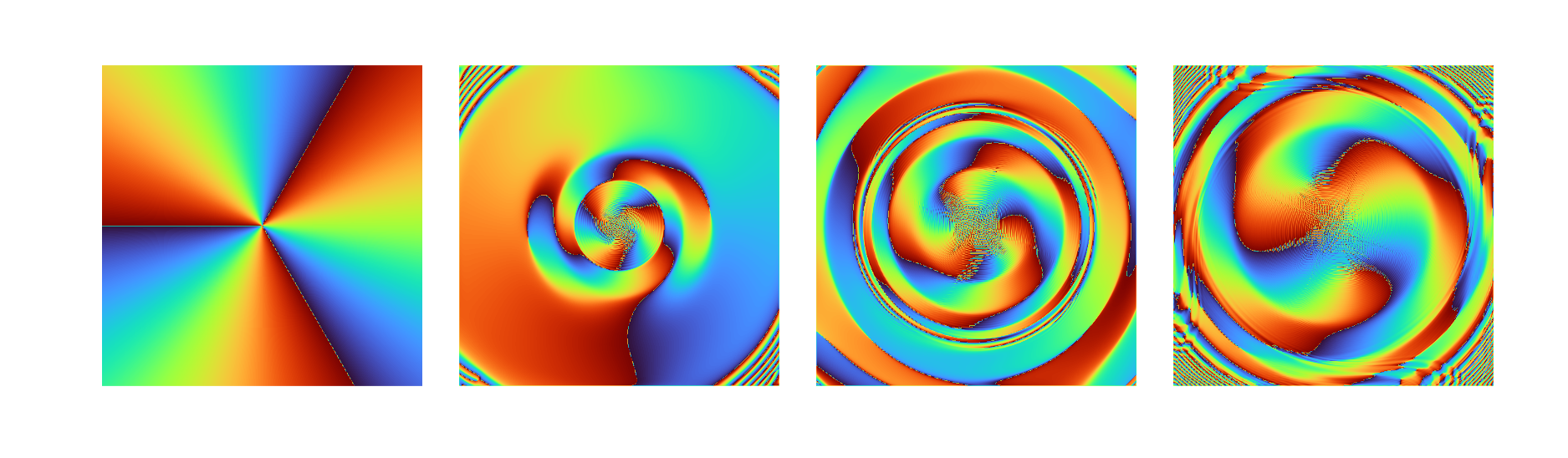}
    \caption{Contour plot of $|\Psi_{h, \tau}^n|$ and $\text{Arg}(\Psi_{h, \tau}^n)$ at $t = 0, 1.7, 3.5,  6.1$ with $N_0 = 3$ (Example 5.4).}
    \label{figure:9}
\end{figure}
\begin{example}[\textbf{Relativistic effect on the bound state}]
To obtain the bound state solutions, we employ the normalized GFLM presented in Appendix A and select the initial data as:
 \[
    z_+^{0} = \frac{ (1-\Omega) \, \phi_g(x,y) + \Omega \, \phi_g(x,y) (x - \textup{i} y) }
    {\left\| (1-\Omega) \, \phi_g(x,y) + \Omega \, \phi_g(x,y) (x - \textup{i} y) \right\|_{L^2}},
    \quad
    z_-^{0} = \overline{z_+^{0}},
\]
with
\[
    \phi_g(x,y) = \frac{1}{\sqrt{\pi}} e^{-(x^2 + y^2)/2}.
\]
The parameters and the external potential are chosen as
\[
 \lambda=50, \qquad\Omega=0.9, \qquad \textup{tol}=1e-8,\qquad V=\frac{x^2 + y^2}{2}.
\]
Figure \ref{figure:10} presents the ground state solutions $z_+$ and $z_-$, showing that during the evolution both components are approximately evenly distributed across two regions. To examine the influence of relativistic effects on vortex structures in the ground state of the coupled RNLS equations, we first compute $z_+$ and $z_-$ using the normalized GFLM described in Appendix~A. The initial data for the RKG equation are then taken as
\begin{align*}
    \psi_0 = z_+ + \overline{z_-}, \qquad \psi_1 = \textup{i}\left(z_+ - \overline{z_-}\right).
\end{align*}
\begin{figure}
    \centering
    \includegraphics[width=1\linewidth]{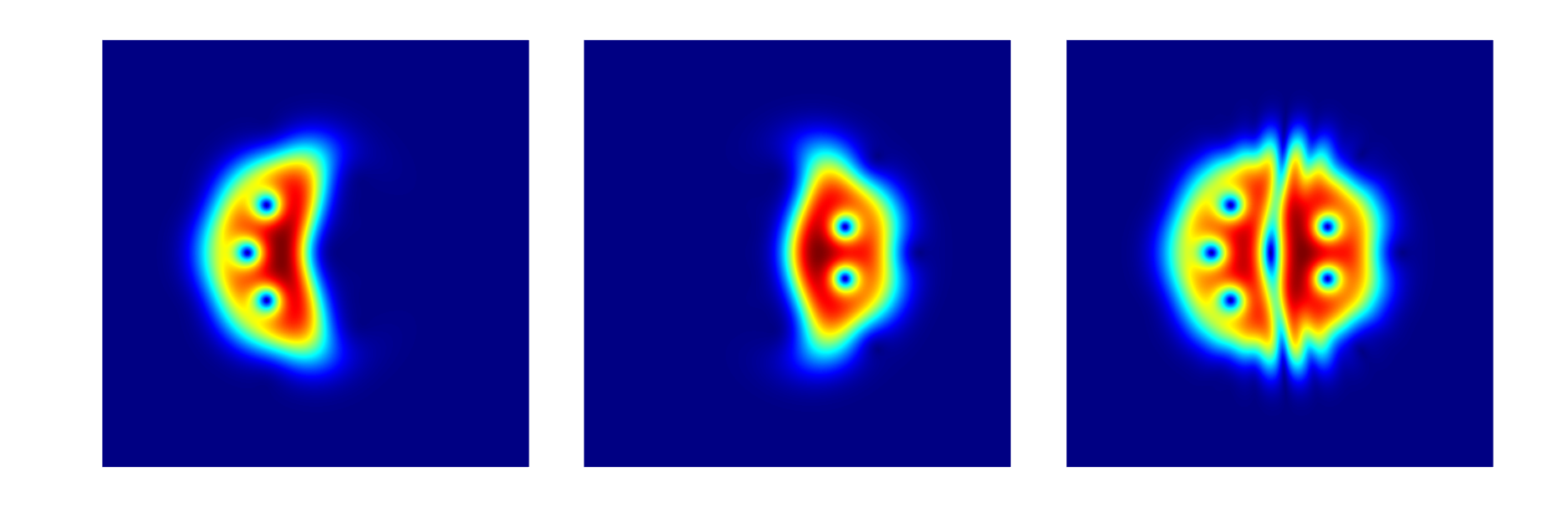}
   \includegraphics[width=1\linewidth]{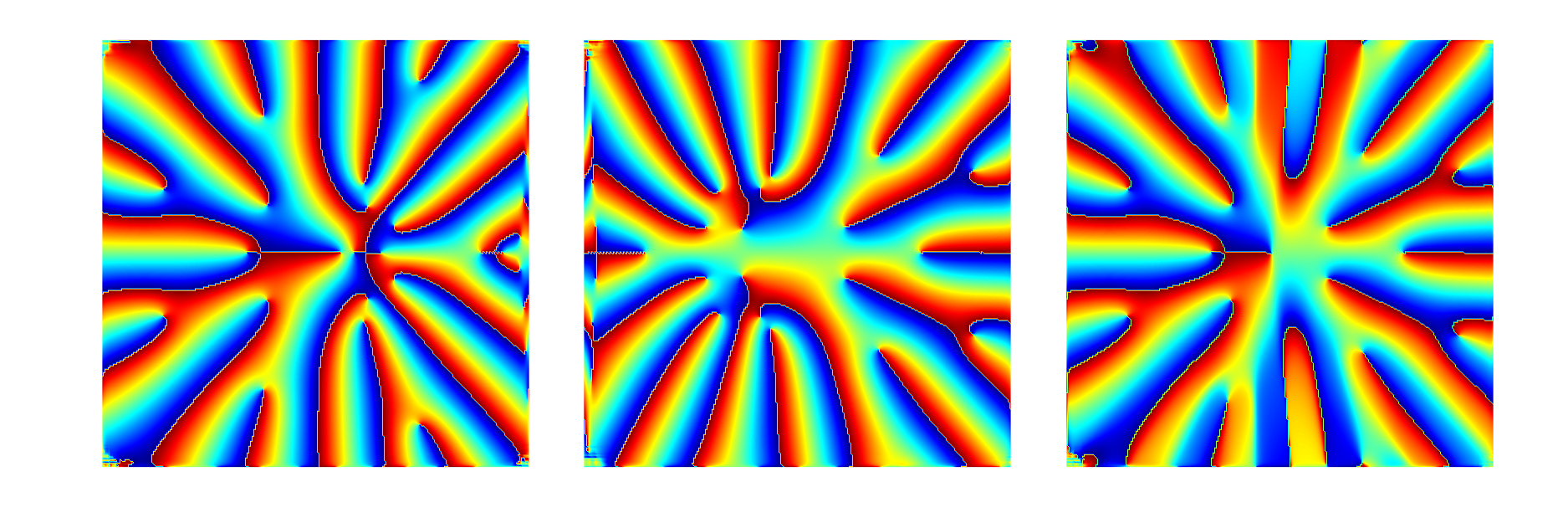}

    \caption{Contour plots of the bound state of the particle $z_{+}$ (top left), 
    	the antiparticle $z_{-}$ (top middle), and the quantization $z_{+} + \overline{z_{-}}$ (top right) (Example~5.5). 
    	The figures below show the corresponding phase (argument) distributions. }
    \label{figure:10}
\end{figure}

The RKG equation is subsequently solved using a conforming FEM for various values of $\epsilon$. Figures \ref{figure:11}-\ref{figure:12} show the numerical solution $\Psi_{h, \tau}^n$ for $\epsilon = 1$ and $\epsilon = 1/8$, respectively, while Figure \ref{figure:13} corresponds to $\epsilon = 1/32$. The results indicate that for $\epsilon = 1$, relativistic effects substantially modify the vortex dynamics during the time evolution. As $\epsilon$ decreases, these effects gradually diminish, and for $\epsilon = 1/32$, relativistic contributions become negligible, yielding solutions that closely match the ground states of the corresponding coupled RNLS equations.

\end{example}
\begin{figure}
    \centering
    \includegraphics[width=1\linewidth]{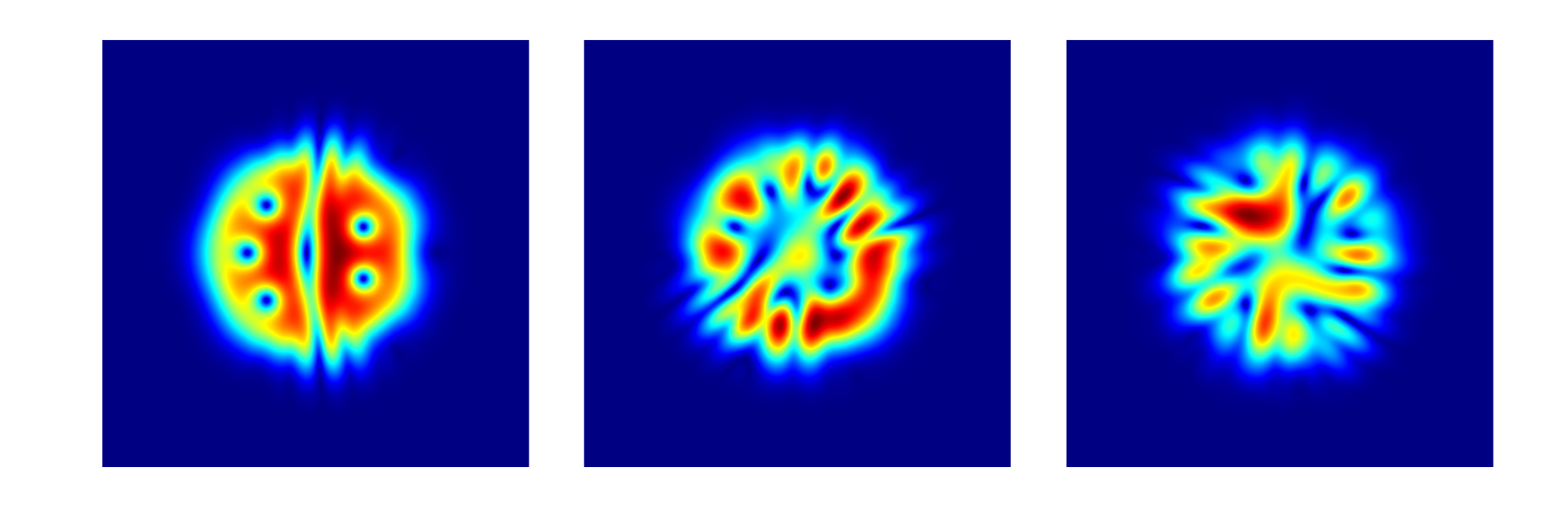}
    \includegraphics[width=1\linewidth]{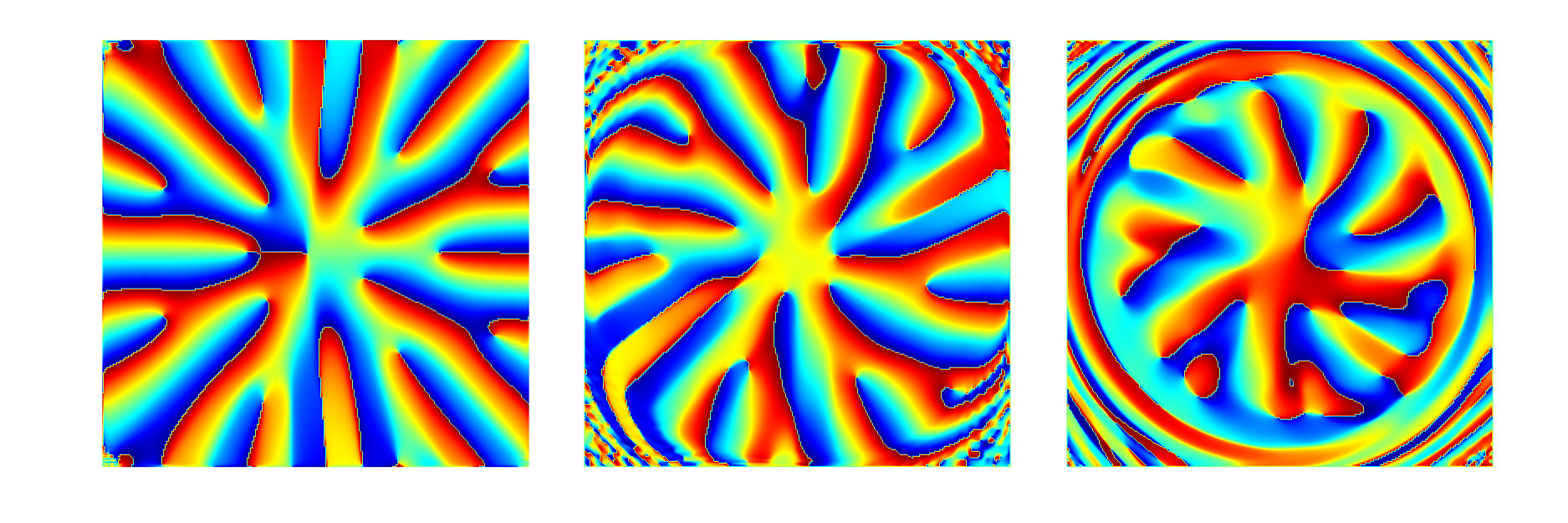}
    \caption{Contour plot of $|\Psi_{h, \tau}^n|$ and Arg$(\Psi_{h, \tau}^n)$ under $\epsilon=1$, at $t=0, 3, 6$ (Example 5.5).}
    \label{figure:11}
\end{figure}
\begin{figure}
    \centering
    \includegraphics[width=1\linewidth]{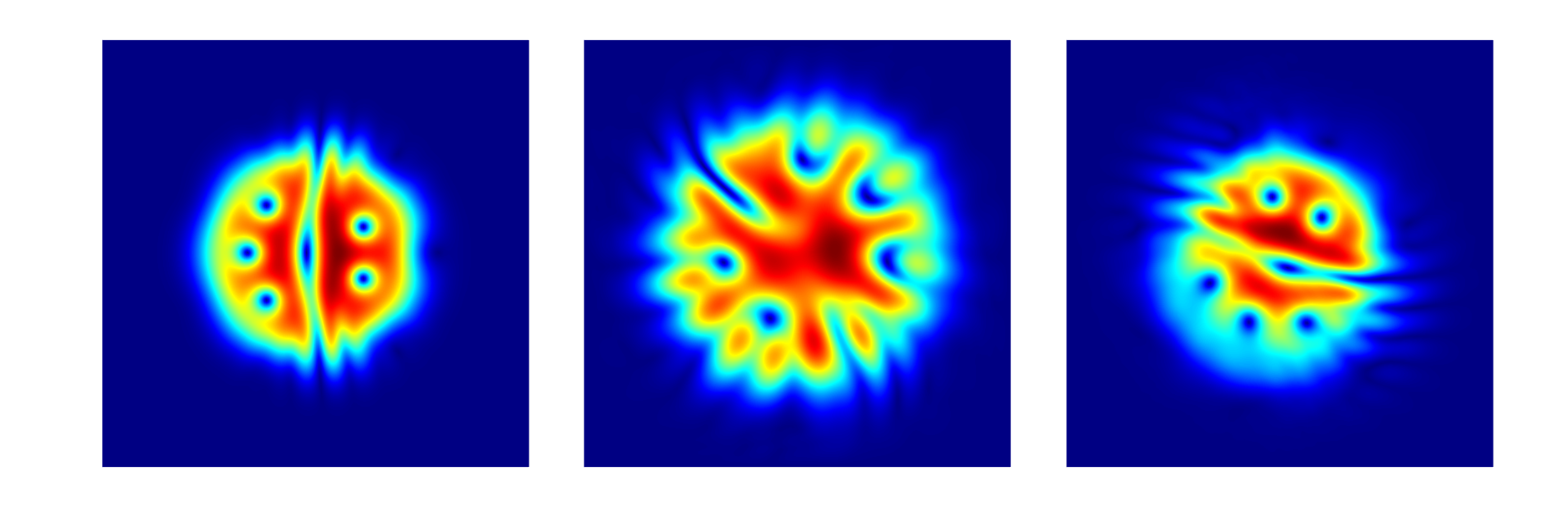}
    \includegraphics[width=1\linewidth]{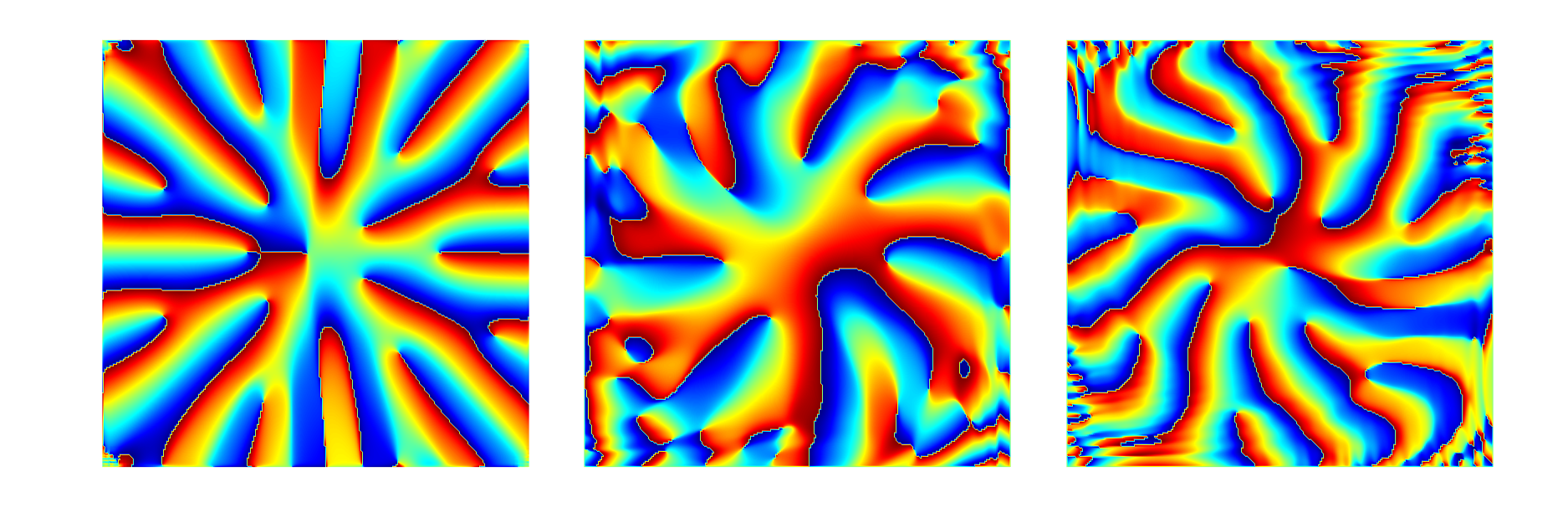}
    \caption{Contour plot of $|\Psi_{h, \tau}^n|$ and Arg$(\Psi_{h, \tau}^n)$ under $\epsilon=1/8$, at $t=0, 3, 6$ (Example 5.5).}
    \label{figure:12}
\end{figure}
\begin{figure}
    \centering
    \includegraphics[width=1\linewidth]{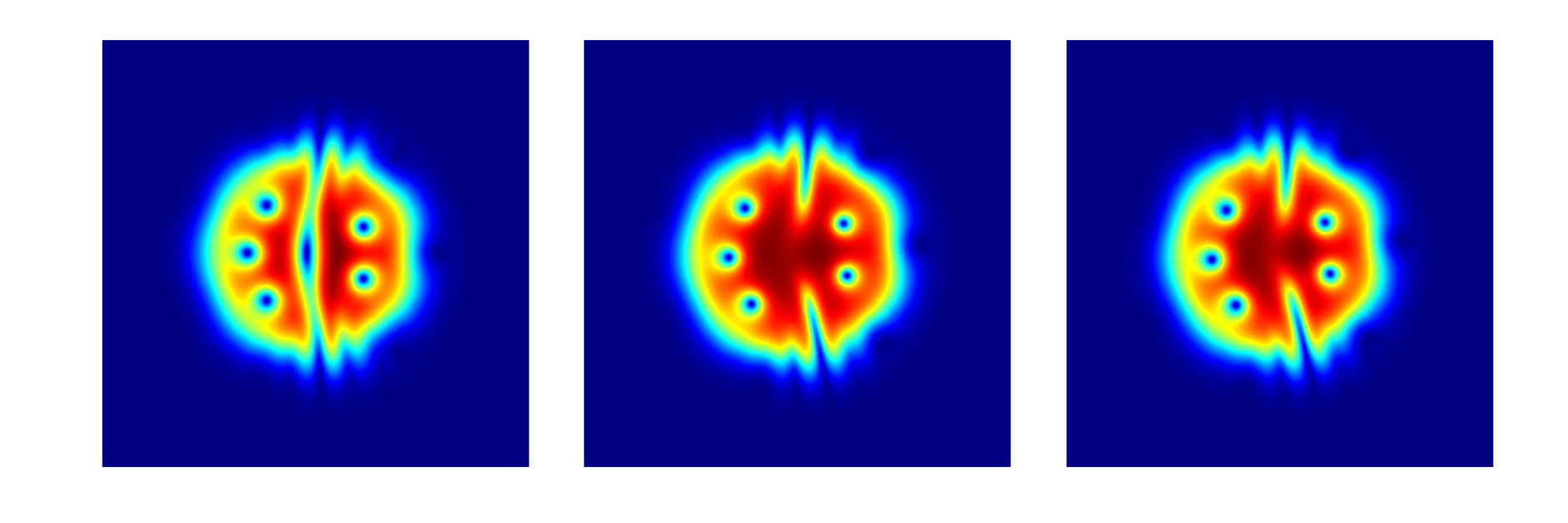}
    \includegraphics[width=1\linewidth]{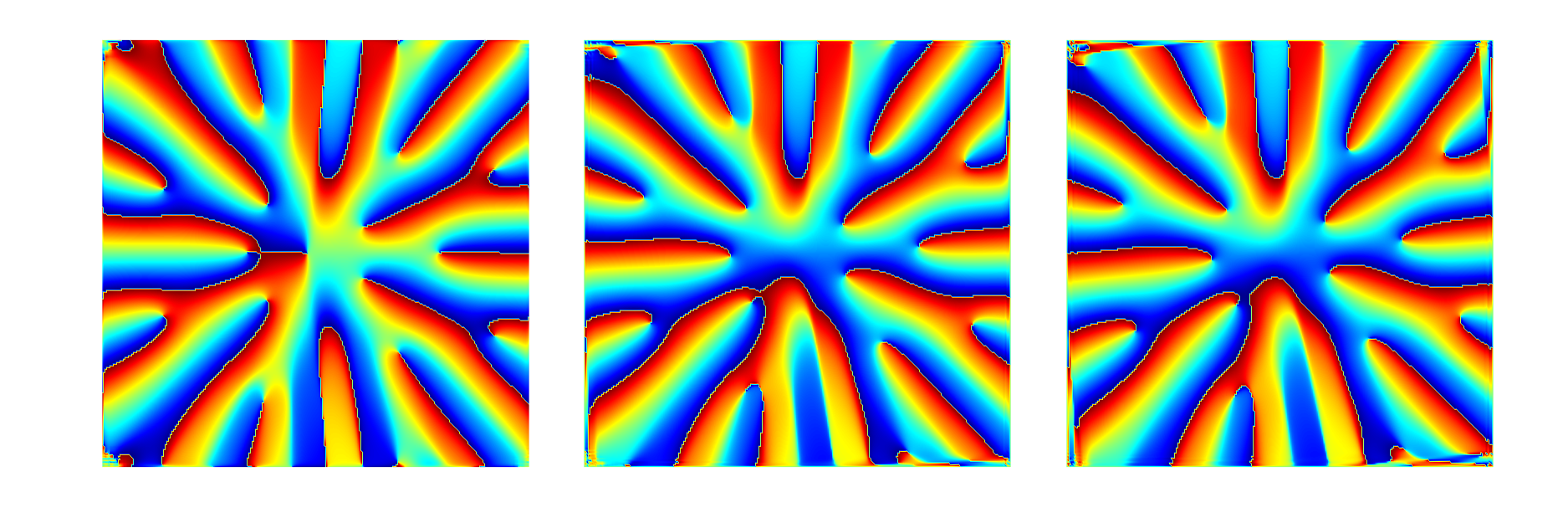}
    \caption{Contour plot of $|\Psi_{h, \tau}^n|$ and Arg$(\Psi_{h, \tau}^n)$ under $\epsilon=1/32$, at $t=0, 3, 6$ (Example 5.5).}
    \label{figure:13}
\end{figure}
\begin{example} [\textbf{Interaction of vortex pairs}]
    In the final example, we examine the interaction dynamics of vortex pairs in the nonrelativistic regime. Following the approach in \cite{Mauser2020on}, we initialize two well-separated vortex pairs as follow
    \[
    \psi_0\left(x,y\right) = \psi_1\left(x,y\right) = 
    \left(\left(x-c_0\right)+\textup{i}y\right) \left(\left(x+c_0\right)+\textup{i}y\right)
    \left(x+\textup{i}\left(y-c_0\right)\right) \left(x+\textup{i}\left(y+c_0\right)\right)
    e^{-\left(x^2+y^2\right)/2},
    \]
    where $c_0$ determines the initial vortex locations.
    The parameters and external potential are chosen as
    \[
    \lambda = 10, \qquad\Omega = 0.5,\qquad c_0 = 1.32,\qquad V=\frac{x^2 + y^2}{2}.
    \]
    Our objective is to investigate the dynamical evolution of two vortex pairs under varying values of $\epsilon$. The results are presented in Figures \ref{figure:14}-\ref{figure:16}. It can be observed that, as $\epsilon$ decreases in the nonrelativistic regime, the interactions between vortices become stronger, and the configurations of the solution gradually develop richer and more intricate structures.

\end{example}
\begin{figure}
    \centering
    \includegraphics[width=1\linewidth]{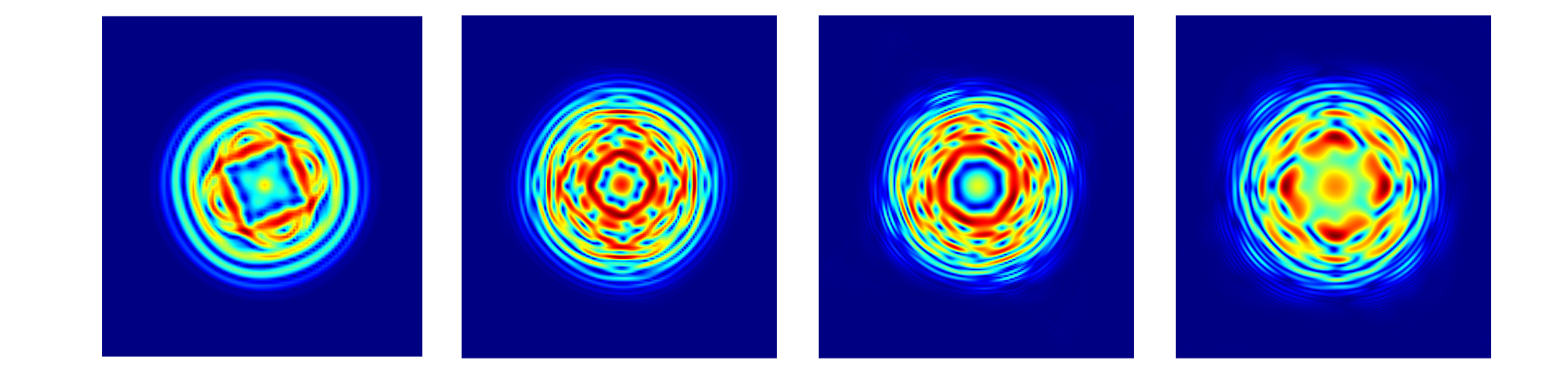}
    \includegraphics[width=1\linewidth]{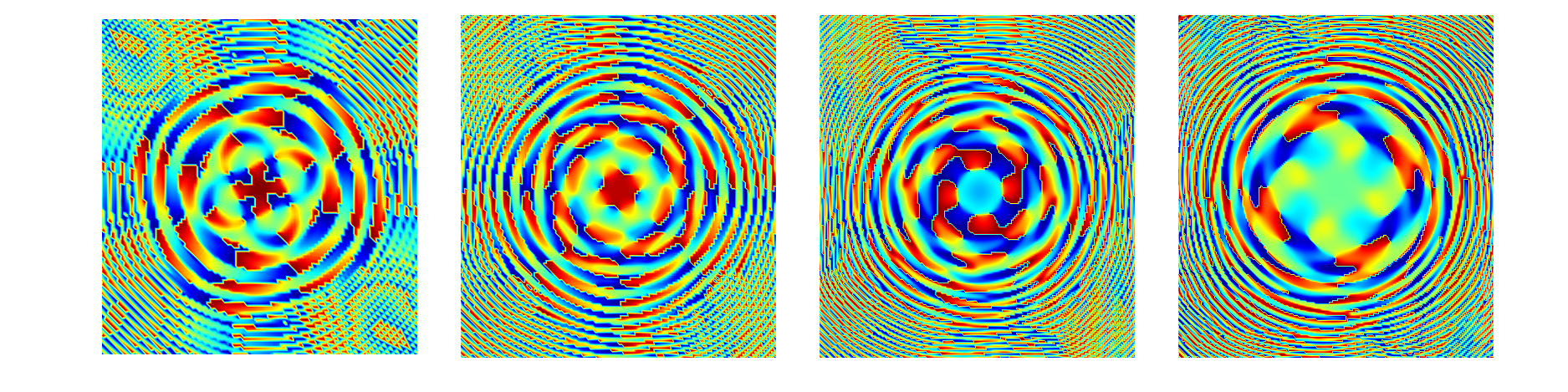}
    \caption{Contour plot of $|\Psi_{h, \tau}^n|$ and Arg$(\Psi_{h, \tau}^n)$ under $\epsilon=1/4$, at $t=0.79, 1.58, 2.37, 3.15$ (Example 5.6).}
    \label{figure:14}
\end{figure}
\begin{figure}
    \centering
    \includegraphics[width=1\linewidth]{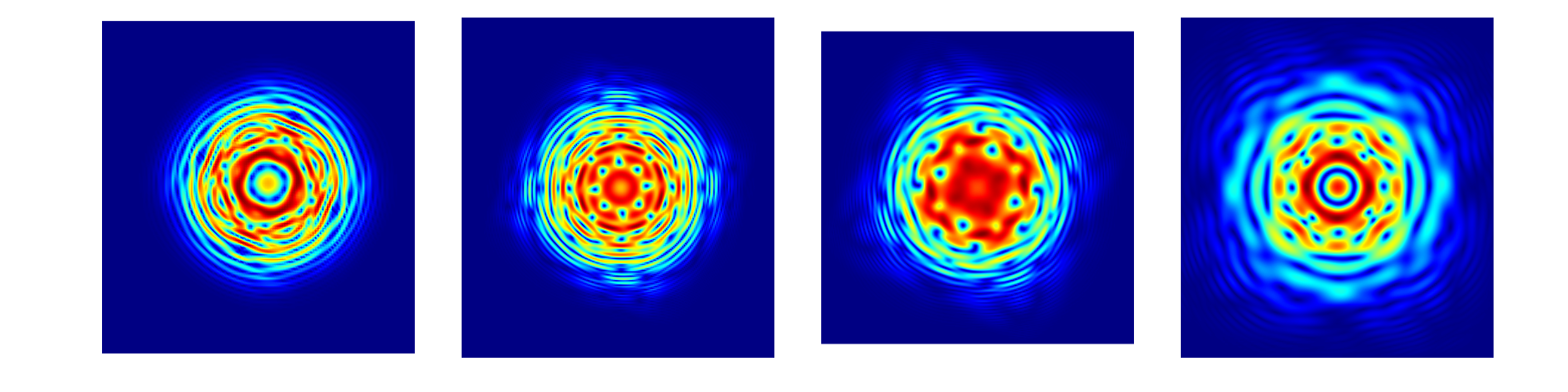}
    \includegraphics[width=1\linewidth]{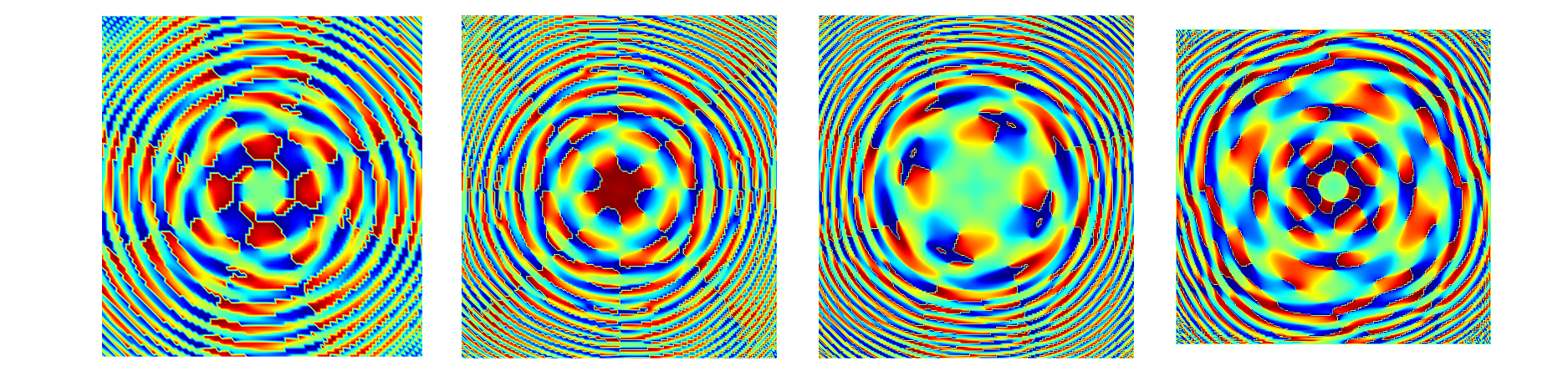}
    \caption{Contour plot of $|\Psi_{h, \tau}^n|$ and Arg$(\Psi_{h, \tau}^n)$ under $\epsilon=1/8$, at $t=0.79, 1.58, 2.37, 3.15$ (Example 5.6).}
    \label{figure:15}
\end{figure}
\begin{figure}
    \centering
    \includegraphics[width=1\linewidth]{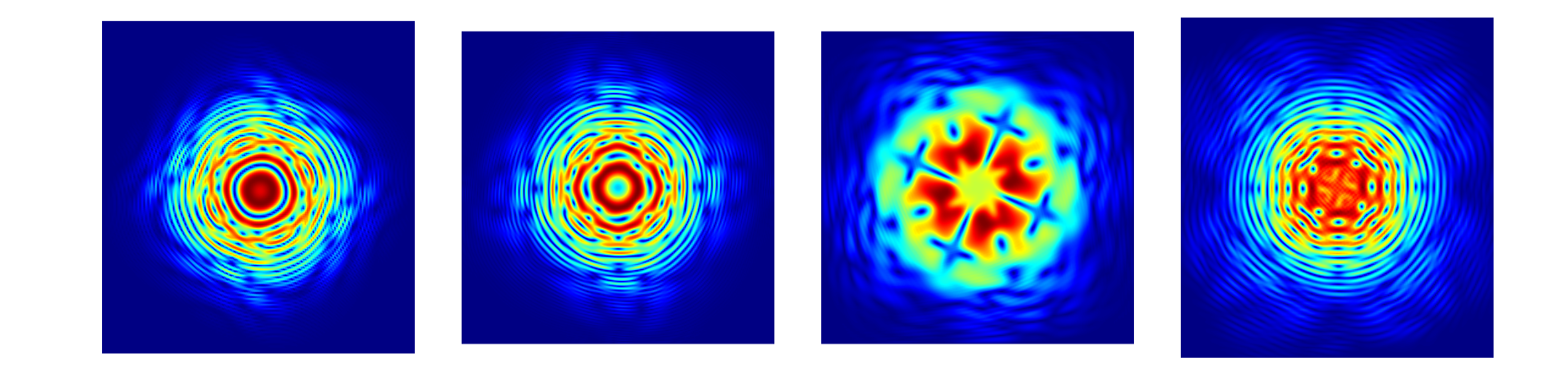}
    \includegraphics[width=1\linewidth]{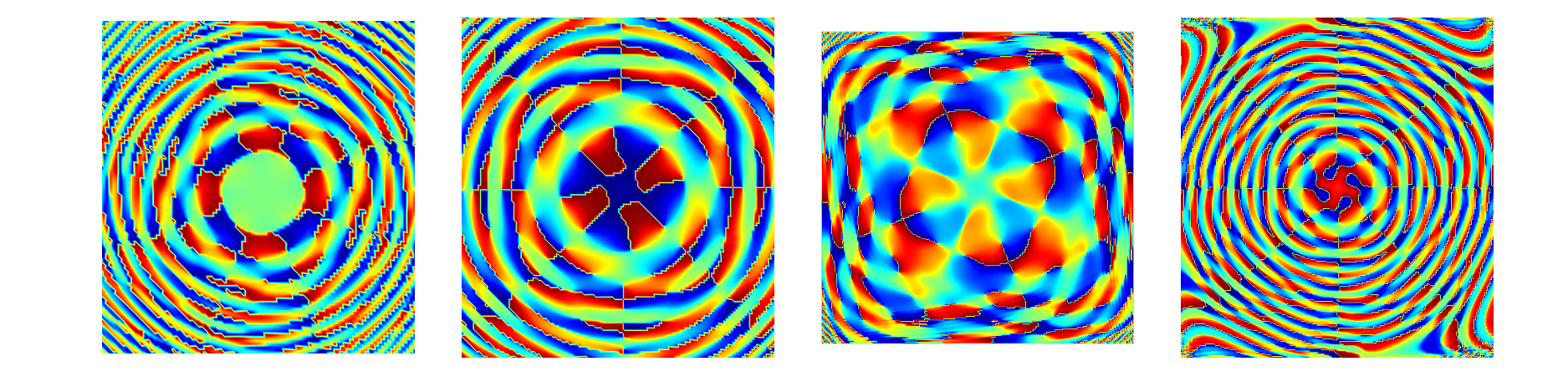}
    \caption{Contour plot of $|\Psi_{h, \tau}^n|$ and Arg$(\Psi_{h, \tau}^n)$ under $\epsilon=1/32$, at $t=0.79, 1.58, 2.37, 3.15$ (Example 5.6).}
    \label{figure:16}
\end{figure}

\section{Conclusions}\label{sec6}
In this paper, we develop and analyze conforming and nonconforming Galerkin FEMs for the RKG equation. With a conservation-adjusting technique, we construct consistent structure-preserving schemes that conserve both energy and charge. A rigorous convergence analysis establishes unconditional optimality and high-order accuracy. Extensive numerical tests confirm the theory and demonstrate the accuracy, efficiency, and robustness of the schemes in both relativistic and nonrelativistic regimes. Simulations of vortex dynamics illustrate the generation of vortices, relativistic effect on bound state, and vortex-pair interactions. The proposed methods provide a reliable framework for further studies of rotating relativistic quantum models.

\bibliographystyle{elsarticle-num}

\bibliography{thebibKG}

\appendix
\section{Normalized GFLM}
The gradient flow with discrete normalization \cite{bao2004computing}, also known as the imaginary time evolution method, has been widely recognized for its effectiveness in computing ground states and has seen extensive application in various BECs. Building upon this framework, Liu et al. \cite{liu2021normalized} introduced the normalized GFLM, which has proven highly efficient in studies of spinor BECs. In this work, we extend the normalized GFLM to compute the bound state solutions of the coupled RNLS equations.

The bound state is defined as the minimizer of the energy functional associated with the coupled RNLS equations, as follows.
\begin{align} \label{mini_energy_func}
     {E_{[z_+, z_-]}(t)} =& \frac{1}{2} \int_{\mathbb{R}^2} \bigg[|\nabla z_+|^2 + |\nabla z_-|^2 + V(x,y)\big( |z_+|^2 + |z_-|^2 \big)  + \frac{\lambda}{2} \Big( |z_+|^4 + |z_-|^4 + 2 |z_+|^2 |z_-|^2 \Big) \bigg] \, \text{d}x\text{d}y \nn \\
    &- \Omega \int_{\mathbb{R}^2} \mathrm{Re} \big( \overline{z_+} L_z z_+ + \overline{z_-} L_z z_- \big) \, \text{d}x\text{d}y,
\end{align}
 with the constraint 
 \begin{align} \label{eqn:constraint}
     \int_{\mathbb{R}^2} z_{+} \overline{z_{+}}\text{d}x\text{d}y\, =\alpha,\qquad
    \int_{\mathbb{R}^2} z_{-} \overline{z_{-}} \text{d}x\text{d}y\, = 1-\alpha,  
 \end{align}
 where  $\alpha\in\left[0, 1\right]$. 
To compute the ground state, we consider the Euler--Lagrange equations associated with the energy \eqref{mini_energy_func} under constraint \eqref{eqn:constraint}:
\begin{align} \label{eqn:E_L_eqn}
\begin{cases}
    & u_{+} z_{+} = -\frac{1}{2}\Delta z_{+} + \frac{1}{2} V z_{+} + \frac{\lambda}{2}\left(|z_{+}|^2 + 2|z_{-}|^2\right) z_{+} - \Omega L_z z_{+}, \\
    & u_{-} z_{-} = -\frac{1}{2}\Delta z_{-} + \frac{1}{2} V z_{-} + \frac{\lambda}{2}\left(|z_{-}|^2 + 2|z_{+}|^2\right) z_{-} - \Omega L_z z_{-},
\end{cases}
\end{align}
where \(u_{\pm}\) are the chemical potentials:
\begin{align} \label{eqn:potential}
\begin{cases}
    & u_{+} = \int_{\mathbb{R}^2} \left[ \frac{1}{2}|\nabla z_{+}|^2 + \frac{1}{2} V |z_{+}|^2 + \frac{\lambda}{2}\left(|z_{+}|^4 + 2|z_{-}|^2 |z_{+}|^2\right) - \Omega \overline{z_{+}} L_z z_{+} \right] \text{d}\mathbf{x}, \\
    & u_{-} = \int_{\mathbb{R}^2} \left[ \frac{1}{2}|\nabla z_{-}|^2 + \frac{1}{2} V |z_{-}|^2 + \frac{\lambda}{2}\left(|z_{-}|^4 + 2|z_{+}|^2 |z_{-}|^2\right) - \Omega \overline{z_{-}} L_z z_{-} \right] \text{d}\mathbf{x}.
\end{cases}
\end{align}
In what follows, we adopt the normalized GFLM to compute the bound state solutions.

We consider a uniform time step $\tau$ and set the temporal mesh points as $t_n = n \tau$.
We first consider the continuous GFLM over $t \in [t_n, t_{n+1})$, which reads
\begin{align} \label{eqn:CGFLM1}
\begin{cases}
    & \frac{\partial}{\partial t} z_{+}(\mathbf{x},t) = \Big[\frac{1}{2}\Delta - \frac{1}{2} V - \frac{\lambda}{2}(|z_{+}|^2 + 2|z_{-}|^2) + \Omega L_z \Big] z_{+}(\mathbf{x},t) + u_{+} z_{+}(\mathbf{x}, t_n), \\
    & \frac{\partial}{\partial t} z_{-}(\mathbf{x},t) = \Big[\frac{1}{2}\Delta - \frac{1}{2} V - \frac{\lambda}{2}(|z_{-}|^2 + 2|z_{+}|^2) + \Omega L_z \Big] z_{-}(\mathbf{x},t) + u_{-} z_{-}(\mathbf{x}, t_n),
\end{cases}
\end{align}
with the discrete normalization at $t = t_{n+1}$ given by
\begin{align} \label{eqn:CGFLM2}
\begin{cases}
    & z_{+}(\mathbf{x}, t_{n+1}) = z_{+}(\mathbf{x}, t_{n+1}^+) = \frac{\sqrt{\alpha}\, z_{+}(\mathbf{x}, t_{n+1}^-)}{\|z_{+}(\mathbf{x}, t_{n+1}^-)\|_{L^2}}, \\
    & z_{-}(\mathbf{x}, t_{n+1}) = z_{-}(\mathbf{x}, t_{n+1}^+) = \frac{\sqrt{1-\alpha}\, z_{-}(\mathbf{x}, t_{n+1}^-)}{\|z_{-}(\mathbf{x}, t_{n+1}^-)\|_{L^2}},
\end{cases}
\end{align}
where $u_{\pm}$ denote the chemical potentials as in \eqref{eqn:potential}.

We then discretize \eqref{eqn:CGFLM1} and \eqref{eqn:CGFLM2} using a backward/forward Euler scheme:
\begin{align} \label{eqn:GFLM}
\begin{cases}
    &\frac{z_{+}^{(1)} - z_{+}^n}{\tau} = \left(\frac{1}{2}\Delta - \beta_{+}^n\right) z_{+}^{(1)} 
    + \Big[\beta_{+}^n - V + \Omega L_{z} - \lambda\left(|z_{+}^n|^2 + 2 |z_{-}^n|^2\right)\Big] z_{+}^n 
    + u_{+}^n z_{+}^n, \\
    &\frac{z_{-}^{(1)} - z_{-}^n}{\tau} = \left(\frac{1}{2}\Delta - \beta_{-}^n\right) z_{-}^{(1)} 
    + \Big[\beta_{-}^n - V + \Omega L_{z} - \lambda\left(|z_{-}^n|^2 + 2 |z_{+}^n|^2\right)\Big] z_{-}^n 
    + u_{-}^n z_{-}^n,
\end{cases}
\end{align}
followed by the discrete normalization
\begin{align*}
     z_{+}^{n+1} = \frac{\sqrt{\alpha}\, z_{+}^{(1)}}{\|z_{+}^{(1)}\|_{L^2}^2}, \qquad
     z_{-}^{n+1} = \frac{\sqrt{1-\alpha}\, z_{-}^{(1)}}{\|z_{-}^{(1)}\|_{L^2}^2}.
\end{align*}
Here, the stabilization parameters are defined as 
\(\beta_{+}^n = \frac{1}{2}\left(\max G_{+}^n + \min G_{+}^n\right)\) with \(G_{+}^n = V + \lambda\left(|z_{+}^n|^2 + 2 |z_{-}^n|^2\right)\), 
and \(\beta_{-}^n = \frac{1}{2}\left(\max G_{-}^n + \min G_{-}^n\right)\) with \(G_{-}^n = V + \lambda\left(|z_{-}^n|^2 + 2 |z_{+}^n|^2\right)\).
Starting from given initial data \(z_{\pm}^0(\mathbf{x})\), the iteration proceeds until convergence, which is defined by
\[
\bigg|E_{[z_+^{n+1}, z_-^{n+1}]} - E_{[z_+^{n}, z_-^{n}]}\bigg| \leq \text{tol},
\]
where \(\text{tol}\) is a suitably small positive constant chosen to ensure numerical accuracy.

\section{Proof of Lemma \ref{lem:timediscrete_truncate_existence}}
\renewcommand{\thesection}{\Alph{section}}
~~To establish the existence of the truncated time-discrete method, we utilize the Brouwer fixed point theorem \cite{brouwer1911abbildung} and the Vitali convergence theorem \cite{leoni2017first}.
\begin{lemma}\label{lem:Brouwer} (Brouwer Fixed Point Theorem)
Let \( \overline{S_1(0)} = \left\{\boldsymbol{\alpha} \in \mathbb{C}^N \mid |\boldsymbol{\alpha}|\leq 1\right\} \) be the closed unit disk in \( \mathbb{C}^N \). Suppose \( g: \mathbb{C}^N \to \mathbb{C}^N \) is a continuous function satisfying the condition  
\[
\operatorname{Re}\left(g(\boldsymbol{\alpha}), \boldsymbol{\alpha}\right) \geq 0, \quad \forall \boldsymbol{\alpha} \in \partial S_1(0).
\]
Then, there exists at least one point \( \boldsymbol{\alpha}_0 \in \overline{S_1(0)} \) such that  
\[
g(\boldsymbol{\alpha}_0) = 0.
\]
\end{lemma}

\begin{lemma}\label{lem:Vit}
	(Vitali Convergence Theorem)
	A sequence $(g_k)_{k\in \mathbb N}\subset L^2(U)$ converges strongly to a function $g\in L^2(U)$ if and only if 
	\begin{itemize}
		\item [(i)] $(g_k)_{k\in \mathbb N}$	 converges locally  to $g$ in measure;
		\item [(ii)] $(g_k)_{k\in \mathbb N}$ is $2$-equi-integrable, which implies that for any $\epsilon>0$, there exists a positive constant $\delta_\epsilon$ such that $\|g_k\|_{L^2(S)}<\epsilon$, for any measurable subset $S\subset U$ with the measure 
		$\mu(S)\leq \delta_\epsilon$.
	\end{itemize}
	
\end{lemma}
    
    We write \eqref{eqn:timediscrete_truncate} into its equivalent form: 
	\begin{align}\label{lem:timediscrete_truncate_existence_pf1}
\widetilde{\Psi}_\tau^{T, n}-\Psi_\tau^{T, n}+f(\widetilde{\Psi}_\tau^{T, n})-\textup{i}\Omega\tau L_z \widetilde{\Psi}_\tau^{T, n}-\frac{\tau^2\Omega^2}{2}L_z^2\widetilde{\Psi}_\tau^{T, n}=	
-\textup{i}\Omega\tau L_z {\Psi}_\tau^{T, n-1},
	\end{align}
	where
	\begin{align*}
		f(\widetilde{\Psi}_\tau^{T, n}) :=\left(\frac{\tau^2}{2\epsilon^4}+\frac{\tau^2}{2\epsilon^2}V\right)\widetilde{\Psi}_\tau^{T, n}-\frac{\tau^2}{2\epsilon^2}\Delta \widetilde{\Psi}_\tau^{T, n}
		+\frac{\lambda\tau^2}{4\epsilon^2}\left[\mu_A\left(\left|2 \widetilde{\Psi}_\tau^{T, n}-{\Psi}_\tau^{T, n-1}\right|^2\right)
		+\mu_A\left(\left|\Psi^{T, n-1}_\tau\right|^2\right)\right]\widetilde{\Psi}_\tau^{T, n}.
	\end{align*}
Hence, we only need to demonstrate the existence of the solution to 
	\begin{align}\label{lem:timediscrete_truncate_existence_pf2}
		F(\omega)=\omega -\Psi_\tau^{T, n}+f(\omega)-\textup{i}\Omega\tau L_z \omega-\frac{\tau^2\Omega^2}{2}L_z^2\omega+\text{i}\Omega\tau L_z {\Psi}_\tau^{T, n-1}=0. 
	\end{align}
	The following proof will be divided into two steps.\\
	\textbf{Step 1}. We first show the existence in a finite-dimensional subspace
$\chi_N:=\left\{\phi_m,~m\in \mathbb N\right\}$, which represents a countable basis of $H_0^1(U)$. For given \(\Psi_\tau^{T, n} \in H_0^1(U)\) and \(\Psi_\tau^{T, n-1} \in H_0^1(U)\), our next goal is to prove the existence of a solution \(X_N \in \chi_N\) to the equation:  
\begin{align}  \label{lem:timediscrete_truncate_existence_pf3}
	F(X_N) = X_N  - \Psi_\tau^{T, n} + f(X_N ) - \text{i}\Omega \tau L_z X_N -\frac{\tau^2\Omega^2}{2}L_z^2X_N + \text{i}\Omega\tau L_z \Psi_\tau^{T, n-1} = 0.  
\end{align}
By taking the inner product of \eqref{lem:timediscrete_truncate_existence_pf3} with $X_N$, and extracting the real part of the equation, we have 
\begin{align}  \label{lem:timediscrete_truncate_existence_pf4}
  \text{Re}\left(F(X_N), X_N\right)
  =&\left\|X_N\right\|_{L^2}^2-\text{Re}\left(\Psi_\tau^{T, n}, X_N\right)+ 
  \text{Re}\left(f(X_N ), X_N\right) 
  - \Omega \tau \text{Re}\left(\text{i}L_z X_N, X_N\right)  \nn\\
  &
  -\frac{\tau^2\Omega^2}{2}\left\|L_zX_N\right\|_{L^2}^2
  + 
  \Omega\tau \text{Re}\left(\text{i}L_z\Psi_\tau^{T, n-1}, X_N\right). 
\end{align}
For the term $\text{Re}\left(f(X_N ), X_N\right)$, from \eqref{eqn:equi2}, it is obvious that 
\begin{align}\label{lem:timediscrete_truncate_existence_pf5}
	\text{Re}\left(f(X_N ), X_N\right)
	&=\left(\frac{\tau^2}{2\epsilon^4}+\frac{\tau^2}{2\epsilon^2}V\right)
	\left\|X_N\right\|_{L^2}^2
	+\frac{\tau^2}{2\epsilon^2}\left\|\nabla X_N\right\|_{L^2}^2
	+\frac{\lambda\tau^2}{4\epsilon^2}\left(\left[\mu_A\left(\left|2 \widetilde{\Psi}_\tau^{T, n}-{\Psi}_\tau^{T, n-1}\right|^2\right)
	+\mu_A\left(\left|\Psi^{T, n-1}_\tau\right|^2\right)\right]X_N, X_N\right)\nn\\
	&\geq  \frac{\tau^2}{2\epsilon^2}\left\|\nabla X_N\right\|_{L^2}^2. 
\end{align}
In addition, we have 
\begin{align}\label{lem:timediscrete_truncate_existence_pf6}
	- \Omega \tau \text{Re}\left(\text{i}L_z X_N, X_N\right) = \Omega \tau \text{Re}\left(\left[x\partial_y-y\partial_x\right] X_N, X_N\right)
	 =\Omega \tau \frac{\left(\left[x\partial_y-y\partial_x\right] X_N, X_N\right)+\left(X_N, \left[x\partial_y-y\partial_x\right] X_N\right)}{2}=0. 
\end{align}
By using \eqref{eqn:equi1}, there holds 
\begin{align}\label{lem:timediscrete_truncate_existence_pf7}
\frac{\tau^2\Omega^2}{2}\left\|L_zX_N\right\|_{L^2}^2 \leq  
	\frac{\tau^2\Omega^2}{2C_{U_1}^2}\left\|\nabla X_N\right\|_{L^2}^2. 
\end{align}
Furthermore, there holds 
\begin{align}\label{lem:timediscrete_truncate_existence_pf8}
\Omega\tau \text{Re}\left(\text{i}L_z\Psi_\tau^{T, n-1}, X_N\right)\leq \Omega\tau\left\|L_z\Psi_\tau^{T, n-1}\right\|_{L^2}\left\|X_N\right\|_{L^2}
	\leq \frac{\Omega\tau}{C_{U_1}}\left\|\nabla \Psi_\tau^{T, n-1}\right\|_{L^2}\left\|X_N\right\|_{L^2}. 
\end{align}
Substituting \eqref{lem:timediscrete_truncate_existence_pf6}-\eqref{lem:timediscrete_truncate_existence_pf8} into \eqref{lem:timediscrete_truncate_existence_pf5} gives that 
\begin{align}  \label{lem:timediscrete_truncate_existence_pf9}
	\text{Re}\left(F(X_N), X_N\right)
	\geq \left[\left\|X_N\right\|_{L^2}-\left\|\Psi_\tau^{T, n}\right\|_{L^2}-\frac{\Omega\tau}{C_{U_1}}\left\|\nabla \Psi_\tau^{T, n-1}\right\|_{L^2}\right]\left\|X_N\right\|_{L^2} +\left[\frac{\tau^2}{2\epsilon^2}
	-\frac{\tau^2\Omega^2}{2C_{U_1}^2}
	\right]\left\|\nabla X_N\right\|_{L^2}^2. 
\end{align}
We assume that $\epsilon$ is selected sufficiently small, such that 
\begin{align}\label{lem:timediscrete_truncate_existence_pf10}
	\frac{\tau^2}{2\epsilon^2}
	-\frac{\tau^2\Omega^2}{2C_{U_1}^2}\geq 0. 
\end{align}
Then, we follow a similar proof process as in Theorem \ref{thm:boundedness} to establish the boundedness of \(\Psi_\tau^{T, n}\) in the \(H^1\)-norm for \(0 \leq n \leq N\). As a result, we obtain the inequality
\begin{align}  \label{lem:timediscrete_truncate_existence_pf11}
\left\|\Psi_\tau^{T, n}\right\|_{L^2} + \frac{\Omega \tau}{C_{U_1}} \left\|\nabla \Psi_\tau^{T, n-1}\right\|_{L^2} \leq M,
\end{align}
where \(M > 0\) is a constant independent of \(\tau\) and \(h\). We then choose a sufficiently large value for \(\|X_N\|_{L^2}\) in \eqref{lem:timediscrete_truncate_existence_pf9} such that \(\|X_N\|_{L^2} \geq M\). By applying Lemma \ref{lem:Brouwer}, we conclude the existence of a solution to \eqref{lem:timediscrete_truncate_existence_pf3}. 
Therefore, $Z_N=2X_N-\Psi_\tau^{T, n-1}$ exists. \\
	\textbf{Step 2}. This step is to prove the existence of weak solution to \eqref{eqn:timediscrete_truncate} in $H_0^1(U)$. Using a similar method as Theorem \ref{eqn-timediscrete-conservation}, we can prove the energy conservation of the solution to the system \eqref{lem:timediscrete_truncate_existence_pf3}. Then, using the energy conservation, we can further establish the boundedness of the \(H^1\)-norm for \(\{Z_N\}_{N\in\mathbb{N}}\), i.e., \(\|Z_N\|_{H^1} \leq K\). Consequently, there exists a subsequence of \(\{Z_N\}_{N\in\mathbb{N}}\) (which, for simplicity, is again denoted by itself) and a function \(Z \in H_0^1(U)\) such that  
	\begin{align}\label{lem:timediscrete_truncate_existence_pf12}  
		\lim_{N \to \infty} Z_N = Z \quad \text{strongly in } L^2(U), \qquad \lim_{N \to \infty} Z_N = Z \quad \text{weakly in } H_0^1(U).  
	\end{align}
Therefore, we can obtain that $X_N$ strongly converges to $X:=\left(Z+\Psi_\tau^{T, n-1}\right)/2$ in $L^2(U)$ and weakly converges to $X$ in $H_0^1(U)$, since $Z_N=2X_N-\Psi_\tau^{T, n-1}$. The weak formulation of \eqref{lem:timediscrete_truncate_existence_pf3} can be written as 
\begin{align}  \label{lem:timediscrete_truncate_existence_pf13}
	\left(F(X_N), \chi\right) = L(X_N, \chi)+N(X_N, \chi)=0,  
\end{align}
where 
\begin{align*}
	& L(X_N, \chi):= \left(X_N - \Psi_\tau^{T, n}+\left(\frac{\tau^2}{2\epsilon^4}+\frac{\tau^2}{2\epsilon^2}V\right)
	-\textup{i}\Omega\tau L_zX_N
	+\textup{i}\Omega\tau L_z\Psi_\tau^{T, n-1}, \chi\right)
	+\frac{\tau^2}{2\epsilon^2}\left(\nabla X_N, \nabla \chi\right)
	+\frac{\tau^2\Omega^2}{2}\left(L_zX_N, L_z\chi\right),\\
	& N(X_N, \chi):= \frac{\lambda\tau^2}{4\epsilon^2}\left(\left[\mu_A\left(\left|2 X_N-{\Psi}_\tau^{T, n-1}\right|^2\right)
	+\mu_A\left(\left|\Psi^{T, n-1}_\tau\right|^2\right)\right]X_N, \chi\right)=:\left(\mathcal{N}(X_N), \chi\right).
\end{align*}
Using \eqref{lem:timediscrete_truncate_existence_pf12}, we have 
\begin{align}\label{lem:timediscrete_truncate_existence_pf14}
	\lim_{N \to \infty}L(X_N, \chi) = L(X, \chi). 
\end{align}
Furthermore, we can use Lemma \ref{lem:Vit} to prove the convergence of the nonlinear term \( N(X_N, \chi) \). Indeed, the strong convergence in \eqref{lem:timediscrete_truncate_existence_pf12} implies that \( X_N \) converges to \( X \) in measure. The continuity of \( \mathcal{N} \) then ensures that the sequence \( \{\mathcal{N}(X_N)\}_{N\in\mathbb{N}} \) converges locally in measure to \( \mathcal{N}(X) \).  
Moreover, we can show that the sequence \( \{\mathcal{N}(X_N)\}_{N\in\mathbb{N}} \) is 2-equi-integrable, owing to the boundedness of the operator \( \mathcal{N}(\cdot) \) and the strong convergence in \eqref{lem:timediscrete_truncate_existence_pf12}. From Lemma \ref{lem:Vit}, we can obtain 
\begin{align}\label{lem:timediscrete_truncate_existence_pf15}
	\lim_{N \to \infty}N(X_N, \chi) = N(X, \chi). 
\end{align}
By using \eqref{lem:timediscrete_truncate_existence_pf14} and \eqref{lem:timediscrete_truncate_existence_pf15}, we conclude 
\begin{align}\label{lem:timediscrete_truncate_existence_pf16}
	\lim_{N \to \infty}\left(F(X_N), \chi\right) = \left(F(X), \chi\right). 
\end{align}
Therefore, the existence of the solution $X_N$ in \textbf{Step 1} implies that the existence of weak solution to \eqref{eqn:timediscrete_truncate} in $H_0^1(U)$. 
\textbf{Step 3}. 
This step aims to establish the existence of a strong solution to \eqref{eqn:timediscrete_truncate} in \( H_0^1(U) \cap H^2(U) \). This result follows from classical elliptic regularity theory \cite{evans2022partial}. Specifically, to show that \( \Delta X_N \) strongly converges to \( \Delta X \) in \( L^2(U) \), it is also necessary to demonstrate its strong convergence in \( H^1(U) \). To achieve this, we first derive an error estimate for the system \eqref{lem:timediscrete_truncate_existence_pf3}, following a similar approach to the error analysis of the time-discrete method presented in the next subsection.


Combining the results from \textbf{Steps 1-3}, we ultimately establish the existence of a solution to the truncated time-discrete method \eqref{eqn:timediscrete_truncate}.
\end{document}